\documentclass[11pt]{article}
\title{Soft edge  limit of the Laguerre beta-ensemble at the lower edge}
\date{}
\author{Yun Li, Benedek Valk\'o, Jiaming Xu}

\usepackage[utf8]{inputenc}

\usepackage{units}
\usepackage{amsmath, amsthm, amssymb,stackengine,bbm}
\usepackage{graphicx}
\usepackage{amsmath, enumerate}
\usepackage{color}
\usepackage{comment}

\usepackage{hyperref}
    \newtheorem{theorem}{Theorem}
        \newtheorem{thmx}{Theorem}
     
    \newtheorem{lemma}[theorem]{Lemma}
    \newtheorem{proposition}[theorem]{Proposition}
    \newtheorem{corollary}[theorem]{Corollary}
    
    \newtheorem{fact}[theorem]{Fact}

\theoremstyle{definition} 

    \newtheorem{remark}[theorem]{Remark}

\newcommand{\eps}{\varepsilon}

\newcommand{\ZZ}{{\mathbb Z}}
\newcommand{\FF}{{\mathcal
F}}

\newcommand{\R}{{\mathbb R}}

\newcommand{\ev}{{\rm   E}}
\newcommand{\pr}{\mbox{\rm P}}
\newcommand{\lstar}{{\raise-0.15ex\hbox{$\scriptstyle \ast$}}}

\theoremstyle{remark} 

\newcommand{\Airyb}{\operatorname{Airy}_{\beta}}
\newcommand{\Bessel}{\operatorname{Bessel}}

\newcommand{\Ai}{\mathtt{Ai}}
\newcommand{\Airyop}{\mathtt{Airy}_{\beta}}

\newcommand{\Lag}{\mathrm{Laguerre}_{n,\beta,a}}

\definecolor{violet}{rgb}{0.8,0,0.2}

\newcommand{\ed}{\stackrel{d}{=}}

\newcommand{\mat}[4]{\left( \begin{array}{cc}
#1 & #2  \\
#3 & #4  \\
\end{array} \right)}

\newcommand{\ind}{\mathbf 1}

\newcommand{\Var}{{\mathrm{Var}}}
\newcommand{\Cov}{{\mathrm{Cov}}}

\newcommand{\tl}{\tilde}
\newcommand{\wtl}{\widetilde}

\newcommand{\spec}{\operatorname{spec}}
\newcommand{\ddd}{\mathbf{d}}

\newcommand{\sech}{{\operatorname{sech}}}

\newcommand{\sfc}{\mathsf{c}}
\newcommand{\sft}{\mathsf{t}}
\newcommand{\cdr}{\mathsf{c}}

\newcommand{\cM}{\mathcal{M}}
\newcommand{\cE}{\mathcal{E}}

\newcommand{\ma}[1]{\mathbf{#1}}

\definecolor{darkgreen}{rgb}{.2,0.4,0.3}

\begin{document}

\maketitle

\begin{abstract}
    We show that the lower edge of the appropriately scaled size $n$ Laguerre beta-ensemble with parameter $a=a_n$ converges to the $\Airyb$ process as $n\to \infty$ when $a_n\to \infty$ and $\tfrac{a_n}{n}\to 0$. This completes the picture of the possible edge scaling limits of the Laguerre beta-ensemble with a fixed $\beta>0$. When $a_n\gg (\log \log n)^3$ our proof establishes operator level convergence of the inverse of the scaled Dumitriu-Edelman tridiagonal matrix to the inverse of the stochastic Airy operator. Our methods allow us to prove similar operator level limits for the known soft edge scaling limits of the Laguerre and Gaussian beta-ensembles.
    For $a_n\le (\log n)^{1/2}$ we give a different argument that relies on coupling  and a result of \cite{DLV} for the transition between the hard and soft edge limits of the Laguerre beta-ensemble.
\end{abstract}

\section{Introduction}\label{sec:intro}

The size $n$ Laguerre beta-ensemble with $\beta>0$ and parameter $a>-1$ is the distribution of a random vector in $\R_+^n$ with joint probability density function
\begin{align}\label{eq:Lag_PDF}
\frac{1}{Z_{n,\beta, a} }\prod_{1\le j<k\le n} |\lambda_j-\lambda_k|^{\beta} \prod_{k=1}^n \lambda_k^{\frac{\beta}{2}(a+1)-1} e^{- \frac{\beta}{2} \lambda_k},
\end{align}
where $Z_{n,\beta,a}$ is an explicitly computable normalizing constant. We denote by $\Lambda \sim \Lag$ if $\Lambda\in \R_+^n$ has this distribution. The density \eqref{eq:Lag_PDF} generalizes the eigenvalue density of the classical null Wishart matrices \cite{Wishart}. More precisely, for $\beta=1,2,4$ and $a\in\mathbb{Z}_{\ge 0}$, this density arises from the eigenvalues of sample covariance matrices of the form $\ma{X}\ma{X}^*$, where $\ma{X}$ is an $n\times (n+a)$ matrix with i.i.d.~real, complex, or quaternion Gaussian entries, respectively.

We  study the asymptotic behavior of the Laguerre beta-ensemble as $n\to \infty$, $\beta>0$ is fixed, and $a=a_n$ could possibly depend on $n$.

The asymptotic behavior of the empirical distribution of the Laguerre beta-ensemble is  described by the  Marchenko-Pastur limit law. Assume that  $a_n>-1$ is a sequence such that $\lim_{n\to \infty} \frac{n+a_n}{n}=\gamma\in [1,\infty)$ exists. Let  $\Lambda_{n}\sim \Lag$, and consider the scaled empirical spectral measure $\nu_n:=\frac{1}{n} \sum_{k=1}^n \delta_{\lambda_{k,n}/n}$. The Marchenko-Pastur theorem (\cite{MP}, \cite{ForBook}) states that the sequence of random probability measures $\nu_n, n\ge 1$ almost surely converge in distribution to a deterministic measure. The limit is the Marchenko-Pastur distribution given by the density
\begin{align}\label{eq:MP_dist}
\frac{1}{2\pi x} \sqrt{(x-b_{-})(b_+-x)}\ind_{[b_{-},b_{+}]}(x), \qquad b_{\pm}=b_{\pm}(\gamma)=(\sqrt{\gamma}\pm 1)^2.
\end{align}

 To understand the microscopic behavior of the Laguerre beta-ensemble one has to study its point process scaling limits. We will focus on the edge scaling limits in this paper, for results on the bulk limits see \cite{JV}. The Marchenko-Pastur theorem suggests that the lower and upper edge of $\Lag$ are near $(\sqrt{n+a}\pm\sqrt{n})^2$. 
 The upper edge limit was derived by Ram\'{\i}rez, Rider, and Vir\'ag in \cite{RRV}. The point process limit is  the stochastic Airy process $\Airyb$, which has the distribution as the a.s.~discrete spectrum of   the stochastic Airy operator $\Airyop$ defined as 
\begin{align}\label{Airyop}
\Airyop=-\frac{d^2}{dx^2}+x+\frac{2}{\sqrt{\beta}} B'_x
\end{align}
on $[0,\infty)$ with Dirichlet boundary condition at 0. Here $B'$ is standard white noise. (See Section \ref{sec:stoch_op} for additional details.)
 
 \begin{thmx}[Upper soft edge limit, \cite{RRV}]\label{thm:uppersoft}
Let $\Lambda_{n}\sim \mathrm{Laguerre}_{n,\beta,a_n}$, then as $n\to \infty$ we have
 \begin{align*}
     \frac{((n+a_n) n)^{1/6}}{(\sqrt{n+a_n}+\sqrt{n})^{4/3}}((\sqrt{n+a_n}+\sqrt{n})^2-\Lambda_n)\Rightarrow \Airyb. 
 \end{align*}
 \end{thmx}
Although \cite{RRV} did not state it explicitly, their methods extend to the lower edge if one assumes $\liminf_{n\to \infty} \frac{a_n}{n}>0$. 
 \begin{thmx}[Lower soft edge limit, $\liminf a_n/n>0$, \cite{RRV}]\label{thm:lowersoft}
Assume $\liminf\limits_{n\to \infty} \frac{a_n}{n}>0$, and let $\Lambda_{n}\sim \mathrm{Laguerre}_{n,\beta,a_n}$. Then as $n\to \infty$, we have
 \begin{align*}
     \frac{((n+a_n) n)^{1/6}}{(\sqrt{n+a_n}-\sqrt{n})^{4/3}}(\Lambda_n-(\sqrt{n+a_n}-\sqrt{n})^2)\Rightarrow \Airyb. 
 \end{align*}
 \end{thmx}
When $a_n=a>-1$ is fixed, the expected lower edge of $\Lag$ is near 0, and one obtains a different scaling limit that depends on $a$. 
This is the hard edge scaling limit of the Laguerre beta-ensemble, proved by Ram\'{\i}rez and Rider in \cite{RR}. The limit process is the stochastic Bessel process $\Bessel_{\beta,a}$, which can be realized as the spectrum of a random Sturm-Liouville differential operator $\mathfrak{G}_{\beta,a}$ built from Brownian motion, see  Section \ref{sec:stoch_op} for additional details. 

\begin{thmx}[Hard edge limit, \cite{RR}, \cite{RR_err}]\label{thm:hard}
    Let $\Lambda_{n}\sim \Lag$ with $a>-1$ fixed. Then as $n\to \infty$, we have
\begin{align*}
    n \Lambda_n\Rightarrow \Bessel_{\beta, a}.
\end{align*}
\end{thmx}
Theorems \ref{thm:lowersoft} and \ref{thm:hard} do not cover the lower edge scaling limit in the case when $a_n\to \infty$ and $\lim_{n\to \infty} \frac{a_n}{n}=0$. It has been conjectured that in Theorem \ref{thm:lowersoft} the condition $\liminf\limits_{n\to \infty} \frac{a_n}{n}>0$ can be removed (see e.g.~Section 5 of \cite{LedouxRider}). Our main result confirms this conjecture. 
\begin{theorem}[Lower edge, $\lim a_n/n=0$]\label{thm:main_lower}
Let $\Lambda_{n}\sim \mathrm{Laguerre}_{n,\beta,2a_n}$ with   $a_n\to\infty$ and $a_n/n\to 0$ as $n\to\infty$. Then we have
 \begin{align}
    a_n^{-4/3} n(\Lambda_n-(\sqrt{n+2 a_n}-\sqrt{n})^2)\Rightarrow \Airyb\qquad \text{as $n\to \infty$}. \label{HTS_lim}
 \end{align}    
\end{theorem}
Note that we state our result for $\mathrm{Laguerre}_{n,\beta,2a_n}$ instead of $\mathrm{Laguerre}_{n,\beta,a_n}$ to get simpler scaling constants. One can check that if $\lim_{n\to \infty} \frac{a_n}{n}=0$ then $ \frac{((n+a_n) n)^{1/6}}{(\sqrt{n+a_n}-\sqrt{n})^{4/3}}\sim  (a_n/2)^{-4/3}\,  n$ as $n\to \infty$, hence our scaling is equivalent to the one in Theorem \ref{thm:lowersoft}.

 We remark that the special case $\beta=2$, $a_n=c \sqrt{n}$ of Theorem \ref{thm:main_lower} was treated in  \cite{DMT}, exploiting the determinantal structure present in the $\beta=2$ case. To the best of our knowledge, prior to the present work, \eqref{HTS_lim} was not known for other diverging rates even when $\beta=2$, and no analogous result was known for $\beta=1$ or $4$. 

Theorems \ref{thm:hard} and \ref{thm:main_lower} together with a  diagonal argument  imply the following point process level transition from the $\Bessel_{\beta,a}$ process to $\Airyb$.
\begin{thmx}[Hard-to-soft edge transition, \cite{DLV}]\label{thm:hardtosoft}
    \begin{align}\label{eq:HtoS_process}
a^{-4/3}(\Bessel_{\beta,2a}-a^2)\Rightarrow \Airyb, \qquad \text{as $a\to \infty$}.
\end{align}
\end{thmx}
This statement has been proved directly (without relying on the finite ensembles) in \cite{DLV}. In fact, our results do not provide a new proof of this statement, since we actually use Theorem \ref{thm:hardtosoft} as an ingredient for proving Theorem \ref{thm:main_lower} in the case when $a_n$ grows very slowly.

Note that the convergence \eqref{eq:HtoS_process} can be analyzed directly at the level of the limiting processes. For the classical values $\beta=1,2,4$, this convergence was first proved in \cite{BF} using the algebraic structure of the models. For general $\beta>0$, Ram\'irez and Rider proved the convergence of the smallest point of the corresponding processes in \cite{RR}; they later extended their work to a full process-level limit in \cite{RR_spike}. In \cite{DLV}, the authors derived the operator-level convergence of \eqref{eq:HtoS_process} by showing that, with an appropriate coupling, the rescaled hard edge operators $\mathfrak{G}_{\beta,a}$ converge to $\Airyop$ a.s.~in the norm resolvent sense.

\subsection{Outline of the proof}

Our proof (just as the proofs of Theorems \ref{thm:uppersoft}, \ref{thm:lowersoft}, \ref{thm:hard}) relies on the tridiagonal matrix representation of the Laguerre beta-ensemble due to Dumitriu and Edelman \cite{DE}.

Fix $n\in \ZZ_{+}$, $\beta>0$ and $a>-1$, and consider the random bidiagonal matrix 
\begin{align}\label{eq:DE-bidiagonal}
    \ma{L}_n=\ma{L}_{n,\beta,a}=\begin{pmatrix}
X_1 & -Y_1 &    \\
& X_2 & -Y_2  \\
 & & \ddots& \ddots\\
& && X_{n-1} & -Y_{n-1}\\
 & &  && X_n
\end{pmatrix}
\end{align}
where $\{\sqrt{\beta}X_k\sim \chi_{\beta(n+a+1-k)},1\le k\le n\}$ and $\{\sqrt{\beta} Y_k\sim \chi_{\beta(n-k)},1\le k\le n-1\}$ are independent chi-random variables. We write $Z\sim \chi_p$ if the probability density function of $Z$ is given by $\frac{1}{2^{p/2-1}\Gamma(p/2)}x^{p-1}e^{-x^2/2}\ind_{x> 0}$.  
One of the main results of \cite{DE} is that the eigenvalues of the random tridiagonal matrix $\ma{L}_n\ma{L}_n^\top$ are distributed according to \eqref{eq:Lag_PDF}.

Edelman and Sutton \cite{ES} outlined a program to study the point process scaling limits of beta-ensembles by studying the scaling limits of the corresponding tridiagonal matrices. Treating  the tridiagonal matrices as discrete approximations of  differential operators, and showing that they converge to a limiting operator in a strong enough sense, one can prove the convergence of the finite ensembles to the spectrum of the limiting operator. This is the strategy that was carried out for the soft edge limit in \cite{RRV} and the hard edge limit in \cite{RR}. Similar results for bulk  and other limits appeared in \cite{BVBV,KS,JV,HM2012,BVBV_op,BVBV_19,LV}. 
Let us also mention that these point process limits have been proved to be universal for a wide class of beta-ensembles, see \cite{BEY,BEYedge,KRV,RW}. Moreover, the results of \cite{KRV,RW} also show universality on the level of random operators near the edge.

We have to use  different approaches to prove Theorem \ref{thm:main_lower} depending on how fast $a_n$ is growing. When $a_n$ grows sufficiently fast, we study the inverse of the rescaled and recentered tridiagonal matrix model for the Laguerre beta-ensemble. In the regime when $a_n$ grows slowly, our proof relies on an explicit coupling between the Laguerre beta-ensemble and the stochastic Bessel operator. We outline the two approaches in more detail below, with some comments on where the conditions on $a_n$ are used.

\subsubsection*{Outline of the proof when $a_n\gg (\log\log n)^3$} 

We introduce the notation
\begin{align}\label{def:mun}
    \mu_n=(\sqrt{n+2a_n}-\sqrt{n})^2,\qquad m_n=a_n^{-2/3}n.
\end{align}
Let $\ma{L}_n=\ma{L}_{n,\beta,2a_n}$ be the bidiagonal matrix model defined as in \eqref{eq:DE-bidiagonal}. We consider the re-centered matrix $\ma{M}_n=\ma{L}_n \ma{L}_n^{\top}-\mu_n\ma{I}_n$ and embed $(a_n^{-4/3}n\ma{M}_n)^{-1}$ with a mesh of $m_n^{-1}=a_n^{2/3} n^{-1}$ as an integral operator on $\R_+$:
\begin{align}\label{def:Kn_kernel}
 \mathsf{K}_n(x,y)= a_n^{2/3}[\ma{M}^{-1}_n]_{\lceil x m_n\rceil,\lceil y m_n\rceil} \ind(0<x, y\le a_n^{2/3} ).  
\end{align}
The coefficient $a_n^{2/3}$ comes from the inverse of the scaling of $\ma{M}_n$ multiplied by $m_n$. The integral operator $\mathsf{K}_n$ is Hilbert-Schmidt, and its non-zero eigenvalues agree with the eigenvalues of $(a_n^{-4/3}n\ma{M}_n)^{-1}$.

The inverse of the stochastic Airy operator $\Airyop$ is an almost surely compact Hilbert-Schmidt integral operator on $\R_+$ with a random kernel $\mathsf{K}_{\Ai}$. 
We will show that there is a coupling of $\ma{L}_n, n\ge 1$ and $\Airyop$ so that $\|\mathsf{K}_n-\mathsf{K}_{\Ai}\|_{2}\to 0$ in probability. (See Theorem \ref{thm:main_HS} below for the precise statement.) This implies the convergence in distribution stated in Theorem \ref{thm:main_lower}. 

The proof of Theorem \ref{thm:main_HS} relies on the careful analysis of the partial solution of the eigenvalue equation $\ma{M}_n \ma{u}=0, \ma{u}=[u_1,u_2,\dots,u_n]^\top$. We divide the analysis of $u_i,1\le i\le n$ into three regimes according to $i$. In the first regime, we apply the classical theory of convergence of Markov chains to diffusions. While the argument is standard, the moment computations are subtle. In the second regime, we develop martingale fluctuation estimates to prove almost optimal results on the asymptotic behavior of the discrete Riccati transform associated with $u_i$.  This is the most technically challenging part of the proof, and this is where we need the assumption on the growth of $a_n$. We believe  this  is not just a technical issue, see Remark \ref{rem:no_go} in Section \ref{subsec:martingale}.
In the third regime, we use concentration bounds for the entries of $\ma{M}_n$ to prove that $u_i$ grows sufficiently fast. Using a discrete Wronskian identity we can turn the obtained information on $\ma{u}$ into information on $\ma{M_n}^{-1}$, which leads to the proof of the claimed norm-resolvent convergence $\|\mathsf{K}_n-\mathsf{K}_{\Ai}\|_2\to 0$ when $a_n\gg (\log\log n)^3$.

\subsubsection*{Outline of the proof when $1\ll a_n\le  (\log n)^{1/2}$} 

For slowly growing  $a_n$ we use a different method to prove our convergence result.
In this case we prove a quantitative version of Theorem \ref{thm:hard}  to show that the  $k$ smallest  eigenvalues of $n \ma{L}_{n,\beta,2a_n} \ma{L}_{n,\beta,2a_n}^{\top}$ are close to the  $k$ smallest points in the $\Bessel_{\beta, 2a_n}$ process. More precisely, using the coupling techniques developed in  \cite{BVBV_19} to study the bulk scaling limit of the circular beta-ensemble, we 
couple the integral operator corresponding to  $(n \ma{L}_{n,\beta,2a_n} \ma{L}_{n,\beta,2a_n}^{\top})^{-1}$ to  the inverse of the appropriate hard edge operator. We obtain quantitative estimates on the $\ell^2$-distance between the eigenvalues of $(n \ma{L}_{n,\beta,2a_n} \ma{L}_{n,\beta,2a_n}^{\top})^{-1}$ and the inverse of the $\Bessel_{\beta,2a_n}$ process, see Proposition \ref{prop:HS_kernel_smalla}. We then rely on Theorem \ref{thm:hardtosoft} and a diagonal argument to prove that (after appropriate scaling) the   $k$ smallest  eigenvalues of $n\ma{L}_{n,\beta,2a_n} \ma{L}_{n,\beta,2a_n}^{\top}$ are close to  the  $k$ smallest  points in the $\Airyb$ process.

The estimates in the coupling step of our proof require $a_n$ to not grow too fast. 
The condition $a_n\le (\log n)^{1/2}$ is not optimal, any rate satisfying $a_n\le (\log n)^{c}$ with $c<1$ works here. We also expect that our method applies to the case when $a_n\equiv a>-1$ is fixed, leading to  quantitative bounds on the convergence rate of $n\text{Laguerre}_{n,\beta,a}$ to the limiting process $\Bessel_{\beta,a}$, see Remark \ref{rmk:hardedge_rate}.

\subsection{Operator level soft edge limits}
The methods of our proof of Theorem \ref{thm:main_lower} when $a_n\gg(\log\log n)^3$ also apply to the known soft edge limits of the Gaussian beta-ensemble, the upper edge of the Laguerre beta-ensemble, and the lower edge of the Laguerre beta-ensemble when $\lim_{n\to\infty} a_n/n\to c\in(0,\infty]$.

Let us first consider the Gaussian beta-ensemble. For $\beta>0$ and $n\ge 1$, define
\begin{align}\label{eq:tri_G}
    \ma{G}_{n} = 
   \begin{pmatrix}
    g_1 & -Y_1 &  &  \\
    -Y_1 & g_2 & -Y_2&  \\
     & \ddots& \ddots& \ddots\\
     & &-Y_{n-2}& g_{n-1}& -Y_{n-1}\\
    & &  &-Y_{n-1}& g_n
    \end{pmatrix},
\end{align}
where the entries $\{g_i\sim N(0,\frac{2}{\beta}),1\le i\le n\}$ and $\{\sqrt{\beta}Y_i\sim \chi_{\beta(n-i)},1\le i\le n-1\}$ are independent. Dumitriu and Edelman \cite{DE} show that the joint eigenvalue density  of $\ma{G}_n$ is given by
\begin{align}\label{eq:Gaussian_jpdf}
    p_n^G(\lambda_1,\lambda_2,\dots,\lambda_n) = \frac{1}{Z_{n,\beta}^G} \prod_{1\le j<k\le n}|\lambda_j-\lambda_k|^\beta \prod_{j=1}^n e^{-\frac{\beta}{4}\lambda_j^2},
\end{align}
where $Z_{n,\beta}^G$ is an explicitly computable normalizing constant. This density defines  the Gaussian beta-ensemble, and we write $\Lambda_n=(\lambda_1,\dots,\lambda_n)\sim G\beta E_{n}$ if the points $\lambda_i,1\le i\le n$ are distributed according to \eqref{eq:Gaussian_jpdf}.
When $\beta=1,2,$ or $4$, \eqref{eq:Gaussian_jpdf} recovers the joint eigenvalue density of the classical Gaussian  orthogonal/unitary/symplectic ensembles. For general $\beta>0$, $G\beta E_{n}$ can be viewed as a one-dimensional Coulomb gas in Gaussian potential at inverse temperature $\beta$.

The global limit of the Gaussian beta-ensemble is described by the famous Wigner's semicircle law. Let $\Lambda_n=(\lambda_1,\dots,\lambda_n)\sim G\beta E_n$ and consider the empirical spectral measure $\nu_n^G:=\frac{1}{n}\sum_{i=1}^n\delta_{\lambda_i/\sqrt{n}}$. Then Wigner's semicircle law states that $\nu_n^G$ converges weakly a.s.~to the semicircle law with density $\frac{1}{2\pi}\sqrt{4-x^2}\ind_{[-2,2]}(x)$ as $n\to\infty$. Wigner's semicircle law also suggests that asymptotically the lower and upper edge of the spectrum are near $\pm 2\sqrt{n}$.
We focus on the lower edge of the spectrum $-2\sqrt{n}$, and define $\ma{M}_n^G=(\ma{G}_n+2\sqrt{n})$. We can view $(n^{1/6} \ma{M}_n^G)^{-1}$ as an integral operator on $\R_+$ with kernel
\begin{align*}
    \mathsf{K}_n^G(x,y) = n^{1/6}[(\ma{M}_n^G)^{-1}]_{\lceil x n^{1/3}\rceil, \lceil y n^{1/3}\rceil} \ind(0<x, y\le n^{2/3}).
\end{align*}
The scaling coefficient $n^{1/6}$ in $\mathsf{K}_n^G(x,y)$ comes from the inverse mesh size $n^{1/3}$ and the scaling $n^{1/6}$ in front of $\ma{M}_n^G$. 

\begin{theorem}\label{thm:Gaussian}
    Fix $\beta>0$. There exists a coupling of $\mathsf{K}_n^G$ and $\mathsf{K}_{\Ai}$ such that $\|\mathsf{K}_{n}^G-\mathsf{K}_{\Ai}\|_{2} \to 0$ in probability as $n\to\infty$. As a consequence, let $\Lambda_n\sim G\beta E_{n}$, then $n^{1/6}(\Lambda_n+2\sqrt{n})\Rightarrow\Airyb$.
\end{theorem}
Note that by  symmetry, we have $\spec(\ma{G}_n) \ed \spec(-\ma{G}_n)$. Consequently, the proof of Theorem \ref{thm:Gaussian} also applies to the upper edge of the Gaussian beta-ensemble.\medskip

We next state the operator level version of Theorem  \ref{thm:lowersoft} for the general Laguerre beta-ensemble.
Let $\ma{L}_n\equiv \ma{L}_{n,\beta,a_n}$ be defined as in \eqref{eq:DE-bidiagonal}.
We introduce the dimension parameter $\kappa:=n+a_n$, and set
\begin{align*}
\mu_{n}^L=(\sqrt{\kappa}-\sqrt{n})^2,\quad 
    m_{n}^L=\Big(\frac{\sqrt{n\kappa }}{\sqrt{\kappa}-\sqrt{n}}\Big)^{2/3},\quad \sigma_{n}^L = \frac{(m_{n}^L)^2}{\sqrt{n\kappa}} = \frac{(n\kappa)^{1/6}}{(\sqrt{\kappa}-\sqrt{n})^{4/3}}.
\end{align*}
Setting $\ma{M}_n^L:=(\ma{L}_n\ma{L}_n^\top-\mu_n)$,
one can view $(\sigma_{n}^L\ma{M}_n^L)^{-1}$ with a mesh of $(m_n^{L})^{-1}$ as an integral operator on $\R_+$:
\begin{align}\label{def:Kn_kernel_general}
 \mathsf{K}_n^L(x,y)= (\sigma_{n}^L)^{-1}m_{n}^{L}[(\ma{M}_n^L)^{-1}]_{\lceil x m_n^L\rceil, \lceil y m_n^L\rceil} \ind(0<x, y\le n/m_{n}^L).  
\end{align}
\begin{theorem}\label{thm:Laguerre}
    Assume $\beta>0$ and $\liminf_{n\to\infty} \kappa/n\in(1,\infty]$, then there exists a coupling of $\mathsf{K}_n^L$ and $\mathsf{K}_{\Ai}$ such that $\|\mathsf{K}_{n}^L-\mathsf{K}_{\Ai}\|_{2} \to 0$ in probability as $n\to\infty$. As a consequence, let $\Lambda_n\sim \mathrm{Laguerre}_{n,\beta,a_n}$, then $\sigma_{n}^L(\Lambda_n-\mu_{n}^L)\Rightarrow\Airyb$.
\end{theorem}
Our methods also lead to the operator level version of Theorem \ref{thm:uppersoft},  see Remark \ref{rmk:Lag_upper} below.

\subsection*{Outline of the rest of the paper}

Section \ref{sec:prelim} covers the basics on inverses of bidiagonal and tridiagonal matrices, and gives a short introduction to  stochastic operators.
Section \ref{sec:biga} contains the proof of Theorem \ref{thm:main_lower} in the case when $a_n\gg (\log \log n)^3$, while Section \ref{sec:small_a} gives the proof of Theorem \ref{thm:main_lower} in the case when $1\ll a_n\le (\log n)^{1/2}$. 
Section \ref{sec:other_thms} includes the proofs of Theorems \ref{thm:Gaussian} and \ref{thm:Laguerre}.
Some  technical proofs are postponed to  Section \ref{sec:Appendix} (the Appendix). 

\subsection*{Acknowledgments} 

Y.L.~was partially supported by the China Postdoctoral Science Foundation under Grant No. 2024M751603 and by the Shuimu Tsinghua Scholar Program. B.V.~was partially supported by  the University of Wisconsin – Madison Office of the Vice Chancellor for Research and Graduate Education with funding from the Wisconsin Alumni Research Foundation and by the National Science Foundation award DMS-2246435. 
This material is based upon work supported by the Swedish Research Council under grant no.~2021-06594 while Y.L.~and J.X.~were in residence at Institut Mittag-Leffler in Djursholm, Sweden during the Fall 2024 semester.

\section{Preliminaries}\label{sec:prelim}

\subsection{Inverse of bidiagonal and tridiagonal matrices}

We review some well known results on the inverse of bidiagonal and tridiagonal matrices. The following lemma is classical. 
\begin{lemma}\label{lem:inv_bidiagonal}
    Suppose that $\ma{M}$ is an $n\times n$ upper bidiagonal matrix with non-zero diagonal entries $d_i, 1\le i\le n$ and off-diagonal entries $-e_i, 1\le i\le n-1$. Then the entries of $\ma{M}^{-1}$ are given by
    \begin{align}
        [\ma{M}^{-1}]_{i,j}=\begin{cases}
\frac{\prod_{k=i}^{j-1} e_k}{\prod_{k=i}^{j} d_k}, \qquad  &i\le j,\\
0, \qquad &i>j.
        \end{cases}
    \end{align}
\end{lemma}

The inverse of a symmetric tridiagonal matrix can be expressed using the determinants of the principal minors, or (equivalently) in terms of two independent solutions of the eigenvector equation. We refer to \cite{Meurant,Usmani} for more details. 

Let $\ma{M}$ be an $n\times n$ real symmetric tridiagonal matrix with diagonal entries $d_i, 1\le i\le n$ and negative off-diagonal entries $-e_i, 1\le i\le n-1$:
\begin{align}\label{eq:tridiagonal}
    \ma{M} = 
    \begin{pmatrix}
    d_1 & -e_1 &  &  \\
    -e_1 & d_2 & -e_2 &  \\
    & -e_2 & d_3 & -e_3\\
    & &\ddots & \ddots& \ddots\\
    & & &-e_{n-2} & d_{n-1} & -e_{n-1}\\
    & & &  & -e_{n-1} & d_n
    \end{pmatrix}.
\end{align}
We assume that $\det \ma{M}\neq 0$.
For $1\le i\le j\le n$ we denote by $\ma{M}_{[i,j]}$ the principal submatrix of $\ma{M}$ that is formed by keeping the rows and columns with index $k$ satisfying $i\le k\le j$. Let $u_i, v_i, 1\le i\le n$ be defined by the following recursions: 
\begin{align}\label{eq:udef1}
    &\qquad\qquad u_1=1, \quad u_1 d_1- u_2 e_1=0, \\
    &-u_{k-1} e_{k-1}+u_k d_k-u_{k+1} e_{k}=0, \qquad 2\le k\le n-1,\label{eq:udef2}
\end{align}
and
\begin{align}\label{eq:vdef1}
    v_n=\prod_{k=1}^{n-1} e_k \cdot (\det \ma{M})^{-1}, \qquad -v_{n-1} e_{n-1}+d_n v_n=0, \\
    -v_{k-1} e_{k-1}+v_k d_k-v_{k+1} e_{k}=0, \qquad 2\le k\le n-1.\label{eq:vdef2}
\end{align}
Note that $\ma{u}=[u_1,\dots,u_n]^\top$ and $\ma{v}=[v_1,\dots,v_n]^\top$ as vectors are (partial) solutions of the eigenvector equation $\ma{M}\ma{x}=0$.

\begin{lemma}[\cite{Meurant,Usmani}]\label{lem:tri_inverse}
Let $\ma{M}$ be an $n\times n$ tridiagonal matrix satisfying the assumptions above.
Then the following statements hold:
\begin{align}\label{eq:uv_1}
u_{i}&= \det \ma{M}_{[1,i-1]} \cdot \prod_{k=1}^{i-1} e_k^{-1}, \qquad\qquad 1\le i\le n,\\\label{eq:uv_2}
v_{j}&= \det \ma{M}_{[j+1,n]} \cdot \prod_{k=1}^{j-1} e_k \cdot (\det \ma{M})^{-1}, \qquad 1\le j\le n,\\ \label{eq:tridiag_inv}
[\ma{M}^{-1}]_{i,j}&=u_{i} v_{j}= \det \ma{M}_{[1,i-1]}\det \ma{M}_{[j+1,n]} (\det \ma{M})^{-1} \prod_{k=i}^{j-1} e_k, \qquad \quad 1\le i\le j\le n,\\
\det \ma{M}&=v_n^{-1} \prod_{k=1}^{n-1}e_k= (d_n u_n-e_{n-1} u_{n-1})\prod_{k=1}^{n-1} e_k.\label{eq:det+uv}
\end{align}
Note that by convention  the determinant of the empty matrix is equal to $1$.
\end{lemma}

We separately state the following Wronskian identity between $\ma{u}$ and $\ma{v}$. 
\begin{lemma}[Discrete Wronskian identity]\label{lem:disc_Wr}
Let $\ma{M}$ be an $n\times n$ tridiagonal matrix satisfying the assumptions above.
Let $u_i, v_i$ be defined with \eqref{eq:udef1}-\eqref{eq:vdef2} or \eqref{eq:uv_1}-\eqref{eq:uv_2}, and assume that $u_i\neq 0$ for $1\le i\le n$. Then we have
\begin{align}
\label{eq:discrete_Wr}
&v_k u_{k+1}-u_k v_{k+1}=\frac{1}{e_k}, \qquad \qquad \qquad 1\le k\le n-1,\\
&u_k v_k=\sum_{\ell=k}^{n-1} \frac{u_k^2}{u_\ell u_{\ell+1} e_\ell}+\frac{u_k^2}{u_n^2}\frac{1}{d_n-e_{n-1} \frac{u_{n-1}}{u_n}}. \qquad 1\le k\le n. \label{eq:uv_Wr}
\end{align}
\end{lemma}
\begin{proof}
    For $k=1$ the identity \eqref{eq:discrete_Wr} follows from \eqref{eq:uv_1},\eqref{eq:uv_2}, and the expansion of $\det \ma{M}$ with respect to its first row. The identity can be extended for $k\ge 2$  by induction using \eqref{eq:udef1}-\eqref{eq:vdef2}. From \eqref{eq:discrete_Wr} we have
    \begin{align}\label{eq:u/v}
        \frac{v_k}{u_k}-\frac{v_{k+1}}{u_{k+1}}=\frac{1}{e_k u_k u_{k+1}}, \qquad 1\le k\le n-1.
    \end{align}
Using \eqref{eq:det+uv} we get
\begin{align}\label{eq:u/v_n}
   v_n=\frac{1}{d_n u_n-e_{n-1} u_{n-1}}, \qquad \text{and}\qquad \frac{v_n}{u_n}=\frac{1}{(d_n u_n-e_{n-1} u_{n-1})u_n}.
\end{align}
By summing \eqref{eq:u/v} for $k\le \ell\le n-1$ and using \eqref{eq:u/v_n} we obtain \eqref{eq:uv_Wr}.
\end{proof}

\subsection{Basics of stochastic operators}
\label{sec:stoch_op} 
This section briefly reviews the definition and basic properties of the stochastic Airy and Bessel operators, and their inverses. See \cite{RRV,RR,AlexPhD,DLV} for further properties of these operators.

The stochastic Airy operator $\Airyop$ defined in \eqref{Airyop} is self-adjoint on the following subspace of $L^2(\R_+)$:
\begin{align*}
    L^*=\{f\in L^2(\R_+):f(0)=0,\|f\|_*<\infty\},\qquad \|f\|_*^2:=\int_{\R^+}(f')^2+(1+x)f^2(x)dx.
\end{align*}
We say that $(\lambda,f)\in\R\times L^*$ is an eigenvalue-eigenfunction pair of $\Airyop$ if 
\begin{align*}
    f''(x)= (x-\lambda+\tfrac{2}{\sqrt{\beta}}B'_x) f(x),
\end{align*}
where both sides are understood as distributions. By It\^o's formula and integration by parts, if $\psi$ solves the equation $\Airyop \psi=0$ with  non-zero deterministic initial conditions $(\psi(0),\psi'(0))=(c_0,c_1)$, then $(\psi, \psi')$ is the strong solution of the stochastic differential equation (SDE) system
\begin{align}\label{AirySDE_1}
d\psi(x)=\psi'(x) dx, \qquad d\psi'(x)=\psi(x)\left(\tfrac{2}{\sqrt{\beta}} dB+x dx   \right).
\end{align}
The strong solution $\psi$ of \eqref{AirySDE_1} satisfies
\begin{align}\label{Airysqr}
\frac{\psi'(x)}{\psi(x) \sqrt{x}}\to 1 \quad\text{a.s.~as }x\to \infty,
\end{align}
see Lemma \ref{lem:Airy_asymp} below for a more precise bound.

The operator $\Airyop$ can also be viewed as a generalized Sturm-Liouville operator of the form
\begin{align}\label{SL}
    \tau f(x)= \frac{1}{r(x)}\left(-(p_1(x)f'(x)-q_0(x) f(x))'-q_0(x)f'(x)+p_0(x) f(x)\right),
\end{align}
with the choices $r(x)=p_1(x)=1,q_0(x)=\frac{2}{\sqrt{\beta}}B_x,p_0(x)=x$, see \cite{AlexPhD,Minami}. Using the asymptotics \eqref{Airysqr}, and classical theory on generalized Sturm-Liouville operators (see e.g.~\cite{Teschl}), lead to the following results. 

\begin{proposition}[Proposition 2.6 of \cite{DLV}]     Let $\psi_d$ be the solution of the equation $\Airyop \psi_d=0$ with Dirichlet initial condition $\psi_d(0)=0,\psi_d'(0)=1$. Then a.s.~$\psi_d\notin L^2(\R_+)$, and $0$ is not an eigenvalue of $\Airyop$. 

There exists a unique $L^2(\R_+)$ solution $\psi_\infty$ of the equation $\Airyop \psi_\infty=0$ with initial condition $\psi_\infty(0)=1$. The functions $\psi_d,\psi_\infty$ satisfy the Wronskian identity
 \begin{align}\label{Wr_0}
\psi_\infty'(x) \psi_{d}(x)-\psi_\infty(x) \psi'_d(x)= -1.
\end{align}
The inverse operator $\Airyop^{-1}$ is an a.s.~Hilbert-Schmidt integral operator on $L^2(\R_+)$ with kernel
\begin{align}\label{AiryopHS}
\mathsf{K}_{\Ai}(x,y)=\psi_\infty(x) \psi_{d}(y) \ind(x\ge y)+\psi_{d}(x) \psi_\infty(y) \ind(x<y).
\end{align}
\end{proposition}

Note that \eqref{Wr_0} is essentially a continuous version of Lemma \ref{lem:disc_Wr}. We can use it to construct $\psi_\infty$ from $\psi_d$. If $\psi_d$ does not have any zeros in $[x_0,\infty)$, then for $x\ge x_0$ we have
\begin{align}\label{eq:psi_infty}
    \psi_\infty(x) = \psi_d(x)\int_x^\infty \psi_d^{-2}(y)dy.
\end{align}
(Because of \eqref{Airysqr}, such an $x_0$ will exist.) Then using the Wronskian identity \eqref{Wr_0} again, one can extend uniquely the above function to $[0,x_0)$ as well.

\medskip

Next we discuss the stochastic Bessel operator $\mathfrak{G}_{\beta,a}$ and its inverse. For fixed $\beta>0,a>-1$, introduce 
\begin{align*}
m_{a}(x)&=e^{-(a+1)x-\frac{2}{\sqrt{\beta}} B(x)}\,, \qquad s_a(x)=e^{ ax +\frac{2}{\sqrt{\beta}} B(x)},
\end{align*}
where $B$ is  standard Brownian motion.
Then $\mathfrak{G}_{\beta,a}$ is a second order random differential operator
\begin{align}\label{def:G}
\mathfrak{G}_{\beta,a}=-\frac{1}{m_a(x)} \frac{d}{dx}\left(\frac{1}{s_a(x)} \frac{d}{dx}\, \cdot\,  \right),
\end{align}
acting on a subset of $L^2(\R_+,m_a)$ with Dirichlet boundary condition at $0$ and Neumann boundary condition at infinity. The operator $\mathfrak{G}_{\beta,a}$ also falls into the framework of Sturm-Liouville operators \eqref{SL} with $r=m_{a}$, $p_1=s_a^{-1}$, $p_0=q_0=0$. Ram\'irez-Rider \cite{RR} shows that $\mathfrak{G}_{\beta,a}^{-1}$ is a Hilbert-Schmidt integral operator
\[
(\mathfrak{G}_{\beta,a}^{-1}f)(x) = \int_0^\infty \int_0^{\min\{x,y\}} s_a(dz)\,f(y)\,m_a(dy)
\]
acting on $L^2(\R_+,m_a)$. 

Let $\ma{L}_{n,\beta,a}$ be defined as in \eqref{eq:DE-bidiagonal}.
The key idea in the proof of Theorem \ref{thm:hard} is to show that the scaled inverse $(\sqrt{n}\ma{L}_{n,\beta,a})^{-1}$, viewed as an integral operator on $L^2[0,1]$, converges to the operator $\mathtt{K}_{\beta,a}$ on $L^2[0,1]$ with kernel
\begin{align}\label{eq:kernel_k}
\mathsf{k}_{\beta,a}(x,y) :=(1-x)^{-(1+a)/2}\exp\left(\int_x^y \frac{dB_z}{\sqrt{\beta(1-z)}}\right)(1-y)^{a/2} \ind_{x\le y}.
\end{align}
The stochastic Bessel operator $\mathfrak{G}_{\beta,a}$ is related to  $\mathtt{K}_{\beta,a}$ through the fact that  $\mathfrak{G}_{\beta,a}^{-1}$ can be obtained from $\mathtt{K}_{\beta,a}^{\top}\mathtt{K}_{\beta,a}$ under the change of variable $(x,y)\mapsto(1-e^{-x},1-e^{-y})$. See \cite{RR} for more details.

\section{Proof of Theorem \ref{thm:main_lower} in the case when $a_n\gg (\log \log n)^3$}\label{sec:biga}
In this section we prove Theorem \ref{thm:main_lower} under the assumption that the sequence $a_n, n\ge 1$ satisfies $(\log \log n)^3\ll a_n\ll n$. For most of this section we will drop the dependence on $n$ in $a=a_n$.

For a fixed $T>0$ we set
\begin{align}
    n_0 = n_0(T) = \lfloor Tna^{-2/3}\rfloor,\qquad n_1=n-\lfloor a\mathfrak{f}(a)\rfloor,\qquad \mathfrak{f}(a) = \log(\min\{a,n/a\}).
\end{align}
 Note that the diverging function $\mathfrak{f}(\cdot)$ is chosen such that $\mathfrak{f}(a)\ll a^\eps$ and $a\mathfrak{f}(a)\ll n$ as $n\to\infty$.

\subsection{Outline of the proof}
Let $\ma{L}_n=\ma{L}_{n,\beta,2a}, n\ge 1$ be the sequence of bidiagonal matrices defined in \eqref{eq:DE-bidiagonal} with independent diagonal and off-diagonal entries
\begin{equation}\label{eq:XYdef}
    X_k\sim \beta^{-1/2}\chi_{\beta(n-k+1+2a)},1\le k\le n,\qquad Y_k\sim \beta^{-1/2}\chi_{\beta(n-k)},1\le k\le n-1.
\end{equation}
(We do not denote the $n$-dependence of these random variables.)

Recall that $\mu_n = (\sqrt{n+2a}-\sqrt{n})^2$ and $\ma{M}_n=(\ma{L}_n\ma{L}_n^\top-\mu_n\ma{I}_n)$. We will prove that the rescaled matrix $a^{-4/3}n\ma{M}_n$ converges to the $\Airyop$ operator in norm resolvent sense.
More precisely, we view $(a^{-4/3}n\ma{M}_n)^{-1}$ as an integral operator on $L^2(\R_+)$ with kernel $\mathsf{K}_n$ defined in \eqref{def:Kn_kernel}. The main result of the section is to show that $\mathsf{K}_n$ converges to the kernel of the integral operator  $\Airyop^{-1}$ in  $L^2(\R_+)$.
\begin{theorem}\label{thm:main_HS}
    Fix $\beta>0$ and assume $(\log\log n)^3\ll a_n\ll n$ as $n\to\infty$. There exists a coupling of the kernels $\mathsf{K}_n$ and $\mathsf{K}_{\Ai}$ such that  
    \begin{align}\label{eq:kernel_conv}
    \int_0^\infty\int_0^\infty |\mathsf{K}_n(x,y)-\mathsf{K}_{\Ai}(x,y)|^2 dxdy \to 0 \text {\quad in probability as $n\to\infty$}.
\end{align}
\end{theorem}

Similar to the proof of Theorem \ref{thm:hardtosoft} in \cite{DLV}, the first step is to compare the corresponding kernels restricted to a fixed (large) box $[0,T]^2$. Set
\begin{align}\label{def:K_Ai_T}
    \mathsf{K}_{\Ai}^{(T)}(x,y) = \psi_T(x)\psi_d(y)\ind(y<x\le T) + \psi_d(x)\psi_T(y)\ind(x<y\le T),
\end{align}
where $\psi_T$ solves $\Airyop\psi_T=0$ with boundary conditions $\psi_T(0)=1,\psi_T(T)=0$. For the finite matrix model, we denote by $\ma{M}_n^{(T)}$ the upper left $n_0\times n_0$ corner of the matrix $\ma{M}_n$.
Then the inverse matrix $(a^{-4/3}n\ma{M}_n^{(T)})^{-1}$ can be viewed as an integral operator supported on $[0,T]^2$ with kernel
\begin{align}\label{def:Kn_T}
\mathsf{K}_n^{(T)}(x,y) = a^{2/3}\Big[\big(\ma{M}_n^{(T)}\big)^{-1}\Big]_{\lceil x m_n\rceil, \lceil y m_n\rceil} \ind(0<x, y\le T),\qquad m_n=a^{-2/3}n.
\end{align}
By the triangle inequality, we have
\begin{align}\label{eq:large_a_triangle}
    \|\mathsf{K}_n-\mathsf{K}_{\Ai}\|_{2} \le   \|\mathsf{K}_n^{(T)}-\mathsf{K}_{\Ai}^{(T)}\|_{2} + \|\mathsf{K}_{\Ai}-\mathsf{K}_{\Ai}^{(T)}\|_{2} +   \|\mathsf{K}_n-\mathsf{K}_{n}^{(T)}\|_{2}.
\end{align}
The statement of \eqref{eq:kernel_conv} is the consequence of the following  lemmas.
\begin{lemma}[Lemma 3.1 of \cite{DLV}]\label{lem:Airy_tail}
$\|\mathsf{K}_{\Ai}-\mathsf{K}_{\Ai}^{(T)}\|_{2} \to 0$ a.s.~as $T\to\infty$.
\end{lemma}

\begin{lemma}\label{lem:HS_T}
    For any fixed $T>0$, there exists a coupling of $\mathsf{K}_n^{(T)},\mathsf{K}_{\Ai}^{(T)}$ such that  $\|\mathsf{K}_n^{(T)}-\mathsf{K}_{\Ai}^{(T)}\|_{2} \to 0$ a.s.~as $n\to\infty$.
\end{lemma}

\begin{lemma}\label{lem:HS_tail}
    $\lim\limits_{T\to\infty}\limsup\limits_{n\to\infty} \|\mathsf{K}_n-\mathsf{K}_{n}^{(T)}\|_{2}=0$ in probability. 
\end{lemma}

In Section \ref{subsec:tool}, we review the tools that will be used in the proofs of  Lemma \ref{lem:HS_T} and \ref{lem:HS_tail}. Then in Section \ref{subsec:diffusion}, we apply the classical theory of convergence of Markov chains to diffusions to show that the kernel $\mathsf{K}_n^{(T)}$ converges to  $\mathsf{K}_{\Ai}^{(T)}$ uniformly on $[0,T]^2$ for any fixed $T>0$. This  implies that $\|\mathsf{K}_n^{(T)} - \mathsf{K}_{\Ai}^{(T)}\|_{2}\to 0$. 

The proof of Lemma \ref{lem:HS_tail} is the most technically challenging part. It relies on a careful analysis of the vector $\ma{u}=[u_1,u_2,\dots,u_n]^\top$, defined via \eqref{eq:udef1} and \eqref{eq:udef2}, across different regimes. See Sections \ref{subsec:martingale} and \ref{subsec:concentration} for the bounds on the ratios of $u_i$. Using the discrete Wronskian identity \eqref{eq:uv_Wr}, we can translate these bounds into bounds on $|\mathsf{K}_n-\mathsf{K}_{n}^{(T)}|$. This allows us to prove Lemma \ref{lem:HS_tail} in Section \ref{subsec:HS_tail}.

\begin{proof}[Proof of Theorem \ref{thm:main_HS}]
    For any $\eps>0$, by Lemmas \ref{lem:Airy_tail} and \ref{lem:HS_tail}, we can find $T$ large such that
    \[
\pr(\|\mathsf{K}_{\Ai}-\mathsf{K}_{\Ai}^{(T)}\|_{2}\ge \eps/4)\le \eps/4, \quad\text{and}\quad \limsup_{n\to\infty} \pr(\|\mathsf{K}_n-\mathsf{K}_{n}^{(T)}\|_{2}\ge \eps/4)\le \eps/4.
    \]    
    Moreover, by Lemma \ref{lem:HS_T}, there exists a coupling of $\mathsf{K}_n^{(T)}$ and $\mathsf{K}_{\Ai}^{(T)}$ such that $\|\mathsf{K}_n^{(T)}-\mathsf{K}_{\Ai}^{(T)}\|_{2} \to 0$ almost surely, and hence in probability. This implies $\lim_{n\to \infty} \pr(\|\mathsf{K}_n^{(T)}-\mathsf{K}_{\Ai}^{(T)}\|_{2} \ge \eps/4)=0$. From \eqref{eq:large_a_triangle} we now obtain $\limsup_{n\to \infty} \pr(\|\mathsf{K}_n-\mathsf{K}_{\Ai}\|_{2}\ge \eps)\le \eps$, which implies the statement of the theorem.
\end{proof}

\begin{proof}[Proof of Theorem \ref{thm:main_lower} when $a_n\gg (\log\log n)^3$]
   The claimed weak convergence of the spectrum  follows from the norm-resolvent convergence proved in Theorem \ref{thm:main_HS}.
\end{proof}

\subsection{Various tools}\label{subsec:tool}

We summarize some of the tools needed for our proofs.

\subsubsection{Moments, concentration bounds, and a ``good'' event}

The following   concentration bounds for the standard normal and the chi-distributions are standard.

\begin{fact}
    Suppose $G\sim N(0,1)$ and $x>0$ then
    \begin{align}\label{eq:gauss_tail}
        \pr(|G|\ge x)\le e^{-\frac{x^2}{2}}.
    \end{align}
\end{fact}

\begin{fact}\label{fact:chi_conc}
Suppose that $X\sim \chi_p$ with $p>0$. Then for $x>0$ we have
\begin{align}\label{eq:chi_conc_largep}
        \pr(|X-\sqrt{p}|\ge x)\le 2 e^{-\frac{x^2}{2}}.
    \end{align}
Moreover, if $0< p\le 1$ and $0<c<1$ then 
\begin{align}\label{eq:chi_lower_smallp}
        \pr(X\le c \sqrt{p})\le 2 c^p.
\end{align}
\end{fact}

The following high-probability event allows us to control the fluctuations of $X_k,1\le k\le n$ and $Y_k,1\le k\le n-1$ defined in \eqref{eq:XYdef}. Define
\begin{align*}
    \mathcal{A}_{n,1}&=\{ |X_k-\sqrt{n-k+2a+1}|\le 2\beta^{-1/2} \sqrt{\log(\beta(n-k+2a+1))}, 1\le k\le n  \},\\
    \mathcal{A}_{n,2}&=\{|Y_k-\sqrt{n-k}|\le 2 \beta^{-1/2}\sqrt{\log(\beta(n-k))}, 1\le k\le n-\lfloor\sqrt{a}\rfloor\},\\
    \mathcal{A}_{n,3}&=\{(\log a)^{-1}(n-k)\le  Y_k^2\le  (2\log a) (n-k), n-\lfloor\sqrt{a}\rfloor\le k\le n-1\},
\end{align*}
and 
\begin{align}
          \mathcal{A}_n&=\mathcal{A}_{n,1}\cap \mathcal{A}_{n,2}\cap \mathcal{A}_{n,3}.\label{eq:A_n}
\end{align}
\begin{lemma}\label{lem:good_event}
Let $\beta>0$ be fixed. Then there is a constant $c_\beta>0$ so that 
    \begin{align*}
        \pr( \mathcal{A}_n)\ge 1- c_\beta (\log a_n)^{-\beta/2}.
    \end{align*}
\end{lemma}
\begin{proof}
    The estimate follows from Fact \ref{fact:chi_conc} and simple union bounds on $\pr(\mathcal A_{n,i}^c)$ for $1\le i\le 3$.
\end{proof}

\subsubsection{Diffusion limit of Markov chains}

We use the following version of a classical result about  convergence of discrete Markov chains to diffusions due to Ethier and Kurtz, see Theorem 7.4.1 and Corollary 7.4.2 of \cite{EthierKurtz}.
\begin{proposition}\label{prop:diffuapprox}
    Let $\ma{s}=(s_{ij})$ be a continuous, symmetric, positive semi-definite $d\times d$ matrix-valued function on $\R^{d}$, $\ma{b}:\R^{d}\rightarrow \R^{d}$ a continuous function. Let $X(t)$ be an It\^{o} diffusion with generator
     $$Af=\frac{1}{2}\sum_{i,j}s_{ij}\partial_{i}\partial_{j}f+\sum_{i}b_{i}\partial_{i}f, \qquad f\in C_{c}^{\infty}(\R^{d}),$$
    and initial distribution $\nu$.

    Suppose $X_n(i), i\ge 0$ is a Markov chain in $\R^d$ with increments $Z_n(i+1)=X_n(i+1)-X_n(i)$ such that for each $r>0$
        \begin{equation}\label{eq:moments_conditions}
        \begin{split}
            \sup_{|x|\le r}\left|m_n\ev[Z_n(i)|X_n(i-1)=x]-\ma{b}(x)\right| &\rightarrow 0,\\
            \sup_{|x|\le r}\left|m_n \ev[Z_n(i)Z_n(i)^{\top}|X_n(i-1)=x]-\ma{s}(x)\right|&\rightarrow 0,\\
            \sup_{|x|\le r}\left|m_n\ev[|Z_n(i)|^{4}|X_n(i-1)=x]\right|&\rightarrow 0,\\
        \end{split}
    \end{equation}
    uniformly for $i/m_n$ on compact subsets of $\R_+$ as $n\to\infty$, and $X_n(0)\Rightarrow \nu$ in law. Then 
    we have
     \begin{align}\label{eq:path_conv}
       \left( X_n({\lfloor t m_n\rfloor}, t\ge 0) \right)\Rightarrow \left(X(t), t\ge 0\right)
     \end{align}
     in law, with respect to the Skorokhod topology on compact subsets of $\R_+$.
\end{proposition}
\begin{remark}\label{rmk:Skorokhod}
     For any fixed $T$, we can apply the Skorokhod's representation theorem to get a coupling of the Markov chains $X_n(\cdot), n\ge 1$ and the diffusion $X(\cdot)$ so that the path convergence \eqref{eq:path_conv} happens a.s.~uniformly on $[0,T]$.
\end{remark}

\subsubsection{A martingale fluctuation lemma}

We need the following variant of Freedman's inequality, proved in \cite{DvZ2001}.
\begin{theorem}[\cite{DvZ2001}]\label{thm:cutoff_Freedman}
Suppose that $\xi_k, k\ge 0$ are martingale differences with respect to a filtration $\FF_k,k\ge 0$, and let $M_k=\sum_{j=1}^k \xi_j$. 
Assume that we have the bounds $\ev[\xi_k^2|\FF_{k-1}]\le \sigma_k^2$ on the conditional second moments. For a fixed $\eta>0$ define
\begin{align}
    H_k^{(\eta)}=\sum_{j=1}^k \xi_j^2 \ind(|\xi_j|>\eta)+\sum_{j=1}^k \sigma_j^2. 
\end{align}
Then
\begin{align}\label{eq:mart_fluct}
    \pr(\max_{1\le k\le n} |M_k|>t, H_n^{(\eta)}\le L)\le 2 \exp\big(-\tfrac{t^2}{2(L+\eta t)}\big).
\end{align}
\end{theorem}
In \cite{DvZ2001} the sharper upper bound $2 e^{-\tfrac{t^2}{2L} \varphi(\tfrac{\eta t}{L}) }$ is proved with $\varphi(x)=\tfrac{2}{x^2}\int_0^x \log(1+y) dy$. The bound \eqref{eq:mart_fluct} follows from the  inequality $\varphi(x)\ge \tfrac{1}{1+x}$ for $x>0$.

Our next result gives a uniform upper bound on the increments of a martingale in terms of upper bounds on the conditional variances of the martingale differences.

\begin{lemma}\label{lem:martingale_fluct}
    Suppose that $\zeta_j, 1\le j\le n$ are martingale differences with respect to the filtration $\FF_j, 0\le j\le n$. Assume that $\sigma_j^2, 1\le j\le n$ are non-decreasing positive reals bounded by 1 with
$\ev[\zeta_j^2|\FF_{j-1} ]\le \sigma_j^2$.
Introduce the notations
\begin{align*}
    M(k_1,k_2)=\sum_{j=k_1+1}^{k_2} \zeta_j, \qquad \sigma(k_1,k_2)^2=\sum_{j=k_1+1}^{k_2} \sigma_j^2, \qquad 0\le k_1<k_2\le n. 
\end{align*}
For $x\ge 0$ let 
\[
\gamma(x)=\max(x, x^{1/2}). 
\]
Then for $c$ large enough we have 
\begin{align*}
   & \pr\left(|M(k_1,k_2)|\ge c \cdot \gamma(\sigma(k_1,k_2))
   \log\big(4+\tfrac{\sigma^2(0,k_1)}{\sigma^2(k_1,k_2)}+3|\log\sigma(k_1,k_2)|\big)\text{ for some $0\le k_1<k_2\le n$}\right)\\&\qquad \qquad \qquad \qquad \le 2^{-\frac{c}{50}}+\pr(|\zeta_j|> \sigma_j^{1/2} \text{ for some } 1\le j\le n).
\end{align*}

\end{lemma}

\begin{proof}
We first prove that for any $K\ge 1$ and $0\le k_1<k_2\le n$ we have
\begin{align}\label{eq:fluct_1}
    \pr(\max_{k_1<j\le k_2} |M(k_1,j)|>K \gamma(\sigma(k_1,k_2)), |\zeta_j|\le \sigma_j^{1/2} \text{ for } k_1<j\le k_2)\le 2 e^{-\frac{K}{4}}. 
\end{align}
Apply Theorem \ref{thm:cutoff_Freedman} for $\zeta_{k_1+j}, 1\le j\le  k_2-k_1$ with $t=K \gamma(\sigma(k_1,k_2))$, $\eta=\sigma_{k_2}^{1/2}$ and $L=\sigma^2(k_1,k_2)$. Using the monotonicity of $\sigma_j$,  the probability in \eqref{eq:fluct_1} is bounded from above by  
\begin{align*}
   & \pr(\max_{k_1<\ell\le k_2} |M(k_1,\ell)|>K \gamma(\sigma(k_1,k_2)), |\zeta_j|\le \sigma_{k_2}^{1/2} \text{ for } k_1<j\le k_2)\\
   &\qquad \qquad \qquad\qquad \le 2 \exp\Bigg(-\frac{K^2 \gamma(\sigma(k_1,k_2))^2}{2\sigma^2(k_1,k_2)+2K \gamma(\sigma(k_1,k_2)) \,\sigma_{k_2}^{1/2}}\Bigg).
\end{align*}
We have $\sigma_{k_2}\le \sigma(k_1, k_2)$, hence
\begin{align*}
    \frac{K^2 \gamma(\sigma(k_1,k_2))^2}{2\sigma^2(k_1,k_2)+2K \gamma(\sigma(k_1,k_2)) \sigma_{k_2}^{1/2}}\ge  \frac{K^2 \gamma(\sigma(k_1,k_2))^2}{2\sigma^2(k_1,k_2)+2 K \gamma(\sigma(k_1,k_2)) \sigma(k_1,k_2)^{1/2}}\ge \frac{K^2}{2(1+K)}\ge \frac{K}{4},
\end{align*}
proving \eqref{eq:fluct_1}.

Set $p_{\min}=\lceil \log_2 \sigma_1^2\rceil$ and $p_{\max}=\lceil \log_2 \sigma(0,n)^2 \rceil$.
Define
$q_p=\max(k\le n: \sigma_k^2\le 2^{p})$ for $p\in [p_{\min},p_{\max}]$. 
For $p_{\min}\le p\le p_{\max}$, using the greedy algorithm we can find indices $\ell_{0,p}=0<\ell_{1,p}<\dots<\ell_{j_p,p}=q_p$ with the property that 
\begin{align}
    2^p  \le \sigma(\ell_{i,p}, \ell_{i+1,p})^2\le 3 \cdot 2^p, \qquad \text{ for all } 0\le i\le j_p-1.
\end{align}
Indeed: if $\ell_k<q_p$ has been defined already then we can set $\ell_{k+1}$ as the smallest index $k\le q_p$ with $\sigma(\ell_k,\ell_{k+1})^2\ge 2^p $, and if such an index does not exist then we can redefine $\ell_k$ to be $q_p$. 
Note that $j_p\le \sigma(0,q_p)^2 2^{-p}$.

We now apply \eqref{eq:fluct_1} for $k_1=\ell_{i,p}, k_2=\ell_{i+1,p}$ for each $p\in [p_{\min},p_{\max}]$ and $0\le i< j_p$ with $K_{i,p}=\tfrac{c}{10} \log(2+|p|+i)$. Using the union bound we can bound the  probability that $|\zeta_j|\le \sigma_j^{1/2}, 1\le j\le n$,  and there is a pair $i,p$ with $\max\limits_{\ell_{i,p}<j\le \ell_{i+1,p}} |M(\ell_{i,p},j)|>K_{i,p} \gamma(\sigma(\ell_{i,p},\ell_{i+1,p}))$  by
\[
\sum_{p={p_{\min}}}^{p_{\max}} \sum_{i=0}^{j_p-1} 2 e^{-\frac{c}{40} \log(2+|p|+i)}< 4 \sum_{j=2}^\infty  j^{-\frac{c}{40}+1}<  2^{-\frac{c}{50}}.
\]

Next we claim that if $\max\limits_{\ell_{i,p}<j\le \ell_{i+1,p}} |M(\ell_{i,p},j)|\le K_{i,p} \gamma(\sigma(\ell_{i,p},\ell_{i+1,p}))$ for all relevant $i,p$ then we have a bound on each $|M(k_1,k_2)|$. Indeed, choose $p$ integer so that  $2^{p-1}  < \sigma(k_1,k_2)^2 \le 2^{p}$. Then there are $0\le i_1<i_2\le j_p$ so that $\ell_{i_1,p}\le k_1 <\ell_{i_1+1,p}$, $ \ell_{i_2-1,p}<k_2 \le \ell_{i_2,p}$ and $i_2-i_1\le 2$. By the triangle inequality
\begin{align}
    |M(k_1,k_2)|&\le 2 \sum_{j=i_1}^{i_2-1} \max\limits_{\ell_{j,p}< u \le \ell_{j+1,p}} |
    M(\ell_{j,p},u)|\le 2\sum_{j=i_1}^{i_2-1} K_{j,p}\gamma(\sigma(\ell_{j,p},\ell_{j+1,p})).
\end{align}
Because of our definitions
\begin{align}
    \sigma(\ell_{j,p},\ell_{j+1,p})^2\le 3 \cdot 2^p  \le 6 \sigma(k_1,k_2)^2.
\end{align}
For  $0<x\le y$ we have $\gamma(\sqrt{6} x)\le \sqrt{6} \gamma( y)$, hence for $i_1\le j<i_2$ we have
\begin{align*}
    \gamma(\sigma(\ell_{j,p},\ell_{j+1,p}))\le \sqrt{6}  \gamma(\sigma(k_1,k_2)).
\end{align*}
In addition, for $i_1\le j<i_2$ we have 
\begin{align*}
    K_{j,p}\le K_{i_2-1,p}
    &\le \frac{c}{10} \log\left(2+|\log_2 \sigma(k_1,k_2)^2|+\frac{\sigma(0,k_2)^2}{2^p}\right)\\
    &\le \frac{c}{10} \log\left(4+3 |\log \sigma(k_1,k_2)|+1+\frac{\sigma(0,k_1)^2}{\sigma(k_1,k_2)^2}\right).
\end{align*}
which yields
\begin{align}
    |M(k_1,k_2)|&\le c \gamma(\sigma(k_1,k_2))\log\left(4+3 |\log \sigma(k_1,k_2)|+\frac{\sigma(0,k_1)^2}{\sigma(k_1,k_2)^2}\right)
\end{align}
proving the lemma. 
\end{proof}

The following  lemma can be used to replace the fluctuation upper bound with a simpler function. Its proof follows from the asymptotics of the logarithm function.
\begin{lemma}\label{lem:martingale_fluc2}
    Suppose that $0<h$, $0<x$. Then
\begin{align}
    \gamma(\sqrt{h})\log(4+ \tfrac{x}{h}+3 |\log \sqrt{h}|)\le c (1+h) \log(\max(x,2)) 
\end{align}
with an absolute constant $c>0$.
\end{lemma}

\subsection{Diffusion limit on $[0,T]$}\label{subsec:diffusion}
The goal of this section is to prove Lemma \ref{lem:HS_T} using Proposition \ref{prop:diffuapprox}. 
Let $u_k,1\le k\le n$ be defined as in \eqref{eq:uv_1} associated with $\ma{M}_n=\ma{L}_n\ma{L}_n^\top -\mu_n\ma{I}_n$, then $u_k$ satisfies the recursion \eqref{eq:udef2}
\begin{align}\label{eq:u_recursion}
    u_{k+1}=\frac{X_k^2+Y_k^2-\mu_n}{X_{k+1}Y_k} u_k-\frac{X_k Y_{k-1}}{X_{k+1}Y_{k}} u_{k-1}, \qquad 2\le k\le n-1,
\end{align}
with $u_1=1$ and $u_2=\frac{X_1^2+Y_1^2-\mu_n}{X_2Y_1}$. Since $X_k,Y_k$ have continuous distributions, we have a.s.~$u_k\ne 0$ for all $1\le k\le n$.
We  introduce the auxiliary sequence $\bar u_k,1\le k\le n$ with the definition
\begin{align}
    \bar u_k= u_k \prod_{j=1}^{k-1} \frac{X_{j+1}}{Y_{j}}\gamma_j,\qquad \gamma_j := \sqrt{\frac{n-j}{n-j+2a}}.
\end{align}
Then $\bar u_k$ satisfies
\begin{align}\label{eq:ubar_recursion}
    \bar u_{k+1}&=\frac{X_k^2+Y_k^2-\mu_n}{Y_k^2} \cdot  \gamma_k  \bar u_k-\frac{X_k^2}{Y_k^2}  \gamma_k \gamma_{k-1} \bar u_{k-1}, \qquad 2\le k\le n-1.
\end{align}
For $1\le k\le n-1$ define
\begin{align}\label{def:Z1Z2}
  Z_{1,k}= \frac{X_k^2+Y_k^2-\mu_n}{Y_k^2} \cdot \gamma_k-2,\qquad Z_{2,k}&= \frac{X_k^2+Y_k^2-\mu_n}{Y_k^2} \cdot \gamma_k-\frac{X_k^2}{Y_k^2} \cdot \gamma_k \gamma_{k-1}-1.
\end{align}
Then from \eqref{eq:ubar_recursion} we obtain
\begin{align}\label{eq:ubar_diff_recursion}
    &(\bar u_{k+1}-\bar u_k)-(\bar u_k-\bar u_{k-1})=Z_{1,k}(\bar u_k-\bar u_{k-1})+Z_{2,k} \bar u_{k-1}, \qquad 2\le k\le n-1.
\end{align}
With the convention $u_0=\bar u_0 =0$, the equations \eqref{eq:ubar_recursion} and \eqref{eq:ubar_diff_recursion} also hold for $k=1$. Note also that the random variables $Z_{1,k},Z_{2,k}$ are independent of $\bar u_j,j\le k$, hence $(\bar u_k,\bar u_{k+1}-\bar u_{k}),0\le k\le n-1$ forms a Markov chain. 

Using the independence of the $\chi$-distributed random variables, we can compute the moments of $Z_{1,k},Z_{2,k}$ directly. The following proposition summarizes the asymptotics of these quantities, its proof is postponed to the Appendix. The idea of the proof is straightforward, but the required estimates are somewhat delicate.

\begin{proposition}\label{prop:moments}
  For $1\le k\le n_1=n-\lfloor a\mathfrak{f}(a)\rfloor$ we have 
    \begin{align*}
        \ev[Z_{1,k}]& = O\left(\frac{1}{n-k}+\frac{a^2 k}{n (n-k)^2}\right),\\
        \ev[Z_{2,k}]&=\frac{a^2 k}{n(n-k)^2}+O\left(\frac{a^3 k}{n(n-k)^3}+\frac{a}{(n-k)^2}\right),\\
        \Var[Z_{1,k}] &=\frac{4}{\beta(n-k)}+O\left(\frac{a}{(n-k)^2}\right),\\  \Var[Z_{2,k}]&=\frac{4a^2}{\beta(n-k)^3}+O\left(\frac{a^3}{(n-k)^3 n}\right),\\
        \Cov(Z_{1,k},Z_{2,k})&=O\left(\frac{a}{(n-k)^{2}}\right),
    \end{align*}
    and 
    \begin{align*}
         \ev[Z_{1,k}^4] =O\left(\frac{1}{(n-k)^{2}}\right),  \qquad \ev[Z_{2,k}^4]=O\left(\frac{a^8 k^4}{n^4(n-k)^8}+\frac{a^4}{(n-k)^6}\right).
    \end{align*}
    Each of the $O(\cdot)$ error terms have constants that only depend on $\beta$.
\end{proposition}

Let $\psi_d$ be the solution of the equation \eqref{AirySDE_1} with Dirichlet initial condition. 
Applying Proposition \ref{prop:diffuapprox} with the moment estimates of Proposition \ref{prop:moments}, we obtain the following statement. 
\begin{proposition}\label{prop:ubar_conv}
    Let $m_n:=a^{-2/3}n$, then we have 
    \begin{align}\label{eq:ubar_conv}
        (m_n^{-1}\bar u_{\lfloor xm_n\rfloor}, \bar u_{\lfloor xm_n\rfloor +1}-\bar u_{\lfloor xm_n\rfloor}) \Rightarrow (\psi_d(x),\psi_d'(x)),
    \end{align}
    in law with respect to the Skorokhod topology on compact subsets of $\R_+$.
\end{proposition}
\begin{proof}
    We rewrite the recursions \eqref{eq:ubar_recursion} and \eqref{eq:ubar_diff_recursion} as 
    \begin{align}\label{eq:mat_rec}
        \binom{m_n^{-1}\bar u_k}{\bar u_{k+1}-\bar u_{k}} - \binom{m_n^{-1}\bar u_{k-1}}{\bar u_k - \bar u_{k-1}}= \mat{0}{m_n^{-1}}{m_nZ_{2,k}}{Z_{1,k}}\binom{m_n^{-1}\bar u_{k-1}}{\bar u_k - \bar u_{k-1}}:= A_k\binom{m_n^{-1}\bar u_{k-1}}{\bar u_k - \bar u_{k-1}} .
    \end{align}
In order to use Proposition \ref{prop:diffuapprox}, we just need to estimate the moments of $A_{\lfloor xm_n\rfloor}(y_1,y_2)^\top$ for $(y_1,y_2)^\top$ in a fixed compact set of $\R^2$, and $x\in[0,T]$. 

By Proposition \ref{prop:moments} we have
    \begin{align*}
         m_n \ev\left[A_{\lfloor xm_n\rfloor} \binom{y_1}{y_2}\right] \to  \mat{0}{1}{x}{0}\binom{y_1}{y_2},&
\\[5pt]
        m_n \ev\left[A_{\lfloor xm_n\rfloor} \mat{y_1^2}{y_1y_2}{y_1y_2}{y_2^2}A_{\lfloor xm_n\rfloor}^\top\right] \to \mat{0}{0}{0}{\frac{4}{\beta}y_2^2},&\qquad
            m_n \ev\left[\left|A_{\lfloor xm_n\rfloor} \binom{y_1}{y_2}\right|^4\right]\to 0,
    \end{align*}
    uniformly for $x\in [0,T]$. We note that, by construction, the term $Z_{2,k}$ is multiplied by $m_n$ in the matrix recursion \eqref{eq:mat_rec}, and hence it gives the dominant contribution.
    
     By definition, the initial conditions of the discrete process satisfy $m_n^{-1}\bar u_0= 0$, and $\bar u_1-\bar u_0=1$.
    Therefore, by Proposition \ref{prop:diffuapprox}, we have that $(m_n^{-1}\bar u_{\lfloor xm_n\rfloor},  \bar u_{\lfloor xm_n\rfloor +1}-\bar u_{\lfloor xm_n\rfloor})$ converges in law to the It\^o diffusion satisfying the SDE
    \[
    d\binom{\psi(x)}{\psi'(x)} = \mat{0}{dx}{\frac{2}{\sqrt{\beta}}dB_x+xdx}{0}\binom{\psi(x)}{\psi'(x)}
    \]
    with $\psi(0)=0, \psi'(0)=1$, which is exactly the $\psi_d$ process. 
\end{proof}
Proposition \ref{prop:ubar_conv} can be extended to  cover the case when the recursion \eqref{eq:ubar_recursion} is started from a different initial condition. 

\begin{corollary}\label{cor:wbar_conv}
    Let $\bar w_k, 0\le k\le n-1$ be the solution to the recursion
    \begin{equation}\label{eq:wbar_recursion}
    \bar w_{k+1}=\frac{X_k^2+Y_k^2-\mu_n}{Y_k^2} \,  \gamma_k  \bar w_k-\frac{X_k^2}{Y_k^2}  \gamma_k \gamma_{k-1} \bar w_{k-1},\qquad \bar w_0=\bar w_1=m_n.
    \end{equation}
    Let $\psi_*$ be the solution to the SDE \eqref{AirySDE_1} with Neumann initial condition $\psi_*(0)=1,\psi_*'(0)=0$. Then we have
        \begin{align}\label{eq:wbar_conv}
        (m_n^{-1}\bar w_{\lfloor xm_n\rfloor}, \bar w_{\lfloor xm_n\rfloor +1}-\bar w_{\lfloor xm_n\rfloor}) \Rightarrow (\psi_*(x),\psi_*'(x)),
    \end{align}
    in law with respect to the Skorokhod topology on compact subsets of $\R_+$. Moreover, the convergence results \eqref{eq:ubar_conv} and \eqref{eq:wbar_conv} hold jointly. 
\end{corollary}

Define 
\begin{equation}\label{def:w}
    w_k:= \bar w_k \prod_{j=1}^{k-1}\frac{Y_j}{X_{j+1}\gamma_j},\qquad 1\le k\le n.
\end{equation}
The next result shows that the ratio $\frac{\bar u_k}{ u_k} = \frac{\bar w_k}{w_k}$ converges uniformly to $1$ for $1\le k\le n_0=\lfloor Tm_n\rfloor$.
\begin{proposition}\label{prop:ratio-to-1}
   We have $\bar u_k/u_k= \bar w_k/w_k\Rightarrow 1$ uniformly for $1\le k\le n_0$, as $n\to\infty$. 
\end{proposition}
\begin{proof}
Taking logarithm, we have 
\[
\log(\bar u_k/u_k) = \sum_{j=1}^{k-1}(W_j + \log \gamma_j),\qquad W_j:= \log (X_{j+1}/Y_j).
\]
By the independence of $X_{j+1}$ and $Y_j$, we obtain the following moment asymptotics of $W_j$ (see e.g.~Proposition 8 of \cite{RR}), 
\begin{align}\label{eq:Wj_moments}
    \ev[W_j] =-\log(\gamma_j)+O\big(\tfrac{a}{(n-j)^2}\big),\qquad \Var[W_j] = \frac{1}{\beta(n-j)}+O\big(\tfrac{a}{(n-j)^2}\big),
\end{align}
and $\ev[W_j^4]=O(a(n-j)^{-2})$ for $j\le n_0$.
Indeed, the moment estimates \eqref{eq:Wj_moments} holds as long as $n-k\ge \lfloor a\rfloor$. 

Applying Proposition \ref{prop:diffuapprox} to the Markov chain $\sum_{j=1}^{k-1}(W_j+\log \gamma_j)$, we see that on the time scale $m_n^{-1}=a^{2/3}/n$, the process $M_{\lfloor xm_n\rfloor},x\le T$ converge in law to constant $0$ uniformly on $[0,T]$. Taking the exponential completes the proof.
\end{proof}

We are now ready to prove Lemma \ref{lem:HS_T}.

\begin{proof}[Proof of Lemma \ref{lem:HS_T}]
    For any fixed $T>0$, by Proposition \ref{prop:ubar_conv} and Corollary \ref{cor:wbar_conv}, we have  
    \begin{align}\label{eq:ubarwbar}
    m_n^{-1}\bar u_{\lfloor x m_n\rfloor}\Rightarrow \psi_d(x),\qquad m_n^{-1}\bar w_{\lfloor x m_n\rfloor}\Rightarrow \psi_*(x),
    \end{align}
    jointly in law on $[0,T]$ in the Skorokhod topology. By Proposition \ref{prop:ratio-to-1}, the same statement holds for the processes $m_n^{-1} u_{\lfloor xm_n\rfloor}$ and $m_n^{-1} w_{\lfloor xm_n\rfloor}$ as well.
 
    Recall the kernel $\mathsf{K}_{\Ai}^{(T)}$ defined in \eqref{def:K_Ai_T}, where
    $\psi_T$ solves $\Airyop \psi_T=0$ with boundary conditions $\psi_T(0)=1,\psi_T(T)=0$. By linearity of the equation $\Airyop\psi=0$, we have 
    \[
    \psi_T(x) = \psi_*(x) - \frac{\psi_*(T)}{\psi_d(T)}\psi_d(x),
    \]
    see the proof of Lemma 5.1 of \cite{DLV} for additional details. 

Recall that $\ma{M}_n^{(T)}$ is the $n_0\times n_0$ upper-left submatrix of $\ma{M}\equiv \ma{M}_n$, and 
the kernel $\mathsf{K}_{n}^{(T)}$ is determined by the inverse of $\ma{M}_n^{(T)}$, see \eqref{def:Kn_T}.  By Lemma \ref{lem:tri_inverse}, for $1\le i\le j\le n_0$ we have $[(\ma{M}_n^{(T)})^{-1}]_{ij} = u_i^{(T)} v_j^{(T)}$, where $u_i^{(T)}=u_i$, and $v_i^{(T)}$ solves the recursions \eqref{eq:vdef1} and \eqref{eq:vdef2} corresponding to the tridiagonal matrix $\ma{M}_n^{(T)}$. Define 
\begin{align}\label{eq:def_vtilde}
    \widetilde{v}_k^{(T)}=\frac{w_k-u_k \frac{w_{n_0+1}}{u_{n_0+1}}}{X_2 Y_1 (w_1 u_2-u_1 w_2)}, \qquad 1\le k\le n_0+1.
\end{align}
With probability one this is well-defined, we will show that $\widetilde{v}_k^{(T)}={v}_k^{(T)}$ for $1\le k\le n_0$. By linearity, $\widetilde{v}_k^{(T)}$ satisfies \eqref{eq:u_recursion} for $2\le k\le n_0-1$, and satisfies $\widetilde{v}_{n_0+1}^{(T)}=0$ by definition. This shows that it satisfies \eqref{eq:vdef2} and the second equation in \eqref{eq:vdef1}. This implies that $\widetilde{v}_k^{(T)}=c {v}_k^{(T)}$ for some $c\neq 0$. Then from \eqref{eq:def_vtilde} we have
\[
c(v_1^{(T)} u_2-u_1 v_2^{(T)})=\widetilde{v}_1^{(T)} u_2-u_1 \widetilde{v}_2^{(T)}=\frac{1}{X_2 Y_1}.
\]
From \eqref{eq:discrete_Wr} and the definition of $\ma{L}_n$ we have 
\[
v_1^{(T)} u_2-u_1 v_2^{(T)}=\frac{1}{X_2 Y_1},
\]
which shows $c=1$ and $\widetilde{v}_k^{(T)}={v}_k^{(T)}$. By definition, we have
\[
X_2 Y_1(w_1 u_2-u_1 w_2)=X_2 Y_1 \frac{Y_1}{\gamma_1 X_2}(\bar w_1 \bar u_2-\bar u_1 \bar w_2)= \frac{Y_1^2}{\gamma_1} ((\bar w_1 (\bar u_2-\bar u_1))-\bar u_1 (\bar w_2-\bar w_1)).
\]
Hence,
\begin{align*}
    n v_k^{(T)}=\left(m_n^{-1} w_k -m_n^{-1} w_k \frac{m_n^{-1} w_{n_0+1}}{m_{n^{-1} u_{n_0+1}}}\right) \left(m_n^{-1} \bar w_1 (\bar u_2-\bar u_1))-m_n^{-1} \bar u_1 (\bar w_2-\bar w_1)\right)^{-1}   \frac{n \gamma_1}{Y_1^2}.
\end{align*}
From Proposition \ref{prop:ubar_conv}, Corollary \ref{cor:wbar_conv}, and Proposition \ref{prop:ratio-to-1} we have
\begin{align*}
    m_n^{-1} w_{\lfloor xm_n\rfloor} -m_n^{-1} w_{\lfloor xm_n\rfloor} \frac{m_n^{-1} w_{n_0+1}}{m_{n^{-1} u_{n_0+1}}}\Rightarrow \psi_*(x)-\psi_d(x) \frac{\psi_*(T)}{\psi_d(T)}=\psi_T(x)
\end{align*}
in law on $[0,T]$ with respect to the Skorokhod topology, and
\[
    m_n^{-1} \bar w_1 (\bar u_2-\bar u_1))-m_n^{-1} \bar u_1 (\bar w_2-\bar w_1)\to \psi_*(0) \psi_d'(0)-\psi_d(0)\psi_*'(0)=1,
\]
in probability as $n\to\infty$. By Fact \ref{fact:chi_conc} we have $\frac{n}{Y_1^2}\to 1$ in probability as $n\to \infty$, and together with $\gamma_1\to 1$ this implies 
\begin{align}\label{eq:v_conv}
    n v_{\lfloor xm_n\rfloor}^{(T)} \Rightarrow \psi_T(x)
\end{align}
in law on $[0,T]$ with respect to the Skorokhod topology. Moreover, this holds jointly with the limits in \eqref{eq:ubarwbar}.

Using the Skorokhod representation theorem, there is a coupling of $\ma{L}_n, n\ge 1$ and $\Airyop$ so that the limits \eqref{eq:ubarwbar}, \eqref{eq:v_conv} hold a.s.~uniformly on $[0,T]$. Under this coupling, for $0\le x\le y\le T$ we have
    \[
    \mathsf{K}_{n}^{(T)}(x,y) = a^{2/3} u_{\lfloor xm_n\rfloor}v_{\lceil ym_n\rceil}^{(T)} = \frac{a^{2/3}}{n} u_{\lceil xm_n\rceil} \cdot nv_{\lceil ym_n\rceil}^{(T)}\to \psi_d(x)\psi_T(y) = \mathsf{K}_{\Ai}^{(T)}(x,y)
    \]
    uniformly almost surely. (The change from lower to upper integer parts has no effect on the limit.) The convergence also holds for $0\le y\le x\le T$ by symmetry, from which a.s.~$\|\mathsf{K}_n^{(T)}-\mathsf{K}_{\Ai}^{(T)}\|_{2} \to 0$ follows as well.     
\end{proof}

\subsection{Extended Airy behavior using martingale arguments}\label{subsec:martingale}

The remaining subsections are devoted to prove Lemma \ref{lem:HS_tail}. In all of the subsequent statements we will assume that $a_n\to \infty$ with $(\log\log n)^3\ll a_n\ll n$.

For an invertible tridiagonal matrix $\ma{M}$ of the form \eqref{eq:tridiagonal}, by Lemma \ref{lem:tri_inverse} and the discrete Wronskian identity \eqref{eq:discrete_Wr}, we have 
\[
[\ma{M}^{-1}]_{ij} = u_i v_j = \frac{u_i}{u_j}  \left(\sum_{\ell = j}^{n-1} \frac{u_j^2}{u_\ell u_{\ell+1}e_\ell}+\frac{u_j^2}{u_n(d_n-e_{n-1}u_{n-1})}\right),\quad 1\le i\le j\le n.
\]
This shows that $\|\ma{M}\|_{\textup{HS}}^2$ can be controlled if we can control  $u_{k_1}/u_{k_2}$ for $k_1< k_2$. This motivates the study of the relative increment $(u_{k+1}-u_k)/u_k$, which can be viewed as a discrete analogue of the Riccati transform.

We introduce the inhomogeneous scaling parameter
\[
h_k = \frac{a^{2/3}}{n-k},\qquad 1\le k\le n_1,
\]
and the discrete Riccati transform
\begin{align*}
    \bar p_{k+1} = h_{k+1}^{-1}\left(\frac{\bar u_{k+1}}{\bar u_k} -1\right),\qquad 1\le k\le n_1.
\end{align*}
(We do not denote the dependence on $n$ in $h_k, \bar p_k$.)
By the recursion \eqref{eq:ubar_diff_recursion}, we have 
\begin{align}\label{eq:pbar_recursion}
    \bar p_{k+1} - \bar p_k 
    & = \frac{h_{k+1}^{-1}Z_{2,k}}{1+h_{k}\bar p_k} +\frac{h_{k}h_{k+1}^{-1}(1+Z_{1,k})-1}{1+h_{k}\bar p_k}\bar p_k -\frac{h_{k}\bar p_k^2}{1+h_{k}\bar p_k}.
\end{align}
Note that
\[
\bar p_{k+1}=\frac{n-k-1}{n} \cdot \frac{\bar u_{k+1}-\bar u_k}{m_n^{-1} \bar u_k},
\]
and for $1\le k\le n_0$ the ratio $\frac{n-k-1}{n}$ converges to 1. Hence by Proposition \ref{prop:ubar_conv} the process $\bar p_{\lfloor x m_n\rfloor}, x\in [0,T]$ converges in law to the Riccati transform of the stochastic Airy diffusion 
$\mathfrak{p}(x):=\psi_d'(x)/\psi_d(x)$ with respect to the Skorokhod topology on the two-point compactification of the real line.  By It\^o's formula $\mathfrak{p}$ satisfies the SDE
\begin{align}\label{eq:Airysde}
    d\mathfrak{p}(t) = (t-\mathfrak{p}(t)^2)dt +\tfrac{2}{\sqrt{\beta}}dB(t),\qquad \mathfrak{p}(0)=\infty.
\end{align} 
Note that the process $\mathfrak{p}(\cdot)$ may blow up to $-\infty$ at finite time (but not to $+\infty$), in this case $\mathfrak{p}$ immediately restarts at $\infty$. By the following lemma there are a.s.~only finitely many such blow-ups.
\begin{lemma}[Proposition 4.1 in \cite{DLV}]\label{lem:Airy_asymp}
    There exists a.s.~finite random time $\mathfrak{t}<\infty$ such that 
\begin{align}\label{eq:Airysde_asymp}
    |\mathfrak{p}(t)-\sqrt{t}|\le t^{-1/4}\log t,\qquad \text{ for all $t\ge \mathfrak{t}$.}
\end{align}
\end{lemma}
The proof of Lemma \ref{lem:Airy_asymp} is based on an exact coupling of $\mathfrak{p}$ with explicitly solvable diffusions, combined with a careful analysis using tools from stochastic calculus. 
We provide a heuristic argument to show that if  $\mathfrak{p}(T)$ is close to $\sqrt{T}$ for large $T$, the process  $\mathfrak{p}(t)$  remains close to $\sqrt{t}$ for $t\ge T$ with a high probability. 

For fixed $0<\delta<1/5$, assume $|\mathfrak{p}(T)-\sqrt{T}|\le \delta\sqrt{T}$. We will show that with high probability $(1-4\delta)\sqrt{t}\le \mathfrak{p}(t)\le (1+4\delta)\sqrt{t}$ for $t\ge T$. We focus on the lower bound,  the upper bound can be treated similarly. 
Let 
\begin{align*}
    \mathfrak{t}_0:=\inf\{t\ge T:\mathfrak{p}(t)=(1-2\delta)\sqrt{t}\}, \quad \mathfrak{t}_1:=\inf\{t\ge \mathfrak{t}_0:\mathfrak{p}(t)= (1-4\delta)\sqrt{t} \text{ or }\mathfrak{p}(t)= (1-\tfrac12\delta)\sqrt{t}\}.
\end{align*}
We may assume that both $\mathfrak{t}_0$ and $\mathfrak{t}_1$ are finite, otherwise our lower bound is proven. We will argue that with high probability, $\mathfrak{p}(\mathfrak{t}_1)=(1-\frac{1}{2}\delta) \sqrt{\mathfrak{t}_1}$, i.e.~$\mathfrak{t}_1$ is triggered by the second equality in its definition. From this our statement follows.  
For $t\in[\mathfrak{t}_0,\mathfrak{t}_1]$ the drift term in the SDE \eqref{eq:Airysde} is lower bounded by $4\delta t$, which means that 
\begin{align*}
    \mathfrak{p}(\mathfrak{t}_1) &=\mathfrak{p}(\mathfrak{t}_0)+ \int_{\mathfrak{t}_0}^{\mathfrak{t}_1} (s-\mathfrak{p}(s)^2) ds  + \tfrac{2}{\sqrt{\beta}}\int_{\mathfrak{t}_0}^{\mathfrak{t}_1} dB(s)\\
    &\ge (1-2\delta)\sqrt{\mathfrak{t}_0} + 2\delta \mathfrak{t}_1(\mathfrak{t}_1-\mathfrak{t}_0)+\tfrac{2}{\sqrt{\beta}}(B(\mathfrak{t}_1)-B(\mathfrak{t}_0)).
\end{align*}
This implies that if $\mathfrak{p}(\mathfrak{t}_1)=(1-4\delta) \sqrt{\mathfrak{t}_1}$ then  $B(\mathfrak{t}_1)-B(\mathfrak{t}_0)$ has to be ``unusually negative'', and the probability of this will be small. (To make this precise, one would need a uniform bound on the increments of Brownian motion, see e.g.~Lemma 7.1 of \cite{DLV} for such a result.)

The above argument for $\mathfrak{p}(\cdot)$ will also serve as a guideline for proving the following result for the discrete process $\bar p_k,n_0\le k\le n_1$. This proposition shows that the discrete process $\bar p_k, n_0\le k\le n_1$ mimics the growth of the continuous process $\mathfrak{p}$, even though this interval goes beyond the regime where we have a process level convergence  $\bar p_{\lfloor \cdot m_n \rfloor}\Rightarrow \mathfrak{p}(\cdot)$.

\begin{proposition}\label{prop:pbar_square_root}
 For any fixed $0<\delta<\frac12$, we have 
    \begin{align}\label{eq:pbar_square_root}
      \lim_{T\to\infty}\liminf_{n\to\infty}\pr\left(\Big|\bar p_k-\sqrt{\tfrac{k}{m_n}}\Big|\le \delta\sqrt{\tfrac{k}{m_n}}\text{\, for all $n_0\le k\le n_1$}\right)=1.
    \end{align}
\end{proposition}
Compared with the heuristic arguments for $\mathfrak{p}$, the main technical difficulties in proving Proposition \ref{prop:pbar_square_root} are twofold: first, the discrete process $\bar p_k$ may jump across the stopping boundary; and second, one lacks the simple uniform fluctuation control that is available for Brownian motion. 

Our first step is to show that the increment $\bar p_{k+1}-\bar p_k$ cannot be too negative. Recall the event $\mathcal{A}_n$ defined in \eqref{eq:A_n}.
\begin{proposition}\label{prop:one-step-drop}
 For any fixed $0<\delta<\frac12$, we have for all $n$ large enough that
    \[
 \left\{\bar p_{k}\ge (1-\tfrac14\delta)\sqrt{\tfrac{k}{m_n}}\right\}\cap \mathcal{A}_n \subset \left\{\bar p_{k+1}\ge (1-\tfrac13\delta)\sqrt{\tfrac{k}{m_n}}\right\}\cap \mathcal{A}_n,\quad n_0\le k\le n_1.
\]
\end{proposition}

\begin{proof}
Fix $0< \delta<\frac12$. We rewrite \eqref{eq:pbar_recursion} as
\begin{align}
    h_{k+1}\bar p_{k+1} 
    &=Z_{2,k} + \frac{h_k\bar p_k}{1+h_k\bar p_k}(1+Z_{1,k}-Z_{2,k}).
\end{align}
We will show that $1+Z_{1,k}-Z_{2,k}$ is close to 1, and $Z_{2,k}$ cannot be too negative, from this it will follow that 
$\bar p_{k}\ge (1-\tfrac14\delta)\sqrt{\tfrac{k}{m_n}}$ implies 
$\bar p_{k+1}\ge (1-\tfrac13\delta)\sqrt{\tfrac{k}{m_n}}$ on the good event $\mathcal{A}_n$.

First note that $1+Z_{1,k}-Z_{2,k} = \frac{X_k^2}{Y_k^2}\gamma_k\gamma_{k-1}$. On the event $\mathcal{A}_n$, we have 
\begin{align}\label{eq:XY_fluc}
|X_k^2-  (n-k+2a+1) |\le  \log(n-k)\sqrt{n-k+2a+1},\quad |Y_k^2- (n-k)|\le \log(n-k)\sqrt{n-k},
\end{align}
for $n$ large enough. This implies that 
\[
1+Z_{1,k}-Z_{2,k} = \frac{X_k^2}{Y_k^2}\gamma_k\gamma_{k-1}= 1+O\big(\tfrac{\log(n-k)}{\sqrt{n-k}}\big)\ge 1-\tfrac{\delta}{100},
\]
for $n$ large enough.

Recall the definition of $Z_{2,k}$ from \eqref{def:Z1Z2}, by simple algebra we have
\begin{equation}
\begin{split}
    Z_{2,k} &-\ev[Z_{2,k}]\\
    &=\frac{2a \gamma_k}{(n-k+2a+1)(1+\gamma_{k-1})}\left(\frac{X_k^2}{Y_k^2}-\ev\left[\frac{X_k^2}{Y_k^2}\right]\right) -\frac{4 a^2 \gamma_k}{(\sqrt{n+2a}+\sqrt{n})^2} \left(\frac{1}{Y_k^2}-\ev\left[\frac{1}{Y_k^2}\right]\right). 
\end{split}
\end{equation}
For all $n$ large enough, using the exact values of the expectations (see the proof of Proposition \ref{prop:moments} in the Appendix) together with the bounds \eqref{eq:XY_fluc} we get
\begin{align}\label{eq:Z2k_bound}
    \left| Z_{2,k} -\ev[Z_{2,k}]\right|\le \frac{2a\log(n-k)}{(n-k)^{3/2}}.
\end{align}
Proposition \ref{prop:moments} gives $\ev[Z_{2,k}]= \frac{a^2 k}{n(n-k)^2}+O(\frac{a^3k}{n(n-k)^3}+\frac{a}{(n-k)^2})$, which leads to the bound 
\begin{align*}
    Z_{2,k}&\ge \frac{a^2 k}{2n(n-k)}-\frac{2a \log(n-k)}{(n-k)^{3/2}}\ge -\frac{2a \log(n-k)}{(n-k)^{3/2}}
    \ge  - \frac{\delta}{100} h_k\sqrt{k/m_n},
\end{align*}
for all $n$ large enough. The last step follows from the bound
\[
\frac{\log(n-k)}{\sqrt{n-k}} \ll \sqrt{\frac{k}{n}},\qquad \lfloor na^{-2/3}T\rfloor \le k\le n- \lfloor a \mathfrak{f}(a)\rfloor.
\]
Therefore, under the event $\{\bar p_k \ge (1-\frac14\delta)\sqrt{k/m_n}\}\cap \mathcal{A}_n$ we have 
\begin{align}\label{eq:pbar_bound}
    \bar p_{k+1}  &\ge   h_{k+1}^{-1}\left(
    (1-\tfrac{\delta}{100}) \frac{(1-\frac14\delta)h_k\sqrt{k/m_n}}{1+(1-\frac14\delta)h_k\sqrt{k/m_n}} - \tfrac{\delta}{100} h_k\sqrt{k/m_n}\right).
\end{align}
Since $h_k \sqrt{k/m_n}=\frac{a}{n-k}\sqrt{k/n}\le 2\mathfrak{f}(a)^{-1}\ll 1$, the first term in the parenthesis can be bounded by $(1-\tfrac{1}{100} \delta)^2 (1-\frac14\delta)h_k\sqrt{k/m_n}$ for $n$ large enough. Considering that $\frac{h_k}{h_{k+1}}=1-\frac{1}{n-k}$ and $n-k\ge \lfloor a \mathfrak{f}(a)\rfloor\gg 1$, the claimed inequality now follows from \eqref{eq:pbar_bound}.
\end{proof}

The previous result controls the single step fluctuation of $\bar p_k$. We are now ready to prove Proposition \ref{prop:pbar_square_root}.

\begin{proof}[Proof of Proposition \ref{prop:pbar_square_root}]
Fix $0<\delta<\frac12$. We  provide a detailed proof to show that with high probability $\bar p_k\ge (1-\delta)\sqrt{k/m_n}, n_0\le k\le n_1$, and sketch the proof for the bound on the probability of $\bar p_k\le (1+\delta)\sqrt{k/m_n},n_0\le k\le n_1$ at the end.

Rewrite \eqref{eq:pbar_recursion} as
\begin{align}
     \bar p_{k+1} - \bar p_k 
    & =  \frac{h_{k+1}^{-1} \ev[Z_{2,k}]}{1+h_{k}\bar p_k} -\frac{h_{k}\bar p_k^2}{1+h_{k}\bar p_k} +\frac{h_{k}h_{k+1}^{-1}(1+\ev[Z_{1,k}])-1}{1+h_{k}\bar p_k}\bar p_k \label{eq:pbar_drift} \\
    &\quad +\frac{h_{k+1}^{-1}\wtl Z_{2,k}}{1+h_{k}\bar p_k} +\frac{h_{k}h_{k+1}^{-1}\wtl Z_{1,k}}{1+h_{k}\bar p_k}\bar p_k,\label{eq:pbar_fluc}
\end{align}
where we use the notation $\wtl X:=X-\ev[X]$ for a random variable $X$ with finite expectation.
Define 
\begin{align*}
    \zeta_k := \left(\frac{h_{k+1}^{-1}\wtl Z_{2,k}}{1+h_{k}\bar p_k} +\frac{h_{k}h_{k+1}^{-1}\wtl Z_{1,k}}{1+h_{k}\bar p_k}\bar p_k\right) \ind(\mathcal{P}_k),
\end{align*}
where
\[
\mathcal{P}_k\equiv\mathcal{P}_{k,\delta}:=\{(1-\delta)\sqrt{k/m_n}\le \bar p_k\le (1-\tfrac15\delta)\sqrt{k/m_n}\}.
\]
(We do not denote the $n$-dependence of these quantities.)
It then follows from the independence of $Z_{1,k},Z_{2,k}$ and $\bar p_k$ that the running sum $M_k:=\sum_{j= n_0}^k\zeta_j, n_0\le k\le n_1$ forms a martingale with respect to the filtration $\mathcal{F}_{k}:=\sigma(X_1,Y_1,\cdots,X_{k-1},Y_{k-1})$.

For $\mathsf{c}>0$ introduce the event
\begin{align}
\mathcal{M}_{n,\mathsf{c}}^{(1)}:= \left\{   |M_{k_2}-M_{k_1}|\le \mathsf{c}\Big(1+\sum_{j=k_1+1}^{k_2} h_j\Big)\log(2+\sum_{j=n_0}^{k_1} h_j), \text{ for all }  n_0\le k_1<k_2\le n_1\right\}.
\end{align}
We first show that 
\begin{align}\label{eq:Mn_limit}
    \lim_{\mathsf{c}\to \infty} \liminf_{n\to \infty} \pr(\mathcal{M}_{n, \mathsf{c}}^{(1)} \cap \mathcal{A}_n)=1.
\end{align}
We want to apply Lemma \ref{lem:martingale_fluct}, for this we first need to  estimate $\Var(\zeta_k | \mathcal{F}_k)$.
We have, if $n$ is large enough, 
\begin{align}
           \Var\left(\frac{h_{k+1}^{-1}\wtl Z_{2,k}}{1+h_k\bar p_k}\ind(\mathcal{P}_k)
        \Big|\mathcal{F}_k\right) &\le  \frac{h_{k+1}^{-2}}{1-h_k \delta \sqrt{k/m_n}}\Var(Z_{2,k}) < \frac{8}{\beta} h_k.
\end{align}
Here the upper bound follows from Proposition \ref{prop:moments}, definition of $h_k$, and the fact $h_k \sqrt{k/m_n}=\frac{a}{n-k}\sqrt{k/n}\le 2\mathfrak{f}(a)^{-1}\ll 1$. Similarly,
\begin{align}
    \Var \left(\frac{h_kh_{k+1}^{-1}\wtl Z_{1,k}}{1+h_k\bar p_k}\bar p_k\ind(\mathcal{P}_k)
    \Big|\mathcal{F}_k\right)\le \frac{8}{\beta(n-k)} \le h_k.
\end{align}
Together these imply the bound 
\[
\Var(\zeta_k | \mathcal{F}_k)\le (\tfrac{16}{\beta}+2)h_k, \quad n_0\le k\le n_1.
\]
Similar to the estimate \eqref{eq:Z2k_bound} of $Z_{2,k}$ as in the proof of Proposition \ref{prop:one-step-drop}, on $\mathcal{A}_n$ we have
\begin{align*}
    |\zeta_k| \le 2h_k^{-1}\frac{a\log(n-k)}{(n-k)^{3/2}}=2 h_k^{1/2}  \log(n-k)  \le h_k^{1/4},\quad  n_0\le k\le n_1,
\end{align*}
where in the last step we used $n-k\ge \lfloor a \mathfrak{f}(a)\rfloor$ and the definition of $h_k$.
This means that 
\begin{align*}
    \pr(|\zeta_j| > h_j^{1/4}\text{ for some $ n_0\le k\le n_1$}) \le \pr(\mathcal{A}_n^c) \le c_\beta(\log a_n)^{-\beta/2}
\end{align*}
by Lemma \ref{lem:good_event}. Then \eqref{eq:Mn_limit} follows from Lemmas \ref{lem:martingale_fluct} and \ref{lem:martingale_fluc2}.

Next we define a sequence of stopping times.
Let $\kappa_0=n_0=\lfloor  T m_n\rfloor$, and define for $\ell\ge 1$ 
\begin{align*}
    \tau_\ell&:=\inf\{k\ge \kappa_{\ell-1}: \bar p_k<(1-\tfrac14\delta)\sqrt{k/m_n}\},\\
    \kappa_\ell&:=\{k>\tau_\ell:\bar p_k<(1-\delta)\sqrt{k/m_n}\text{ or } \bar p_k>(1-\tfrac15\delta)\sqrt{k/m_n}\}.
\end{align*}
Let 
\begin{align*}
   \mathcal{B}_{n,\delta}:=\left\{ (1-\tfrac15\delta)\sqrt{n_0/m_n}\le \bar p_{n_0}\le (1+\tfrac15\delta)\sqrt{n_0/m_n} \right\}. 
\end{align*}
By Proposition \ref{prop:ubar_conv} we have $\bar p_{n_0}=\bar p_{\kappa_0}\Rightarrow \mathfrak{p}(T)$. By Lemma \ref{lem:Airy_asymp} we have $\lim_{T\to \infty} \lim_{n\to \infty} \pr (\mathcal{B}_{n,\delta})=1$.

For the rest of the proof, we will work on the event 
\begin{align}\label{eq:En1}
   \mathcal{E}_{n,\delta, \sfc}^{(1)}:=\mathcal{A}_n\cap \mathcal{B}_{n,\delta}\cap \mathcal{M}_{n,\sfc}^{(1)}.
\end{align} By Proposition \ref{prop:one-step-drop}, we have $\bar p_{\tau_1}\ge (1-\tfrac13\delta)\sqrt{\tau_1/m_n}$.
Over the regime $\tau_1 \le k\le \kappa_1$, the fluctuation of the martingale sequence $M_k$ can be controlled by the event $\mathcal{M}_{n,\sfc}^{(1)}$. We will show that the martingale increment cannot overtake the starting value and the contribution of the drift (particularly the term  $h_{k+1}^{-1}\ev[Z_{2,k}]$). This will lead to the conclusion that the process $\bar p_{\kappa_1}$ hits the upper boundary $(1-\tfrac15\delta)\sqrt{k/m_n}$. 
We will prove by induction on $\ell\ge 1$ that $\bar p_{\kappa_\ell}\ge (1-\tfrac15\delta)\sqrt{\kappa_\ell/m_n}$. This implies that $\bar p_k\ge (1-\delta)\sqrt{k/m_n}$ for $n_0\le k\le n_1$ on $\mathcal{E}_{n,\delta,\sfc}^{(1)}$. From this it follows that 
\begin{align}\label{eq:lower_fluct}
    \lim_{T\to \infty} \liminf_{n\to \infty} \pr (\bar p_k\ge (1-\delta)\sqrt{k/m_n} \text{ for }n_0\le k\le n_1)=1.
\end{align}
Indeed, we can choose $\sfc>0$ large that $\liminf_{n\to \infty} \pr (\mathcal{M}_{n,\sfc}^{(1)})\ge 1-\eps$, and then the previous arguments imply that expression in \eqref{eq:lower_fluct} is at least $1-\eps$. Since this holds for any $\eps>0$, \eqref{eq:lower_fluct} follows.

Now assume that $\ell\ge 1$, $\tau_\ell\le n_1$, and $\bar p_{\tau_\ell}\ge (1-\tfrac13\delta)\sqrt{\tau_\ell/m_n}$. (As we have seen, this holds for $\ell=1$.) We will show that if $\kappa_\ell\le n_1$ then  $\bar p_{\kappa_\ell}\ge (1-\tfrac15\delta)\sqrt{\kappa_\ell/m_n}$. By Proposition \ref{prop:one-step-drop}, this will also imply $\bar p_{\tau_{\ell+1}}\ge (1-\tfrac13\delta)\sqrt{\tau_{\ell+1}/m_n}$, allowing us to continue the induction for $\ell+1$.

By Proposition \ref{prop:moments},  for $\tau_\ell \le k< \kappa_\ell\le n_1$, and for all $n$ large enough, we have
\begin{equation}\label{eq:pbar_drift_asymp}
    \begin{split}
         \frac{h_{k+1}^{-1} \ev[Z_{2,k}]}{1+h_{k}\bar p_k}  &\ge \left(1-\frac{1}{100}\delta\right) \frac{a^{4/3}k}{n(n-k)},\\
   \frac{h_{k}\bar p_k^2}{1+h_{k}\bar p_k}&\le \Big(1-\frac{9}{25}\delta\Big) \frac{a^{4/3}k}{n(n-k)},\\
 \frac{h_{k}h_{k+1}^{-1}(1+\ev[Z_{1,k}])-1}{1+h_{k}\bar p_k}\bar p_k&=O\left(\frac{1}{n-k}+\frac{a^2k}{n(n-k)^2}\right)\sqrt{\frac{k}{m_n}}\ge -\frac{1}{200}\delta \frac{a^{4/3}k}{n(n-k)}.
    \end{split}
\end{equation}
This shows that the sum of the terms on the right side of \eqref{eq:pbar_drift} is bounded from below by $\frac{8}{25}\delta\frac{a^{4/3}k}{n(n-k)}$.
Therefore, by the equations \eqref{eq:pbar_drift}, \eqref{eq:pbar_fluc}, and our induction hypothesis on $\bar p_{\tau_\ell}$, we have 
\begin{align}\label{eq:pbar_comparison}
    \bar p_{\kappa_\ell} &\ge (1-\tfrac13\delta)\sqrt{\frac{\tau_\ell}{m_n}}+ \frac{8}{25}\delta \sum_{k=\tau_{\ell}}^{\kappa_\ell-1} \frac{a^{4/3}j}{n(n-j)} -|M_{\kappa_\ell}-M_{\tau_{\ell}}|.
\end{align}
We will show that the right side of \eqref{eq:pbar_comparison} is at least $(1-\tfrac45 \delta) \sqrt{\kappa_\ell/m_n}$, from this it follows that we must have $\bar p_{\kappa_\ell}\ge (1-\tfrac15\delta)\sqrt{\kappa_\ell/m_n}$. 
Our first step is to show
\begin{align}\label{eq:drift>fluc_Lag1}
(1-\tfrac13\delta)\sqrt{\frac{\tau_\ell}{m_n}} \ge 10\delta^{-1 }\sfc\log (2+x),\qquad x:=\sum_{j=n_0}^{\tau_\ell} h_j.
\end{align}
Indeed, if $n_0\le \tau_\ell\le n/2$,
we have for $T$ large enough
\[
10\delta^{-1}\sfc \log (2+x)\le 10\delta^{-1}\sfc \log \Big(2+\frac{2a^{2/3}\tau_\ell}{n}\Big) \le (1-\tfrac13\delta)\sqrt{\frac{a^{2/3}\tau_\ell}{n}}
\]
since $a^{2/3}\tau_\ell/n\ge m_n^{-1} \lfloor m_n T\rfloor\ge T-1$.
Now assume $n/2 <\tau_\ell <n_1$, then $\tfrac{\tau_\ell}{m_n}\ge \tfrac{1}{2} a^{2/3} $.
We have
\begin{align*}
\log(2+x)=\log \Big(2+\sum_{j=n_0}^{\tau_\ell} \frac{a^{2/3}}{n-j}\Big)\le \log \Big(2+a^{2/3}\log n\Big)\le 
\log a+\log\log n.
\end{align*}
Since $a\gg (\log\log n)^{3}$, we now get 
\[
(1-\tfrac13\delta )\sqrt{\tfrac{\tau_\ell}{m_n}}\ge (1-\tfrac13\delta )\sqrt{\tfrac{a^{2/3}}{2}} \ge 10\delta^{-1}\sfc( \log\log n + \log a)\ge 10\delta^{-1}\sfc\log(2+x),
\]
proving \eqref{eq:drift>fluc_Lag1}.

Next we show 
\begin{align}\label{eq:drift>fluc_Lag2}
 \frac{8}{25}\delta\sum_{j={\tau_\ell}}^{\kappa_\ell-1} \frac{a^{4/3}j}{n(n-j)}\ge 10\delta^{-1 }\sfc\, h \log (2+x),\qquad  h=\sum_{j=\tau_\ell}^{\kappa_\ell-1} h_j.
\end{align}
To this end, we observe that 
\[
\frac{8}{25}\delta\frac{a^{4/3}j}{n(n-j)}h_j^{-1} = \frac{8}{25}\delta\frac{a^{2/3}j}{n}\ge \frac{8}{25}\delta\frac{a^{2/3}\tau_\ell}{n}\ge 10\delta^{-1}\sfc\log(2+x),
\]
where the last inequality follows from \eqref{eq:drift>fluc_Lag1}.  This completes the proof of \eqref{eq:drift>fluc_Lag2}. 

On $\mathcal{M}_{n,\sfc}^{(1)}$ we have 
\begin{align}
    |M_{\kappa_\ell}-M_{\tau_\ell}|\le \sfc(1+h)\log(2+x).
\end{align}
This inequality together with \eqref{eq:drift>fluc_Lag1} and \eqref{eq:drift>fluc_Lag2} implies that 
\begin{align*}
    \bar p(\kappa_\ell) \ge  (1-\tfrac{\delta}{10})\left((1-\tfrac13\delta)\sqrt{\frac{\tau_\ell}{m_n}}+\frac
    {8\delta}{25}\frac{a^{4/3}}n{}\sum_{j=\tau_\ell}^{\kappa_\ell}\frac{j}{n-j}\right).
\end{align*}
To finish our argument, it suffices to show that 
\begin{equation}\label{eq:stopping_comparison}
(1-\tfrac13\delta) \sqrt{\frac{\tau_\ell}{m_n}}+\frac{8\delta}{25}\frac{a^{4/3}}{n}\sum_{j=\tau_\ell}^{\kappa_\ell}\frac{j}{n-j} \ge (1-\tfrac23\delta) \sqrt{\frac{\kappa_\ell}{m_n}}.
\end{equation}
When $\kappa_\ell\le (1+\tfrac23\delta)\tau_\ell$, we have $(1-\tfrac13\delta) \sqrt{\frac{\tau_\ell}{m_n}}\ge (1-\tfrac23\delta) \sqrt{\frac{\kappa_\ell}{m_n}}$ so the statement holds. For $\kappa_\ell> (1+\tfrac23\delta)\tau_\ell\ge (1+\tfrac23\delta)n_0 $, we have 
\begin{align*}
    \frac{8\delta}{25}\frac{a^{4/3}}{n}\sum_{j=\tau_\ell}^{\kappa_\ell}\frac{j}{n-j}&\ge \frac{4\delta}{25}\frac{a^{4/3}}{n^2}\kappa_\ell(\kappa_\ell-\tau_\ell)
    \ge \frac{8\delta^2}{75}\frac{a^{4/3}}{n^2}\kappa_\ell\tau_\ell=\frac{8\delta^2}{75} m_n^{-2} \kappa_\ell\tau_\ell.
\end{align*}
We have $\kappa_\ell \tau_\ell\ge \sqrt{\kappa_\ell} n_0^{3/2}$ and $m_n^{-1} n_0\ge T-1$, hence
\begin{align*}
     \frac{8\delta^2}{75}m_n^{-2} \kappa_\ell\tau_\ell\ge \frac{8\delta^2}{75}(T-1)^{3/2}\sqrt{\kappa_\ell/m_n}\ge (1-\tfrac23\delta)\sqrt{\kappa_\ell/m_n},
\end{align*}
for all $T$ large enough, proving \eqref{eq:stopping_comparison}. This finishes the proof of  $\bar p_{\kappa_\ell}\ge (1-\frac45\delta)\sqrt{\kappa_\ell/m_n}$, which is sufficient to conclude that $\bar p_{\kappa_\ell}\ge (1-\frac15\delta)\sqrt{\kappa_\ell/m_n}$ and the proof of \eqref{eq:lower_fluct}.

For the upper bound, one can analogously construct a sequence of stopping times corresponding to the hitting times of the levels  $(1+\frac14\delta)\sqrt{k/m_n}$ and $(1+\delta)\sqrt{k/m_n}$. Similar to Proposition \ref{prop:one-step-drop}, one can show that $\bar p_{k+1}-\bar p_k$ can not increase too much in a single step. Whenever the process lies between these two values, the drift term \eqref{eq:pbar_drift} is negative and dominates the martingale fluctuation. This in turn implies that with high probability the process $\bar p_k$ will not hit the upper boundary  $(1+\delta)\sqrt{k/m_n}$, completing the proof. 
\end{proof}

The result controls $u/\bar u$, which can be viewed as an extension of Proposition \ref{prop:ratio-to-1}.
\begin{proposition}\label{prop:u/ubar}
Define 
\begin{align}\label{eq:u/ubar1}
    \mathcal{M}_{n,\sfc}^{(2)} :=\left\{\tfrac{u_{k_1}/\bar u_{k_1}}{u_{k_2}/\bar u_{k_2}} \le \exp\Big( \sfc\big(\sum_{j=k_1}^{k_2-1}\tfrac{1}{n-j}+1\big)\log\big(2+\log \tfrac{n-n_0}{n-k_1}\big)\Big),\text{ for all } n_0\le k_1<k_2\le n_1\right\},
\end{align}
then
    \begin{align*}    \lim_{\sfc\to\infty}\liminf_{n\to\infty}\pr(\mathcal{M}_{n,\sfc}^{(2)}\cap \mathcal{A}_n)=1.
    \end{align*}
\end{proposition}
\begin{proof}
Taking the logarithm of the ratio $u/\bar u$, we have 
\begin{align*}
\log \left(\frac{u_{k_1}/\bar u_{k_1}}{u_{k_2}/\bar u_{k_2}}
\right)&=\sum_{j=k_1}^{k_2-1}\wtl W_j+ \sum_{j=k_1}^{k_2-1}\Big(\ev[W_j]+\log (\gamma_j)\Big),
\end{align*}
where we write again $W_j = \log (X_{j+1}/Y_j)$, and $\wtl W_j=W_j-\ev[W_j]$.
By the moment estimates \eqref{eq:Wj_moments} for $W_j$, we have $\sum_{j=k_1}^{k_2-1}(\ev[W_j]+\log \gamma_j)=O(\frac{a}{n-k_2})$, with an absolute constant in the $O(\cdot)$ term. Since $n-k\ge \lfloor a \mathfrak{f}(a) \rfloor$, this term can be further bounded by a constant multiple of $\mathfrak{f}(a)^{-1}$, which goes to $0$ as $n\to \infty$.  On the good event $\mathcal{A}_n$, by the concentrations bounds on $X_j,Y_j$, we  can directly check that 
\[
\wtl W_j \le (n-j)^{-1/2}\log(n-j),\quad n_0\le j\le n_1,
\]
if $n$ is large enough. By the moment estimates \eqref{eq:Wj_moments} this upper bound is much smaller than $\Var(\wtl W_j)^{1/4}$, hence we can apply 
 Lemmas \ref{lem:martingale_fluct} and \ref{lem:martingale_fluc2} to the martingale $\sum_{j=k_1}^{k_2-1}\wtl W_j$ to obtain
\begin{align*}
    \Big|\sum_{j=k_1}^{k_2-1}\wtl W_j\Big|\le \sfc \big(\sum_{j=k_1}^{k_2-1}\tfrac{1}{n-j}+1\big)\log\big(2+\log \tfrac{n-n_0}{n-k_1}\big)\qquad \text{on $\mathcal{A}_n$}.
\end{align*}
Taking exponential and applying Lemma \ref{lem:good_event} completes the proof. 
\end{proof}
Propositions \ref{prop:pbar_square_root} and \ref{prop:u/ubar} together yield the following. Recall the events $\mathcal{E}_{n,\delta,\sfc}^{(1)}$ and $\mathcal{M}_{n,\sfc}^{(2)}$ defined in \eqref{eq:En1} and \eqref{eq:u/ubar1}, respectively.

\begin{proposition}\label{prop:u_ratio}
  For fixed $0<\delta<\frac12$, define 
  \begin{align}\label{eq:En2}
      \mathcal{E}_{n,\delta,\sfc}^{(2)}:= \left\{ \frac{u_{k_1}}{u_{k_2}}\le \exp\Big(-(1-\delta)\sum_{j=k_1}^{k_2-1} h_j\sqrt{\tfrac{j}{{m_n}}}+\sfc\log\big(2+\log\tfrac{n}{n-k_1}\big)\Big), \,\, \text{for all } n_0\le k_1< k_2\le n_1\right\}.
  \end{align}
 Then 
 \[\lim\limits_{\sfc\to\infty}\liminf\limits_{n\to\infty}\pr(\cE_{n,\delta,\sfc}^{(2)})=1.
 \]
\end{proposition}
\begin{proof}
    Fix $0<\delta<\frac12$. On $\cE_{n,\frac{\delta}{2},\sfc}^{(1)}$ by  \eqref{eq:pbar_square_root} we have 
    \[(1-\tfrac12\delta)\sqrt{\tfrac{k}{m_n}}\le \bar p_k\le (1+\tfrac12\delta)\sqrt{\tfrac{k}{m_n}},\qquad \text{for all $n_0\le k\le n_1$}.
    \]
    Then $h_k\bar p_k\le (1+\frac12\delta) \frac{a}{n-k}\le 2\mathfrak{f}(a)^{-1}$, and  we obtain 
    \[
    \frac{\bar u_{k}}{\bar u_{k-1}}= 1+h_k\bar p_k\ge \exp\Big((1-\tfrac23\delta)h_k\sqrt{\tfrac{k}{m_n}}\Big) \quad \text{ for all }n_0\le k\le n_1.
    \]This bound together with Proposition \ref{prop:u/ubar} yields that
    \[\frac{u_{k_1}}{u_{k_2}} 
    \le \exp\left(-(1-\tfrac23\delta)\sum_{j=k_1}^{k_2-1}h_j\sqrt{\tfrac{j}{m_n}}+ \sfc\Big(\sum_{j=k_1}^{k_2-1}\tfrac{1}{n-j}+1\Big)\log\Big(2+\log \tfrac{n-n_0}{n-k}\Big)\right),
    \]
    for all $n_0\le k_1<k_2\le n_1$ on $\cE_{n,\delta,\sfc}^{(1)}\cap \cM_{n,\sfc}^{(2)}$. 
    Since $a\gg (\log\log n)^3$ and $j\ge n_0$, we have 
    \[
    \frac{h_j\sqrt{\frac{j}{m_n}}}{\frac{1}{n-j}}\ge a^{2/3}\sqrt{T-1}\gg (\log\log n)^2\ge  \log\Big(2+\log \frac{n}{n-n_1}\Big).
    \]
    Therefore, for any fixed $\delta$ and $\sfc$, the term $\sfc (\sum_{j=k_1}^{k_2-1}\frac{1}{n-j})\log(2+\log\frac{n-n_0}{n-k})$ can be absorbed by the sum $\sum_{j=k_1}^{k_2-1}h_j\sqrt{j/m_n}$ for all $n$ large enough. By changing the constant $1-\frac23\delta$ to $1-\delta$, we have $\cE_{n,\frac{\delta}{2},\sfc}^{(1)}\cap\cM_{n,\sfc}^{(2)}\subset \cE_{n,\delta,\sfc}^{(2)}$ for all $n$ large enough. The statement now follows from Propositions \ref{prop:pbar_square_root} and \ref{prop:u/ubar}.
\end{proof}

\begin{remark}\label{rem:no_go}
We end this section with a heuristic explanation  why the statement in Proposition \ref{prop:pbar_square_root} might break down when $1\ll a_n\ll \log\log n$. 
When $a$ is too small, the drift terms \eqref{eq:pbar_drift} in the recursion for $\bar p_k$ might not be strong enough to control the fluctuations arising from the terms \eqref{eq:pbar_fluc}. This can also be seen from the continuous analogue given by the SDE \eqref{eq:Airysde}. Suppose that $\mathfrak{p}_f(\cdot)$ solves the SDE 
\[
d\mathfrak{p}_f(t):= (f(t)-\mathfrak{p}_f(t)^2 )dt + dB_t, \quad \mathfrak{p}_f(0)=\infty
\]
where the drift term $f(t)\ll t$. 
Consider now the time-homogeneous diffusion
\[
d\mathfrak{p}_a(t) = (a-\mathfrak{p}_a(t)^2)dt+dB,\qquad \mathfrak{p}_a(0)=\infty.
\]
Let $\gamma_a$ be the first time at which $\mathfrak{p}_a$ explodes to $-\infty$, then
$\ev[\gamma_a] = \frac{\pi}{\sqrt{a}}\exp(\frac{8}{3}a^{3/2})(1+o(1))$
(c.f.~Section 2.2 of \cite{DL}). This shows that when $f^{-1}(a)\gg e^{a^{3/2}}$ the process $\mathfrak{p}_f$ can be bounded in between by the coupled processes $\mathfrak{p}_{a}$ and $\mathfrak{p}_{a+1}$, and may have infinitely many blow-ups. Hence one cannot expect a uniform lower bound for $\mathfrak{p}_f$ or its discrete approximation.
\end{remark}

\subsection{Bounds from concentration}\label{subsec:concentration}
For $k\ge n_1$, the error from the approximation $u_k/\bar u_k$ becomes large. Instead, we will control the ratio $u_{k+1}/u_k$ directly. Recall the recursion \eqref{eq:u_recursion}, we have 
\[
 \frac{u_{k+1}}{u_k}=\frac{X_k^2+Y_k^2-\mu_n}{X_{k+1}Y_k} -\frac{X_k Y_{k-1}}{X_{k+1}Y_{k}} \frac{u_{k-1}}{u_k}.
\]
Observe that if $u_k/u_{k-1}\ge1$, then we have 
\begin{align}\label{eq:u_ratio_lower_bound}
\frac{u_{k+1}}{u_k}-1>\frac{X_k^2+Y_k^2-\mu_n-X_kY_{k-1}-X_{k+1}Y_k}{X_{k+1}Y_k}=:R_k.
\end{align}
Using the concentration bounds on the event $\mathcal{A}_n$, we obtain the following result.
\begin{proposition}\label{prop:u_ratio_concentration}
For $k\ge n_1$ and for all $n$ large enough we have 
\begin{align*}
\left\{\frac{u_k}{u_{k-1}}\ge 1\right\}\cap \mathcal{A}_n \subset \left\{\frac{u_{k+1}}{u_k}>1+\frac13 \frac{a^2 }{(n-k+2a)^{3/2} \sqrt{n-k}}\right\}\cap \mathcal{A}_n.
\end{align*}
\end{proposition}
\begin{proof}
    Suppose $u_k/u_{k-1}\ge 1, k\ge n_1$, then by \eqref{eq:u_ratio_lower_bound} it is sufficient to bound $R_k$ from below. Note that $R_k$ is a function of $X_k, X_{k+1}, Y_{k-1}, Y_k$, and $\mu_n$. 
    By replacing $X_k,X_{k+1}$ with $\sqrt{n-k+2a}$ and $Y_{k-1},Y_{k}$ with $\sqrt{n-k}$ in this function, we get the lower bound
    \begin{align}\label{eq:rk_lower}
    r_k:=\frac{2(n-k)+2a-\mu_n-2\sqrt{n-k}\sqrt{n-k+2a}}{\sqrt{n-k}\sqrt{n-k+2a}}\ge \frac12\frac{a^2}{(n-k+2a)^{3/2}\sqrt{n-k}},
    \end{align} see also the proof of Proposition \ref{prop:moments} in the Appendix for a similar estimate. 

    On $\mathcal{A}_n$, for $n_1\le k\le n-\lfloor\sqrt{a}\rfloor$ we have  
    \[
    |X_k-\sqrt{n-k+2a}|\le \log(n-k+2a),\qquad  |Y_k-\sqrt{n-k}|\le \log(n-k),
    \]
    for all $n$ large.
    In this regime, the difference between $R_k$ and $r_k$ can be bounded with a constant multiple of $\frac{\log(n-k+2a)}{\sqrt{n-k}}$. Since
    \[
    \frac{\log(n-k+2a)}{\sqrt{n-k}}\ll \frac{a^2}{(n-k+2a)^{3/2}\sqrt{n-k}},\qquad n_1\le k\le n-\lfloor\sqrt{a}\rfloor,
    \]
    by replacing the constant $\frac12$ by $\frac13$ in the lower bound \eqref{eq:rk_lower}, the statement for $n_1\le k\le n-\lfloor\sqrt{a}\rfloor$ follows.

In the regime when $ n-\lfloor\sqrt{a}\rfloor<k<n$, on $\mathcal{A}_n$  we have  $(\log a)^{-1/2}\le \frac{Y_k}{\sqrt{n-k}}\le \sqrt{2\log a}$ and $ |X_k-\sqrt{n-k+2a}|\le \log(n-k+2a)$ for $n$ large enough. Hence the difference between $R_k$ and $r_k$ can be bounded by a constant multiple of $\log a$.
Since
\[
 \log a \ll \frac{a^2}{(n-k+2a)^{3/2}\sqrt{n-k}} = O\left(\sqrt{\frac{a}{n-k}}\right),\quad k\ge n-\lfloor\sqrt{a}\rfloor,
\]
we obtain the claimed bound for $u_{k+1}/u_k$ for $k> n-\lfloor\sqrt{a}\rfloor$ as well.
\end{proof}

\subsection{Hilbert-Schmidt norm}\label{subsec:HS_tail}
The goal of this section is to complete the proof of Lemma \ref{lem:HS_tail}. Recall the definition of $\mathsf{K}_n$ and $\mathsf{K}_n^{(T)}$ in \eqref{def:Kn_kernel} and \eqref{def:Kn_T}, respectively. 
We have the upper bound
\begin{align}\label{eq:HS_tail}
    \|\mathsf{K}_n-\mathsf{K}_n^{(T)}\|_{2}^2 
    &\le 2\frac{a^{8/3}}{n^2}\sum_{1\le i\le j\le n_0} ([(\ma{M}_n)^{-1}]_{ij}-[(\ma{M}_n^{(T)})^{-1}]_{ji})^2 +  2\frac{a^{8/3}}{n^2}\sum_{j> n_0,j\ge i} [(\ma{M}_n)^{-1}]_{ij}^2.
\end{align}
By Lemma \ref{lem:tri_inverse} we have
\[
[(\ma{M}_n)^{-1}]_{ij} = u_iv_j\quad \text{for $1\le i\le j\le n$},\qquad [(\ma{M}_n^T)^{-1}]_{ij}  = u_iv_j^{(T)},\quad \text{for $1\le i\le j\le n_0$}.
\]
where $u_j, v_j$ are defined as in \eqref{eq:uv_1} and \eqref{eq:uv_2} associated with $\ma{M}_n$, and $v_j^{(T)}$ is defined as in \eqref{eq:uv_2} associated with $\ma{M}_n^{(T)}$. Recall \eqref{eq:def_vtilde} that $v_j^{(T)}$ can be represented as a linear combination of $w_j$ and $u_j$, where $w_j$ is defined in \eqref{def:w}. 
Since $v_j$ also satisfies the recursion \eqref{eq:vdef2}, by the linearity again one can represent $v_j^{(T)}$ as a linear combination of $v_j$ and $u_j$. By the identity $v_1u_2-u_1v_2=(X_2Y_1)^{-1}$ in \eqref{eq:discrete_Wr} and by comparing the boundary conditions, we have 
\begin{align}\label{eq:v-v^T}
    v_j^{(T)}  = v_j -\frac{v_{n_0+1}}{u_{n_0+1}}u_{j},\qquad 1\le j\le n_0. 
\end{align}
This allows us to rewrite 
\begin{align*}
 2\frac{a^{8/3}}{n^2}\sum_{1\le i\le j\le n_0} ([(\ma{M}_n)^{-1}]_{ij}-[(\ma{M}_n^{(T)})^{-1}]_{ji})^2&=2\frac{a^{8/3}}{n^2}\sum_{1\le i\le j\le n_0}(u_iv_j-u_iv_j^{(T)})^2 \\
 &=  \frac{a^{8/3}}{n^2}v_{n_0+1}^2 u_{n_0+1}^2 \left(\sum_{i\le n_0}\frac{u_i^2}{u_{n_0+1}^2}\right)^2
\end{align*}

Introduce 
\begin{align}
    F_1(i)=F_{1,n}(i) := a^{2/3} v_iu_i,\qquad  F_2(i)=F_{2,n}(i):= \frac{a^{2/3}}{n} \sum_{j\le i}\frac{u_j^2}{u_i^2}.
\label{def:F1F2}
\end{align}
Then  the first term on the right side of \eqref{eq:HS_tail} can be bounded from above by $F_1(n_0+1)^2 F_2(n_0)^2\frac{u_{n_0}^4}{u_{n_0+1}^4}$, and the second term on the right side of \eqref{eq:HS_tail} can be bounded by $2\frac{a^{2/3}}{n}\sum_{i>n_0}F_1(i)^2F_2(i)$.
Hence \eqref{eq:HS_tail} turns into
\begin{align}
     \|\mathsf{K}_n-\mathsf{K}_n^{(T)}\|_{2}^2 
     &\le F_1(n_0+1)^2F_2(n_0)^2\frac{u_{n_0}^4}{u_{n_0+1}^4}+ 2 \frac{a^{2/3}}{n}\sum_{i>n_0} F_1(i)^2F_2(i).\label{eq:kernel_upper_bound}
\end{align}
This shows that in order to estimate $\|\mathsf{K}_n-\mathsf{K}_n^{(T)}\|_2$ we need to control the functions $F_1$ and $F_2$.  By \eqref{def:F1F2}, this can be done by  controlling the ratio $u_{k_1}/u_{k_2},k_1< k_2$.

Let 
\begin{align}\label{eq:Omega1}
    \Omega_{n,T}^{(1)}:=\{F_2(n_0)\le T^{-1/2}\log T\}.
\end{align}
For $\sfc>0$ and $0<\delta<\frac12$ we recall the event $\cE_{n,\delta,\sfc}^{(2)}$ defined in \eqref{eq:En2}. Note that for each $\sfc>0$, we can find $\sfc_1=\sfc_1(\sfc)>0$ such that $\sfc\log(2+x)\le \sfc_1+\frac{1}{10} x$. For $n_0\le k_1<k_2\le n_1$, we define
\begin{align}\label{eq:Omega2}
    \Omega_{n,\sfc}^{(2)}:= \left\{ \frac{u_{k_1}}{u_{k_2}}\le \exp\Big(\sfc_1-\tfrac34\sum_{j=k_1}^{k_2-1} h_j\sqrt{\tfrac{j}{{m_n}}}+\tfrac{1}{10}\log\tfrac{n}{n-k_1}\Big), \,\, \forall n_0\le k_1< k_2\le n_1\right\}.
\end{align}
Note that by setting $\delta=\frac14$ in the definition of $\cE_{n,\delta,\sfc}^{(2)}$, we get $\cE_{n,\frac14,\sfc}^{(2)}\subset \Omega_{n,\sfc}^{(2)}$.
Finally, for $n_1\le k<n$ we define
\begin{align}\label{eq:Omega3}
    \Omega_{n}^{(3)} = \left\{\frac{u_{k+1}}{u_k}\ge 1+ \frac{1}{3}\frac{a^2}{(n-k+2a)^{3/2}\sqrt{n-k}},\quad \forall n_1\le k<n\right\}.
\end{align} 

\begin{proposition}\label{prop:Omega}
    Let $\Omega_{n,T,\sfc}:=\Omega_{n,T}^{(1)} \, \cap \Omega_{n,\sfc}^{(2)} \cap \, \Omega_{n}^{(3)}\cap \,\mathcal{A}_n$, then 
    \[
    \lim_{\sfc\to\infty}\lim_{T\to\infty}\liminf_{n\to\infty}\pr(\Omega_{n,T,\sfc})=1.
    \]
\end{proposition}
\begin{proof}
    By Propositions \ref{prop:ubar_conv} and \ref{prop:ratio-to-1}, we have $ u_{\lfloor xm_n\rfloor}/u_{n_0}\Rightarrow \psi_d(x)/{\psi_d(T)}$ in law  on $[0,T]$ with respect to the Skorokhod topology. This implies  $F_2(n_0)\Rightarrow\int_0^T \psi_d(x)^2\psi_d(T)^{-2}dx$.   
    By
    Lemma \ref{lem:Airy_asymp}, there exists a.s.~finite $C$ such that  
    \begin{align*}
    \int_0^T \frac{\psi_d(x)^2}{\psi_d(T)^2}dx  \le CT^{-1/2},
    \end{align*}
    see e.g.~(2.11) of \cite{DLV}. Hence $\lim\limits_{T\to\infty}\lim\limits_{n\to\infty}\pr(\Omega_{n,T}^{(1)})=1$.  
    
    Next, by applying Proposition \ref{prop:u_ratio} with $\delta=1/4$ and using that $\cE_{n,\frac14,\sfc}^{(2)}\subset \Omega_{n,\sfc}^{(2)}$, we get that $\lim_{\sfc\to\infty}\liminf_{n\to\infty}\pr(\Omega_{n,\sfc}^{(2)})=1$. Note that on $\Omega_{n,\sfc}^{(2)}\cap\mathcal{A}_n$ we also have $u_{n_1}\ge u_{n_1-1}$, which together with Proposition \ref{prop:u_ratio_concentration} shows that $\Omega_{n,\sfc}^{(2)}\cap \mathcal{A}_n\subset \Omega_{n,T}^{(3)}\cap\mathcal{A}_n$. The statement then follows from  simple union bounds over the complements of these events, together with the estimates in Propositions \ref{prop:u_ratio} and \ref{prop:u_ratio_concentration}.
\end{proof}

Using these bounds, we are able to control the functions $F_1$ and $F_2$ on $\Omega_{n,T,\sfc}$.
\begin{proposition}\label{prop:F1_bound}
    On the event $\Omega_{n,T,\sfc}$, when $n$ is large enough, we have 
    \begin{equation}\label{eq:F1_bound}
    \begin{split}
        F_1(i) \lesssim_\sfc \begin{cases}
            \sqrt{\frac{n}{ia^{2/3}}}+a^{-1/3}(\log a)^{5/2}\quad &\mbox{for $n_0\le i< \lfloor\frac{n}{2}\rfloor
            $,}\\
            a^{-1/3}(\log a)^{5/2} (\frac{n}{n-i})^{1/10}\quad &\mbox{for $\lfloor\frac{n}{2}\rfloor+1\le i< n_1$,}\\
            a^{-1/3}(\log a)^{5/2}\quad &\mbox{for $n_1\le i\le n$.}
        \end{cases}
    \end{split}
\end{equation}
Here $\lesssim_\sfc$ means that the respective inequality holds with  a $\sfc$-dependent multiplier, which is independent of $i, n,T$. 
\end{proposition}

\begin{proof}
First note that by the Wronskian identity \eqref{eq:uv_Wr}, we have
\[
F_1(i)=a^{2/3}\left(\sum_{\ell=i}^{n-1}\frac{u_{i}^2}{u_{\ell}u_{\ell+1} e_\ell}+ \frac{u_i^2}{u_n^2}\frac{1}{d_n-e_{n-1}\frac{u_{n-1}}{u_n}}\right),\qquad d_n = X_n^2-\mu_n, \quad e_\ell = X_{\ell+1}Y_{\ell}.
\]
We will work on the event $\Omega_{n,T,\sfc}$. We have the following concentration bounds for all $n$ large enough:
\begin{align}
    e_{\ell} &= \frac{1}{X_{\ell+1}Y_\ell}\le \begin{cases}
        \frac{2}{\sqrt{n-\ell}\sqrt{n-\ell+2a}} \quad &\mbox{$\ell\le n-\lfloor\sqrt{a}\rfloor$,} \\
         \frac{2\log a}{\sqrt{n-\ell}\sqrt{a}} \quad &\mbox{$\ell>n-\lfloor\sqrt{a}\rfloor$,}
    \end{cases}\label{eq:el}\\
    d_n & = X_n^2 -\mu_n  \ge 0.99a. \label{eq:dn}
\end{align}
For $n_1\le i<n$, by \eqref{eq:el} and \eqref{eq:dn}, we have
\begin{align}\label{eq:F1_first}
    F_1(i) \le  2a^{2/3}\sum_{\ell=i}^{n-1}\frac{u_i^2\log a}{u_\ell u_{\ell+1}\sqrt{n-\ell}\sqrt{a}} + 2a^{-1/3}\frac{u_i^2}{u_n^2}\le 2a^{2/3}\sum_{\ell=i}^{n-1}\frac{u_i\log a}{u_{\ell+1}\sqrt{n-\ell}\sqrt{a}} + 2a^{-1/3},
\end{align}
since we are on $\Omega_{n,T,\sfc}\subset \Omega_{n,T}^{(3)}$.

On $\Omega_{n,T}^{(3)}$, we have
\begin{align}\label{eq:qk}
\frac{u_{k+1}}{u_k}\ge q_k^{-1}:=\begin{cases}
    1+ \tfrac{1}{20}(\frac{a}{n-k})^2,\quad &\mbox{if $n_1\le k\le n-\lfloor a\rfloor$},\\
    1+\tfrac{1}{20}\sqrt{\frac{a}{n-k}},\quad &\mbox{if $n-\lfloor a\rfloor\le k<n$},
\end{cases}
\end{align}
and
\[
\frac{u_i}{u_{\ell+1}}\le \prod_{j=i}^\ell q_j\le q_i^{\ell-i+1},\qquad q_i<1.
\]
We claim that
\begin{align}\label{eq:geometric}
    \sum_{\ell=i}^{n-1}q_i^{\ell-i+1}\frac{1}{\sqrt{n-\ell}} \le 3q_i\frac{1}{\sqrt{n-i}(1-q_i)},\qquad \text{for $n_1\le i<n$.}
\end{align}
Indeed, for $i\le \ell\le n-1-\lfloor \frac{n-i}{2}\rfloor$ we have $1/\sqrt{n-\ell}\le \sqrt{2/(n-i)}$, and for  $n-\lfloor \frac{n-i}{2}\rfloor\le \ell \le n-1$ we bound $q^{\ell -i} \le q^{\lfloor \frac{n-i}{2}\rfloor}$. This gives that  
\begin{align*}
     \sum_{\ell=i}^{n-1}q_i^{\ell-i+1}\frac{1}{\sqrt{n-\ell}}&\le \frac{\sqrt{2}q_i}{\sqrt{n-i}}\sum_{\ell =i}^{n-\lfloor \frac{n-i}{2}\rfloor-1} q_i^{\ell-i} + q_i^{\lfloor\frac{n-i}{2}\rfloor+1}\sum_{\ell = n-\lfloor \frac{n-i}{2}\rfloor}^{n-1}\frac{1}{\sqrt{n-i}}\\
     &\le  \frac{\sqrt{2}q_i}{(1-q_i)\sqrt{n-i}}+2q_i^{\lfloor\frac{n-i}{2}\rfloor+1}\sqrt{n-i}\\
     &\le 
      \frac{q_i}{(1-q_i)\sqrt{n-i}}\Big(\sqrt{2}+  2q_i^{\lfloor\frac{n-i}{2}\rfloor}(n-i)(1-q_i)\Big).
\end{align*}
Using the bound \eqref{eq:qk}, and the simple facts $\log x\le -(1-x)$ and $xe^{-x}\le e^{-1}$ for $x\ge 0$, we have $q_i^{\lfloor\frac{n-i}{2}\rfloor}(n-i)(1-q_i)\le 2e^{-1}$. The claim follows from the inequality $\sqrt{2}+4e^{-1}\le 3$.

This bound \eqref{eq:geometric} turns \eqref{eq:F1_first} into
\begin{align*}
F_1(i)&\le 6a^{2/3}\frac{\log a}{\sqrt{a}} q_i\frac{1}{\sqrt{n-i}(1-q_i)}+2a^{-1/3}
\le 6a^{-1/3}(\log a)^{5/2},\quad i\ge n_1.
\end{align*}
Next we rewrite the bound in the event $\Omega_{n,\sfc}^{(2)}$ defined in \eqref{eq:Omega2}. 
Note that if $k_1\le \lfloor \frac{n}{2}\rfloor$ then $\log\frac{n}{n-k_1}\le \log 2$, and if $k_1\ge \lfloor \frac{n}{2}\rfloor+1$, then $\sqrt{k_1/n}\ge \sqrt{1/2}$.
Thus, there exists an absolute constant $c_0>0$ such that on $\Omega_{n,\sfc}^{(2)}$ we have 
\begin{align}\label{eq:mart_u_bound2}
     \frac{u_{k_1}}{u_{k_2}} &\lesssim_\sfc \begin{cases}
       \exp\big(-c_0\sum_{j=k_1}^{k_2} \frac{a}{n}\sqrt{\frac{j}{n}}\big), &\mbox{if $n_0\le k_1\le \lfloor\frac{n}{2}\rfloor$},\\
        (\frac{n-k_2}{n-k_1})^{ac_0}(\frac{n}{n-k_1})^{1/10}, &\mbox{if $\lfloor\frac{n}{2}\rfloor+1\le k_1\le n_1$.}
     \end{cases}
\end{align}
By the second inequality of \eqref{eq:mart_u_bound2}, for $\lfloor \frac{n}{2}\rfloor+1\le i\le n_1$ we have
\begin{align*}
    F_1(i)
    &\le2 a^{2/3}\sum_{\ell=i}^{n_1} \frac{u_{i}^2}{u_{\ell}u_{\ell+1}}\frac{1}{n-\ell} + \frac{u_i^2}{u_{n_1}^2} F_1(n_1)\\
    &\lesssim_\sfc a^{2/3} \sum_{\ell=i}^{n_1} \Big(\frac{n-\ell}{n-i}\Big)^{ac_0}\Big(\frac{n}{n-i}\Big)^{1/10}\frac{1}{n-\ell}
    +\Big(\frac{n-n_1}{n-i}\Big)^{ac_0}\Big(\frac{n}{n-i}\Big)^{1/10}a^{-1/3} (\log a)^{5/2}\\
    &\lesssim_\sfc a^{-1/3}\Big(\frac{n}{n-i}\Big)^{1/10} + \Big(\frac{n}{n-i}\Big)^{1/10}a^{-1/3} (\log a)^{5/2}\lesssim_\sfc a^{-1/3}(\log a)^{5/2}\Big(\frac{n}{n-i}\Big)^{1/10}.
\end{align*}
For $n_0\le i\le \lfloor\frac{n}{2}\rfloor$, by the first inequality of \eqref{eq:mart_u_bound2} we have
\begin{align*}
    F_1(i) &\le 2 a^{2/3}\sum_{\ell=i}^{\lfloor\frac{n}{2}\rfloor} \frac{u_i^2}{u_\ell u_{\ell+1}}\frac{1}{n-\ell} + \frac{u_i^2}{u_{\lfloor \frac{n}{2}\rfloor+1}^2}F_1(\lfloor \tfrac{n}{2}\rfloor+1)+\frac{u_i^2}{u_{n_1}^2}F_1(n_1)\\
    &\lesssim_\sfc \frac{a^{2/3}}{n}\sum_{\ell=i}^{\lfloor\frac{n}{2}\rfloor}\exp\left(-c_0\sum_{j=i}^\ell \frac{a}{n}\sqrt{\frac{j}{n}}\right) + a^{-1/3}(\log a)^{5/2},
    \\
    &\lesssim_\sfc \sqrt{\frac{n}{ia^{2/3}}} + a^{-1/3}(\log a)^{5/2}.
\end{align*}
In the last step we used the  Riemann integral approximation with mesh $m_n^{-1}=a^{2/3}n^{-1}$:
\begin{align}\label{eq:Riemann}
    m_n^{-1}\sum_{\ell=\lfloor xm_n\rfloor}^{\lfloor \frac{n}{2}\rfloor} \exp\left(-c_0m_n^{-1}\sum_{j=\lfloor xm_n\rfloor}^{\ell}\sqrt{\tfrac{j}{m_n}}\right)\to \int_x^{\infty} \exp\left(-c_0\int_x^z \sqrt{u}du\right) dz\lesssim x^{-1/2}.
\end{align}
This completes the proof of the statement.
\end{proof}

\begin{proposition}\label{prop:F2_bound} Recall the definition of $F_2$ from \eqref{def:F1F2}.
   On $\Omega_{n,T,\sfc}$, for all $n$ large enough we have
    \begin{equation}\label{eq:F2_bound}
    \begin{split}
        F_2(i) \lesssim_\sfc \begin{cases}
           \sqrt{\frac{n}{ia^{2/3}}}\log T\quad &\mbox{for $n_0\le i\le \lfloor\frac{n}{2}\rfloor$,}\\
            a^{-1/3} (\frac{n-i}{n})^{9/10}\quad &\mbox{for $\lfloor \frac{n}{2}\rfloor
            +1\le i< n_1$,}\\
           a^{-1/3}(\frac{a\mathfrak{f}(a)}{n})^{9/10}\quad &\mbox{for $n_1\le i\le n$},
        \end{cases}
    \end{split}
\end{equation}
where $\lesssim_\sfc$ means that the inequality holds with  a $\sfc$-dependent multiplier, independent of $i, n,T$.
\end{proposition}
\begin{proof}
    For $i=n_0$, by the definition \eqref{eq:Omega1} and $\Omega_{n,T,\sfc}\subset \Omega_{n,T}^{(1)}$ we have $F_2(n_0)\le T^{-1/2}\log T$. For $n_0< i\le \lfloor \frac{n}{2}\rfloor$, by  the first inequality of \eqref{eq:mart_u_bound2} we have 
\begin{align*}
    F_2(i)&=  F_2(n_0)  \frac{u_{n_0}^2}{u_i^2}+\frac{a^{2/3}}{n}\sum_{j=n_0}^i \frac{u_j^2}{u_i^2}\\
    &\lesssim_\sfc T^{-1/2}\log T\cdot \exp\left(-c_0 \sum_{\ell=n_0}^i m_n^{-1}\sqrt{\tfrac{\ell}{m_n}}\right)+\frac{a^{2/3}}{n}\sum_{j=n_0}^i \exp\left(-c_0 \sum_{\ell=j}^i m_n^{-1}\sqrt{\tfrac{\ell}{m_n}}\right).
\end{align*}
Similar to the Riemann integral approximation \eqref{eq:Riemann}, for $j\le i$ we have 
\begin{align}
    \exp\left(-c_0\sum_{\ell=j}^i m_n^{-1}\sqrt{\tfrac{\ell} {m_n}}\right) \le 2 \exp\left(-c_0 \int_{x(j)}^{x(i)} \sqrt{y} dy\right)\le 2\exp\left(-\tfrac23c_0\sqrt{x(i)} (x(i)-x(j)\right),
\end{align}
where $x(i):= i/m_n$. Therefore, for $n_0\le i\le \lfloor \frac{n}{2}\rfloor$ we get
\begin{align*}
    F_2(i) &\lesssim_\sfc T^{-1/2}\log T \exp\Big(-\tfrac23c_0 \sqrt{x(i)} (x(i)-T)\Big)+\int_T^{x(i)} e^{-\tfrac23c_0\sqrt{x(i)}(x(i)-s)} ds\\
    &\lesssim_\sfc T^{-1/2}\log T \exp\Big(-\tfrac23c_0 \sqrt{x(i)} (x(i)-T)\Big)+(x(i))^{-1/2}\log T\lesssim_\sfc x(i)^{-1/2}\log T.
\end{align*}
For $\lfloor\frac{n}{2}\rfloor+1 \le i\le n_1$, by the second inequality of \eqref{eq:mart_u_bound2} we have 
\begin{align*}
    F_2(i)&= \frac{u_{\lfloor \frac{n}{2}\rfloor}^2}{u_i^2} F_2(\lfloor\tfrac{n}{2}\rfloor)+\frac{a^{2/3}}{n} \sum_{j=\lfloor
    \frac{n}{2}\rfloor+1}^i \frac{u_j^2}{u_i^2}\\
    &\lesssim_\sfc a^{-1/3} \log T\, \Big(\frac{n-i}{n}\Big)^{ac_0}  + \frac{a^{2/3}}{n} \sum_{j=\lfloor \frac{n}{2}\rfloor+1}^i \Big(\frac{n-i}{n-j}\Big)^{ac_0}\Big(\frac{n}{n-j}\Big)^{1/10}
    \lesssim_\sfc  a^{-1/3} \Big(\frac{n-i}{n}\Big)^{9/10}.
\end{align*}
Finally, for $n_1\le i\le n$, on $\Omega_{n,T,\sfc}\subset \Omega_{n,T}^{(3)}$ we recall the bound \eqref{eq:qk} on $u_{k+1}/u_{k}$. Hence
\begin{align*}
    F_2(i)&=  \frac{u_{n_1}^2}{u_i^2} F_2(n_1)+\frac{a^{2/3}}{n} \sum_{j=n_1}^i \frac{u_j^2}{u_i^2}\lesssim_\sfc a^{-1/3} \Big(\frac{a\mathfrak{f}(a)}{n}\Big)^{9/10}+ \frac{a^{2/3}}{n}\mathfrak{f}(a)^2\lesssim_\sfc  a^{-1/3} \Big(\frac{a\mathfrak{f}(a)}{n}\Big)^{9/10},
\end{align*}
where last step follows from the definition of $\mathfrak{f}(a)=\log(\min\{a,n/a\})$. 
This completes the proof.
\end{proof}

We now have all the ingredients to prove Lemma \ref{lem:HS_tail}.
\begin{proof}[Proof of Lemma \ref{lem:HS_tail}]
The goal is to show that $\lim\limits_{T\to\infty}\limsup\limits_{n\to\infty}\pr(\|\mathsf{K}_n-\mathsf{K}_n^{(T)}\|_2\ge \eps)=0$ for any fixed $\eps>0$. To this end, we will show that under $\Omega_{n,T,\sfc}$, the norm difference $\|\mathsf{K}_n-\mathsf{K}_n^{(T)}\|_2$ can be upper bounded by an error which goes to 0 after sending $n\to\infty$ and then $T\to\infty$. The proof then follows from Proposition \ref{prop:Omega}.

We will work on $\Omega_{n,T,\sfc}$. Recall the bound \eqref{eq:kernel_upper_bound} on $\|\mathsf{K}_n-\mathsf{K}_n^{(T)}\|_2^2$. 
The first term on the right side of \eqref{eq:kernel_upper_bound} can be estimated on $\Omega_{n,T,\sfc}$ using $F_1(n_0+1)^2F_2(n_0)^2\lesssim_\sfc T^{-2}(\log T)^2$ and $u_{n_0}/u_{n_0+1}\lesssim_\sfc 1$.  

For the second term on the right side of \eqref{eq:kernel_upper_bound}, we cut the sum into three parts
\begin{align}\label{eq:HS_split}
    \frac{a^{2/3}}{n}\sum_{i=n_0}^n F_1(i)^2 F_2(i)&=\frac{a^{2/3}}{n}\left(\sum_{i=n_0}^{\lfloor
    \frac{n}{2}\rfloor}+\sum_{i=\lfloor
    \frac{n}{2}\rfloor+1}^{n_1-1}+\sum_{i=n_1}^{n}\right) F_1(i)^2 F_2(i).
\end{align}
By the bounds in Propositions \ref{prop:F1_bound} and \ref{prop:F2_bound}, the first sum on the right side of \eqref{eq:HS_split} can be bounded using Riemann approximation as
\begin{align*}
\frac{a^{2/3}}{n}\log T\sum_{i=n_0}^{\lfloor
    \frac{n}{2}\rfloor}\Big(\frac{n}{ia^{2/3}}+\frac{(\log a)^5}{a^{2/3}}\Big)\sqrt{\frac{n}{ia^{2/3}}}&\lesssim \log T\int_{T}^{a^{2/3}} \frac{1}{y^{3/2}}dy + \frac{(\log a)^5}{a^{2/3}}\log T\int_T^{a^{2/3}}\frac{1}{y^{1/2}}dy\\
&\lesssim T^{-1/2} \log T+ a^{-1/3}(\log a)^5\log T.
\end{align*}
Here and in the rest of the proof $\lesssim$ means that the inequality holds with an extra absolute constant multiplier on the right.
For the second sum on the right side of \eqref{eq:HS_split}, we have
\begin{align*}
    \frac{a^{2/3}}{n}\sum_{i=\lfloor
    \frac{n}{2}\rfloor+1}^{n_1-1} F_1(i)^2 F_2(i)
    &\lesssim_\sfc   \frac{a^{-1/3} (\log a)^5}{n} \sum_{i=\lfloor
    \frac{n}{2}\rfloor+1}^{n_1-1} \Big(\frac{n-i}{n}\Big)^{7/10}\lesssim  a^{-1/3} (\log a)^5.
\end{align*}
For the last term on the right side of \eqref{eq:HS_split}, we have 
\begin{align*}
     \frac{a^{2/3}}{n}\sum_{i=n_1}^{n} F_1(i)^2 F_2(i)&\lesssim_\sfc \frac{a^{-1/3}(\log a)^5}{n}\sum_{i=n_1}^{n} \Big(\frac{a\mathfrak{f}(a)}{n}\Big)^{9/10} 
     \lesssim {a^{-1/3} (\log a)^5} \Big(\frac{a\mathfrak{f}(a)}{n}\Big)^{19/10}.
\end{align*}
From the assumption $a\mathfrak{f}(a)\ll n$, we now obtain
\[
   \|\mathsf{K}_n-\mathsf{K}_n^{(T)}\|_{2}^2\lesssim_\sfc T^{-1/2}\log T +a^{-1/3}(\log a)^5\log T.
\]
Sending $n\to\infty$ (hence $a\to\infty$), and then $T\to\infty$ completes the proof. 
\end{proof}

\section{Proof of Theorem \ref{thm:main_lower} in the case when $a_n\le (\log n)^{1/2}$}\label{sec:small_a}

The goal of this section is to prove Theorem \ref{thm:main_lower} in the regime when $a_n$ grows slowly. In all of the subsequent statements, we assume $\beta>0$ is fixed and $1\ll a_n\le (\log n)^{1/2}$ unless otherwise specified.

\subsection{Outline of the proof}

Recall the definition of $\mu_n$ from \eqref{def:mun}. Because of the bound \eqref{eq:mu_bound} and $a_n\le (\log n)^{1/2}$,
we have $a_n^{-4/3}(n\mu_n-a_n^2)\to 0$ as $n\to \infty$. Hence to prove  \eqref{HTS_lim}
it is sufficient to show
\begin{align}\label{eq:HTS_small_a}
   a_n^{-4/3}\big(n \Lambda_{n,\beta,2a_n}-a_n^2\big) \Rightarrow \Airyb,\qquad \Lambda_{n,\beta,2a_n}\sim\text{Laguerre$_{n,\beta,2a_n}$}.
\end{align}
Recall Theorem \ref{thm:hardtosoft} and the hard-to-soft edge transition \eqref{eq:HtoS_process} for the limiting processes. Theorem \ref{thm:main_lower} would follow if we can show that $n\Lambda_{n,\beta,2a_n}$ is  ``sufficiently close'' to $\Bessel_{\beta,2a_n}$. The main result of this section is the following.

\begin{proposition}\label{prop:hardedge_growing_a} 
Assume $1\ll a_n\le (\log n)^{1/2}$. There exists a coupling between $\Lambda_{n,\beta,2a_n}=\{\lambda_{1,2a_n}^{(n)}<\lambda_{2,2a_n}^{(n)}<\cdots<\lambda_{n,2a_n}^{(n)}\}$ and $\Bessel_{\beta,2a_n}=\{\lambda_{1,2a_n}^B<\lambda_{2,2a_n}^B<\cdots\}$ so that for any fixed $k$ and $0<\alpha<1$ we have 
\begin{align}\label{eq:hardedge_growing_a}
\lim_{n\to\infty}\pr\left(\sum_{j=1}^k |\lambda_{j,2a_n}^B-n\lambda_{j,2a_n}^{(n)}|^2\ge n^{-1+\alpha}\right) =0.
\end{align}    
\end{proposition}

Recall Theorem \ref{thm:hard}, when $a_n\equiv a>-1$ is fixed Ram\'irez and Rider \cite{RR} proved that  $n\Lambda_{n,\beta,a}\Rightarrow \Bessel_{\beta,a}$ as $n\to\infty$. In this sense, Proposition \ref{prop:hardedge_growing_a} can be viewed as a generalization of Theorem \ref{thm:hard} in the (slowly) growing regime. Our result also provides a quantitative bound on the rate of convergence.

Let $\ma{L}_{n}=\ma{L}_{n,\beta,2a_n}$ be the Dumitriu-Edelman bidiagonal matrix defined in \eqref{eq:DE-bidiagonal} with entries defined in \eqref{eq:XYdef}.
Following \cite{RR} and Lemma \ref{lem:inv_bidiagonal}, one can view $(\sqrt{n}\ma{L}_{n,\beta,2a_n})^{-1}$ 
as an integral operator $\mathtt{K}^{(n)}_{\beta,2a_n}$ on $L^2[0,1]$ with kernel
\begin{align}\label{eq:kernel_Kn}
    \mathsf{k}^{(n)}_{\beta,2a_n} (x,y) =  \frac{\sqrt{n}}{X_i}\exp\left\{\sum_{k=i}^{j-1}\log \Big(\frac{Y_k}{X_{k+1}}\Big)\right\}\ind_L(x,y),
\end{align}
where 
\[
\ind_L(x,y) = \ind_{[\frac{i-1}{n},\frac{i}{n})}(x)\ind_{[\frac{j-1}{n},\frac{j}{n})}(y),\qquad i\le j.
\]

Define the operator $\mathtt{K}_{\beta,2a_n}$ as an integral operator on $L^2[0,1]$ with kernel
\begin{align}\label{eq:kernel_K}
    \mathsf{k}_{\beta,2a_n} (x,y) =(1-x)^{-a_n-\frac12}\exp\left\{\int_x^y \frac{dB_z}{\sqrt{\beta(1-z)}}\right\}(1-y)^{a_n}\ind_{x\le y}(x,y),
\end{align}
where $B_t$ is  standard Brownian motion. 
The stochastic Bessel operator $\mathfrak{G}_{\beta,2a_n}$ can be obtained from $\mathtt{K}_{\beta,2a_n}^\top \mathtt{K}_{\beta,2a_n}$ under the change of variables $(x,y)\mapsto (1-e^{-x},1-e^{-y})$.
Lemma 6 of \cite{RR} proved that for fixed $a_n\equiv a$, there exists a coupling of $\mathsf{k}_{\beta,a}^{(n)},n\ge 1$ and $\mathsf{k}_{\beta,a}$ so that any sequence of $\mathsf{k}_{\beta,a}^{(n)}$ contains a further subsequence $\mathsf{k}_{\beta,a}^{(n')}$ which converges to $\mathsf{k}_{\beta,a}$ a.s.~in Hilbert-Schmidt norm.
Using the coupling techniques developed in \cite{BVBV_19}, our next result provides an explicit coupling of $\mathsf{k}_{\beta,2a_n}^{(n)}$ and $\mathsf{k}_{\beta,2a_n}$ so that $\|\mathsf{k}_{\beta,2a_n}^{(n)}-\mathsf{k}_{\beta,2a_n}\|_2$ can be controlled quantitatively when $a_n\le (\log n)^{1/2}$.

\begin{proposition}\label{prop:HS_kernel_smalla}
    Assume $1\ll a_n\le (\log n)^{1/2}$. The kernels $\mathsf{k}_{\beta,2a_n}^{(n)}$ and $\mathsf{k}_{\beta,2a_n}$ are a.s.~Hilbert-Schmidt.
    Moreover, there exists a coupling of ${\mathsf{k}}^{(n)}_{\beta,2a_n}$ and $\mathsf{k}_{\beta,2a_n}$ such that for any fixed $0<\alpha<1$, we have 
    \begin{align}
        \pr\big(\|\mathsf{k}^{(n)}_{\beta,2a_n} -  \mathsf{k}_{\beta,2a_n}\|^2_2\ge n^{-1+\alpha}\big)\le \frac{(\log n)^7}{a_n n^{\alpha}}.
    \end{align}
\end{proposition}
The proof of Proposition \ref{prop:HS_kernel_smalla} will be presented in Sections \ref{subsec:HS_coupling_idea} -- \ref{subsec:coupling}.

We also need the following fact from functional analysis (see e.g.~Section 9.2 in \cite{conway2013course}).
\begin{fact}\label{fact:HS_squared}
    Let  $\mathtt{K}_{1}$, $\mathtt{K}_{2}$ be Hilbert-Schmidt operators on a real Hilbert space $\mathcal{H}$, then we have 
    \begin{align*}
    \|\mathtt{K}_{1}^{\top}\mathtt{K}_{1}-\mathtt{K}_{2}^{\top}\mathtt{K}_{2}\|_{\textup{HS}}\le \|\mathtt{K}_1-\mathtt{K}_2\|_{\textup{HS}}(\|\mathtt{K}_1\|_{\textup{HS}}+\|\mathtt{K}_2\|_{\textup{HS}}).
    \end{align*}
\end{fact}

Then next proposition gives a (sub-optimal) growth bound on the points of $\Bessel_{\beta,a}$.
\begin{proposition}\label{prop:Bessel_growth}
Let $\Bessel_{\beta,a}=\{\lambda_{1,a}^B< \lambda_{2,a}^B< \cdots\}$, then we have 
\begin{align*}
\pr\Big(\lambda_{\lfloor a\rfloor,a}^B\le a^3\Big)\ge 1- O\big(a^{-3}\log (a+3)\big).
\end{align*}
The constant in the $O(\cdot)$ term only depends on $\beta$.
\end{proposition}
\begin{proof}
We can assume $a$ is sufficiently large by changing the constant in the $O(\cdot)$ term. Let $M_a(\cdot)$ denote the counting function of the positive square root of the $\Bessel_{\beta,a}$ process. Theorem 1.4 of \cite{Holcomb_2018} gives that 
    \[
    M_a(\lambda) \ed \lim_{t\to\infty} \lfloor \tfrac{1}{4\pi}\varphi_{a,\lambda}(t)\rfloor,
    \]
    where $\varphi_{a,\lambda}$ is the solution to the SDE
    \begin{align*}
        d\varphi_{a,\lambda}=\frac{\beta}{2}\big(a+\frac{1}{2}\big)\sin\big(\frac{\varphi_{a,\lambda}}{2}\big)dt+\beta\lambda e^{-\beta t/8}dt+\frac{\sin \varphi_{a,\lambda}}{2}dt+2\sin\big(\frac{\varphi_{a,\lambda}}{2}\big)dB_t,
    \end{align*}
    with initial condition  $\varphi_{a,\lambda}(0)=0$.

With $\lambda_a=a^{3/2}$ and $t_a=\frac{8}{\beta} \log \lambda_a$ we get
        \begin{align}\label{eq:counting_phi}
        \varphi_{a,\lambda_a}(t_a)=8(a^{\frac{3}{2}}-1)+\int_0^{t_a}\tfrac{\sin\varphi_{a,\lambda_a}}{2}dt+\tfrac{\beta}{2}\big(a+\tfrac{1}{2}\big)\int_0^{t_a}\sin \big(\tfrac{\varphi_{a,\lambda_a}}{2}\big)dt+2\int_0^{t_a}\sin\big(\tfrac{\varphi_{a,\lambda_a}}{2}\big)dB_t.
    \end{align}
We can bound the sum of the first two integrals in \eqref{eq:counting_phi} with  a constant multiple of $a\log a$. The third integral is mean zero with variance at most $4{t_a}=c_\beta \log a$. 

\cite{Holcomb_2018} also shows that the process $t \to \lfloor \frac{1}{4\pi}\varphi_{a,\lambda}(t)\rfloor$ is non-decreasing.
Therefore we can use $\varphi_{a,\lambda_a}({t_a})$ to estimate $M_a(\lambda_a)=M_a(a^{3/2})$ as follows. 
\[
\pr(M_a(a^{3/2})< \lfloor a\rfloor)\le \pr(\varphi_{a,\lambda_a}({t_a})\le 8\pi a)\le 
\pr(|\varphi_{a,\lambda_a}({t_a})-8 a^{3/2}|\ge  a^{3/2})\le  c_\beta \log(a) a^{-3}.
\]
This implies  $\pr(\lambda_{\lfloor a\rfloor ,a}^B\le a^{3})\ge 1-O(a^{-3}\log a)$, finishing the proof.
\end{proof}

We are now ready to prove Proposition \ref{prop:hardedge_growing_a}.
\begin{proof}[Proof of Proposition \ref{prop:hardedge_growing_a}]
We first estimate the Hilbert-Schmidt norm  of the difference of the operators $\mathtt{J}_1:=(\mathtt{K}^{(n)}_{\beta,2a_n})^\top \mathtt{K}^{(n)}_{\beta,2a_n}$ and $\mathtt{J}_2:=\mathtt{K}_{\beta,2a_n}^\top \mathtt{K}_{\beta,2a_n}$ using Fact \ref{fact:HS_squared}. (We do not denote the $n$-dependence of $\mathtt{J}_1$.)
\begin{align}\label{eq:firstHS_bound}
    \|\mathtt{J}_1-\mathtt{J}_2\|_{\textup{HS}}\le \|\mathtt{K}^{(n)}_{\beta,2a_n}-\mathtt{K}_{\beta,2a_n}\|_{\textup{HS}}(\|\mathtt{K}^{(n)}_{\beta,2a_n}-\mathtt{K}_{\beta,2a_n}\|_{\textup{HS}}+2\|\mathtt{K}_{\beta,2a_n}\|_{\textup{HS}}).
\end{align}
Note that the eigenvalues of the non-negative self-adjoint integral operators $\mathtt{J}_1,\mathtt{J}_2$ are given by the eigenvalues of $\big(n\ma{L}_{n,\beta,2a_n}\ma{L}_{n,\beta,2a_n}^\top\big)^{-1}$ and $\mathfrak{G}_{\beta,2a_n}^{-1}$, respectively.  More precisely, the non-zero eigenvalues of $\mathtt{J}_1$ are given by the eigenvalues of $\big(n\ma{L}_{n,\beta,2a_n}\ma{L}_{n,\beta,2a_n}^\top\big)^{-1}$, and it also has infinitely many 0 eigenvalues.

An explicit computation shows that $\ev[\|\mathtt{K}_{\beta,2a_n}\|_{\textup{HS}}^2]\le a_n^{-1}$ for $a_n\ge 1/\beta$ (see Lemma \ref{lem:HS_triangle_3} below). From \eqref{eq:firstHS_bound},
Proposition \ref{prop:HS_kernel_smalla}, and Markov's inequality we now get the bound
\begin{align}
    \pr\left(\|\mathtt{J}_1-\mathtt{J}_2\|_{\textup{HS}}^2>  n^{-1+2\alpha/3}\right)\notag&\le \pr\left(\|\mathtt{K}_{\beta,2a_n}\|_{\textup{HS}}^2>\tfrac19 n^{\alpha/3}\right)+\pr\left(\|\mathtt{K}^{(n)}_{\beta,2a_n}-\mathtt{K}_{\beta,2a_n}\|^2_{\textup{HS}}>n^{-1+\alpha/3} \right)\\&\le   \frac{2(\log n)^7}{a_n n^{\alpha/3}}.\label{eq:inv_bound}
\end{align}
By the Hoffman-Wielandt inequality (see e.g.~\cite{BhatiaElsner}), we have 
\begin{align}\label{eq:HW_bound}
 \|\mathtt{J}_1-\mathtt{J}_2\|_{\textup{HS}}^2 \ge   \sum_{k=1}^n \Big|(\lambda_{k,2a_n}^B)^{-1}-(n\lambda_{k,2a_n}^{(n)})^{-1}\Big|^2+ \sum_{k>n} (\lambda_{k,2a_n}^B)^{-2},
\end{align}
where $\lambda_{k,2a_n}^{(n)},1\le k\le n$ and $\lambda_{k,2a_n}^{B},k\ge 1$ denote the ordered points of the processes $\Lambda_{n,\beta,2a_n}$ and $\Bessel_{\beta,2a_n}$, respectively.

Recall that $1\ll a_n\le (\log n)^{1/2}$. Fix $k\ge 1$, and assume that $n$ is large enough so that $k\le 2a_n\le n$. 
Proposition \ref{prop:Bessel_growth} implies that  we have
\begin{align}\label{eq:Bess_bound}
\pr\big(\lambda_{k,2a_n}^B\ge (2a_n)^{3}\big)\le \pr\big(\lambda_{\lfloor 2a_n\rfloor,2a_n}^B\ge (2a_n)^{3}\big)\le O(a_n^{-3}\log a_n).
\end{align}
Consider the event
\[
\mathcal{E}_n:=\{ \|\mathtt{J}_1-\mathtt{J}_2\|_{\textup{HS}}^2\le n^{-1+2\alpha/3}\}\cap \{\lambda_{\lfloor 2 a_n\rfloor ,2a_n}^B\le (2a_n)^{3}\}.
\]
From the bound
\[
(n\lambda_{k,2a_n}^{(n)})^{-1}\ge (\lambda_{k,2a_n}^{B})^{-1}-\left|(n\lambda_{k,2a_n}^{(n)})^{-1}-(\lambda_{k,2a_n}^{B})^{-1}\right|,
\]
we have that on 
$
\mathcal{E}_n
$ the following  holds, for $n$ large enough: \[n\lambda_{k,2a_n}^{(n)}\le \left((\lambda_{k,2a_n}^{(B)})^{-1}-n^{-1/2+\alpha/3}\right)^{-1}\le 2\lambda_{\lfloor 2a_n\rfloor ,2a_n}^{(B)}< (3a_n)^3.
\]
This means that on $\mathcal{E}_n$, for $n$ large enough,  we have
\begin{align*}
    \sum_{i=1}^k |\lambda_{i,2a_n}^{B} - n\lambda_{i,2a_n}^{(n)}|^2 &=   \sum_{i=1}^k \Big|(\lambda_{i,2a_n}^{B})^{-1} - (n\lambda_{i,2a_n}^{(n)})^{-1}\Big|^2(\lambda_{i,2a_n}^{B}\cdot n\lambda_{i,2a_n}^{(n)})^2\\
    &\le (6a_n^6)^{2}  \sum_{i=1}^k \Big|(\lambda_{i,2a_n}^{B})^{-1} - (n\lambda_{i,2a_n}^{(n)})^{-1}\Big|^2\\
    &\le (6a_n^6)^{2} n^{-1+2\alpha/3}\le n^{-1+\alpha}.
\end{align*}
In the last step we again used $a_n\le (\log n)^{1/2}$. From \eqref{eq:inv_bound} and \eqref{eq:Bess_bound} we 
have $\lim_{n\to \infty} P(\mathcal{E}_n^c)=0$, which yields the statement of the proposition. 
\end{proof}
\begin{proof}[Proof of Theorem \ref{thm:main_lower} when $a_n\le (\log n)^{1/2}$]
Fix $k\ge 1$, and consider the random vectors 
\begin{align*}
\textbf{Lag}^{(n)}_{k,\beta,2a_n}&=(n\lambda_{1,\beta,2a_n}^{(n)},\cdots, n\lambda_{k,\beta,2a_n}^{(n)}), \qquad \textbf{Bess}_{k,\beta, 2a_n}=(\lambda_{1,\beta,2a_n}^B,\cdots, \lambda_{k,\beta,2a_n}^B),
\end{align*}
and $\textbf{Ai}_{k,\beta}=(\lambda_{1,\beta}^{Ai},\cdots, \lambda_{k,\beta}^{Ai})$,
where $\lambda_{k,\beta,2a_n}^{(n)}, \lambda_{k,\beta,2a_n}^B$ and $\lambda_{k,\beta}^{Ai}$ are the $k$th largest points in the processes $\Lambda_{n,\beta,2a_n}$, $\Bessel_{\beta,2a_n}$, and $\Airyb$, respectively.

    Proposition \ref{prop:hardedge_growing_a} implies that there is a coupling under which  $\|a^{-2/3} (\textbf{Lag}_{k,\beta,2a_n}^{(n)}-\textbf{Bess}_{k,\beta,2a_n})\|_2$ converges to 0 in probability. Theorem \ref{thm:hardtosoft}  shows that $a^{2/3} (\textbf{Bess}_{k,\beta,2a_n}-a^2)$ converges to $\textbf{Ai}_{k,\beta}$ in distribution. This implies that $a^{-2/3} (\textbf{Lag}_{k,\beta,2a_n}^{(n)}-a^2) $ converges to $\textbf{Ai}_{k,\beta}$ in distribution as well, proving the theorem. 
\end{proof}

\subsection{Proof of Proposition \ref{prop:HS_kernel_smalla}}\label{subsec:HS_coupling_idea}
This section outlines the proof of Proposition \ref{prop:HS_kernel_smalla}. Introduce the kernel $\tl {\mathsf{k}}^{(n)}_{\beta,2a_n}$ on $L^2[0,1]$ as 
\begin{align}\label{eq:kernel_tlKn}
    \tl {\mathsf{k}}^{(n)}_{\beta,2a_n} (x,y) =  \frac{1}{\sqrt{1-x}}\exp\left\{\sum_{k=i}^{j-1}\log\Big(\frac{Y_k}{X_{k+1}}\Big)\right\}\ind_L(x,y).
\end{align}
Note that  $\tl {\mathsf{k}}^{(n)}$ can be obtained from the  ${\mathsf{k}}^{(n)}$ by replacing $\frac{\sqrt{n}}{X_i}$ in \eqref{eq:kernel_Kn} with its deterministic approximation $\frac{1}{\sqrt{1-i/n}}$. 
By the triangle inequality, we have
\begin{align}
        \| \mathsf{k}^{(n)}_{\beta,2a_n} -  \mathsf{k}_{\beta,2a_n}\|_2 &\le  \| {\mathsf{k}}^{(n)}_{\beta,2a_n} \ind(\max(x,y)\ge 1-\eps)\|_2 +  \| \mathsf{k}_{\beta,2a_n} \ind(\max(x,y)\ge 1-\eps)\|_2\label{eq:HS_triangle_tail}\\
    &\quad +\| (\mathsf{k}^{(n)}_{\beta,2a_n}-\tl {\mathsf{k}}^{(n)}_{\beta,2a_n}) \ind(\max(x,y)\le 1-\eps)\|_2\label{eq:HS_triangle_LLN}\\
    &\quad +   \| (\tl {\mathsf{k}}^{(n)}_{\beta,2a_n}-{\mathsf{k}}_{\beta,2a_n}) \ind(\max(x,y)\le 1-\eps)\|_2, \label{eq:HS_triangle_coupling}
\end{align}
where the constant $\eps=\eps_n$ might depend on $n$. We will show that the terms on the right hand side of \eqref{eq:HS_triangle_tail} and \eqref{eq:HS_triangle_LLN} can be controlled by Markov's inequality, using second-moment estimates for the corresponding kernels. See  Lemmas \ref{lem:HS_triangle_1}, \ref{lem:HS_triangle_2} and \ref{lem:HS_triangle_3} below. The main technical difficulty is to obtain suitable control over the term in \eqref{eq:HS_triangle_coupling}. This  is based on an explicit coupling between the kernels $\tl {\mathsf{k}}^{(n)}_{\beta,2a_n}, {\mathsf{k}}_{\beta,2a_n}$, see Proposition \ref{prop:HS_coupling} for the precise statement.

\begin{lemma}\label{lem:kn_moments}
For $2a_n+1>2/\beta$, we have 
\begin{align}
        \ev\|{\mathsf{k}}_{\beta,2a_n}^{(n)}\|_2^2 =\frac{1}{2a_n+1-2/\beta}.
\end{align}
\end{lemma}
\begin{proof}
From \eqref{eq:kernel_Kn} and  the moments of chi-distributed random variables \eqref{eq:chi_moments}, we obtain
\begin{align}
     \ev\|\mathsf{k}_{\beta,2a_n}^{(n)}\|_2^2&=\frac{\beta}{n}\sum_{i=1}^n \sum_{j=i}^n \frac{\beta(n-i) \cdots \beta(n-j+1)}{(\beta(n+2a_n+1-i)-2)\cdots (\beta(n+2a_n+1-j)-2)}\notag\\
     &= \frac1{n} \sum_{i=1}^n \sum_{k=0}^{n-i} \frac{(n-i)^{\downarrow (k)}}{(n+2a_n+1-i-2/\beta)^{\downarrow (k+1)}},\label{eq:kernel_Kn_moments}
\end{align}
where we use the notation $n^{\downarrow k}:=n(n-1)\cdots(n-k+1)$.
By the telescoping identity
\begin{align}\label{id:telescoping}
    \frac{(n-i)^{\downarrow (k)}}{(n-i+x)^{\downarrow (k+1)}} = \frac{1}{x}\left( \frac{(n-i)^{\downarrow (k)}}{(n-i+x)^{\downarrow (k)}} -  \frac{(n-i)^{\downarrow (k+1)}}{(n-i+x)^{\downarrow (k+1)}}\right),\quad x\neq 0
\end{align}
we get 
\begin{align*}
     \sum_{k=0}^{n-i}\frac{(n-i)^{\downarrow (k)}}{(n+2a_n+1-i-2/\beta)^{\downarrow (k+1)}} =  \frac{1}{2a_n+1-2/\beta},\qquad 1\le i\le n,
\end{align*}
from which the statement follows.
\end{proof}

\begin{lemma}\label{lem:HS_triangle_1}
   Assume $a_n\ge 10/\beta+8$ and $2a_n/n\le \eps_n\le 1$. Then we have
   \begin{align*}
       \ev[\|\mathsf{k}_{\beta,2a_n}^{(n)} \ind(\max(x,y)>1-\eps_n) \|_2^2]\le \frac{2\eps_n}{a_n}.
   \end{align*}
\end{lemma}
\begin{proof}
Similar to the proof of Lemma \ref{lem:kn_moments}, we have
\begin{align}
    &\ev\|\mathsf{k}_{\beta,2a_n}^{(n)} \ind(\max(x,y)>1-\eps_n) \|_2^2\notag\\
    &= \frac{1}{n}\sum_{i=1}^{n(1-\eps_n)} \sum_{j=n(1-\eps_n)}^n  \frac{(n-i)^{\downarrow (j-i)}}{(n+2a_n+1-i-2/\beta)^{\downarrow (j-i+1)}}+  \frac{1}{n}\sum_{i=n(1-\eps_n)}^n \sum_{j=i}^n  \frac{(n-i)^{\downarrow (j-i)}}{(n+2a_n+1-i-2/\beta)^{\downarrow (j-i+1)}}\notag\\
    &=\frac{1}{n(2a_n+1-2/\beta)}\sum_{i=1}^{n(1-\eps_n)}\frac{(n-i)^{\downarrow (n(1-\eps_n)-i)}}{(n+2a_n+1-i-2/\beta)^{\downarrow (n(1-\eps_n)-i)}}+ \frac{\eps_n}{2a_n+1-2/\beta},\label{eq:HS_kn_tail}
    \end{align}
    where we used the telescoping identity \eqref{id:telescoping} in the second step, and do not denote integer parts in $n(1-\eps_n)$. To prove the statement  we need to bound the first sum in \eqref{eq:HS_kn_tail}. 
    
   If $0\le x\le 1/2$ then $\tfrac{1}{1+x}\le e^{-x/2}$.   Assuming $n-j\ge n \eps_n$ (together with $2a_n/n\le \eps_n$) we have
    \[
    \frac{n-j}{n+2a_n+1-j-2/\beta}\le\frac{1}{1+\frac{a_n}{n-j}}\le \exp\left( -\tfrac{a_n}{2(n-j)}\right).
    \]
   Hence for $1\le i\le n(1-\eps_n)$ we have
     \begin{align*}
        \frac{(n-i)^{\downarrow (n(1-\eps_n)-i)}}{(n+2a_n+1-i-2/\beta)^{\downarrow (n(1-\eps_n)-i)}}  
        &\le \exp\left(-\frac{a_n}2 \sum_{j=i}^{n(1-\eps_n)} \frac{1}{n-j} \right)\le \left(\frac{n\eps_n}{n-i}\right)^{a_n/4}.
    \end{align*}
    This leads to the upper bound
   \begin{align*}
        \frac{1}{n(2a_n+1-2/\beta)}&\sum_{i=1}^{n(1-\eps_n)}\frac{(n-i)^{\downarrow (n(1-\eps_n)-i)}}{(n+2a_n+1-i-2/\beta)^{\downarrow (n(1-\eps_n)-i)}}  
        \\
        &\le \frac{1}{a_n} \int_0^{1-\eps_n} \left(\frac{\eps_n}{1-x}\right)^{a_n/4}\le \frac{8\eps_n}{a_n^2}\le \frac{\eps_n}{a_n},
    \end{align*}
which implies the statement in the lemma.
\end{proof}

\begin{lemma}\label{lem:HS_triangle_2}
 Assume $a_n\ge 10/\beta+8$ and $2a_n/n\le \eps_n\le 1$. Then we have
    \begin{align}
        \ev\|(\mathsf{k}_{\beta,2a_n}^{(n)} -\tilde {\mathsf{k}}_{\beta,2a_n}^{(n)} )\ind(\max(x,y)\le 1-\eps_n) \|_2^2 \le \frac{10}{na_n}\log \eps_n^{-1}.
    \end{align}
\end{lemma}
\begin{proof}
For $a_n\ge 10/\beta$, by an exact computation we have 
$\ev(\frac{1}{X_i}-\frac{1}{\sqrt{n-i}})^2 \le \frac{2}{(n-i)^2}$.
Therefore,
\begin{align*}
   & \ev\|(\mathsf{k}_{\beta,2a_n}^{(n)} -\tilde {\mathsf{k}}_{\beta,2a_n}^{(n)} )\ind(\max(x,y)\le 1-\eps_n) \|_2^2\\
   &\qquad =\frac{1}{n}  \sum_{i=1}^{n(1-\eps_n)} \sum_{j=i}^{n (1-\eps_n)} \ev (X_i^{-1}-(n-i)^{-1/2})^2\frac{(n-i)^{\downarrow (j-i)}}{(n+2a_n+1-i-2/\beta)^{\downarrow (j-i)}}\\
   &\qquad \le \frac{2}{n}  \sum_{i=1}^{n(1-\eps_n)} \sum_{j=i}^{n (1-\eps_n)} \frac{1}{(n-i)^2}\cdot \frac{(n-i)^{\downarrow (j-i)}}{(n+2a_n+1-i-2/\beta)^{\downarrow (j-i)}}.
   \end{align*}
Since $n-j\ge n\eps_n\ge 2a_n$, using a similar estimate as in the proof of Lemma \ref{lem:HS_triangle_1}, we have 
\[
 \frac{(n-i)^{\downarrow (j-i)}}{(n+2a_n+1-i-2/\beta)^{\downarrow (j-i)}}\le \left(\frac{n-j}{n-i}\right)^{a_n/4},
\]
hence
\begin{align*}
 \ev\|(\mathsf{k}_{\beta,2a_n}^{(n)} -\tilde {\mathsf{k}}_{\beta,2a_n}^{(n)} )\ind(\max(x,y)\le 1-\eps_n) \|_2^2 &\le  \frac2n \sum_{i=1}^{n(1-\eps_n)} \sum_{j=i}^{n (1-\eps_n)}\frac{(n-j)^{\frac{a_n}{4}}}{(n-i)^{2+\frac{a_n}{4}}}\\
   & 
   \le \frac{9}{na_n}\sum_{i=1}^{n(1-\eps_n)}\frac{1}{n-i}\\
   &\le \frac{10}{na_n}\log\frac{1}{\eps_n}
\end{align*}
as claimed. 
\end{proof}

\begin{lemma}\label{lem:HS_triangle_3}
    For $2a_n+1> 2/\beta$ and $0< \eps\le 1$, we have 
    \begin{align}
          \ev[\|\mathsf{k}_{\beta,2a_n}\ind(\max(x,y)>1-\eps) \|_2^2]\le \frac{2\eps}{2a_n+1-2/\beta}. 
    \end{align}
\end{lemma}
\begin{proof}
Since 
\begin{align*}
    \ev\left[\exp\left(2\int_x^y \frac{db_z}{\sqrt{\beta(1-z)}}\right)\right] = 
    (1-x)^{2/\beta}(1-y)^{-2/\beta},
\end{align*}
we have 
    \begin{align*}
    \ev[\|\mathsf{k}_{\beta,2a_n}\ind(\max(x,y)>1-\eps) \|_2^2]& = \int_0^{1-\eps}\int_{1-\eps}^1 (1-x)^{-(1+2a_n)+2/\beta}(1-y)^{2a_n-2/\beta} dydx \\
        &\quad \quad +\int_{1-\eps}^1\int_{x}^1 (1-x)^{-(1+2a_n)+2/\beta}(1-y)^{2a_n-2/\beta} dydx\\
        & =\frac{\eps-\eps^{2a_n+1-2/\beta}}{2a_n-2/\beta}+\frac{\eps}{2a_n+1-2/\beta}\le \frac{2\eps}{2a_n+1-2/\beta}
    \end{align*}
as claimed. In particular, when $\eps=1$ we have $\ev[\|\mathsf{k}_{\beta,2a_n}\|_2^2]=\frac{1}{2a_n+1-2/\beta}$.
\end{proof}

The next result controls the remaining term $\| (\tl {\mathsf{k}}^{(n)}_{\beta,2a_n}-\mathsf{k}_{\beta,2a_n}) \ind(\max(x,y)\le 1-\eps)\|_2^2$. 
\begin{proposition}\label{prop:HS_coupling}
    There exists a coupling of $\tl k^{(n)}_{\beta,2a_n}$ and $k_{\beta,2a_n}$ such that for $n$ large enough, for any fixed $0<\alpha<1$, we have 
    \begin{align}\label{eq:HS_coupling_tilde}
        \pr(\|(\tl {\mathsf{k}}^{(n)}_{\beta,2a_n} -  \mathsf{k}_{\beta,2a_n}) \ind(\max(x,y)\le 1-\eps)\|^2_2\ge n^{-1+\alpha})\le \frac{(\log n)^{7-1/5}}{a_n n^{\alpha}}.
    \end{align}
\end{proposition}

We will introduce the coupling between the kernels $\tl k^{(n)}_{\beta,2a_n}$ and $k_{\beta,2a_n}$ in Sections \ref{subsec:coupling_idea} and \ref{subsec:coupling}. Then we will complete the proof of Proposition \ref{prop:HS_coupling} in Section \ref{subsec:HS_coupling}. Taking Proposition \ref{prop:HS_coupling} for granted, we now prove Proposition \ref{prop:HS_kernel_smalla}.
\begin{proof}[Proof of Proposition \ref{prop:HS_kernel_smalla}]
    Applying Lemmas \ref{lem:HS_triangle_1}, \ref{lem:HS_triangle_2} and \ref{lem:HS_triangle_3} with $\eps=\eps_n=\frac{(\log n)^2}{n}$, by Markov's inequality we have 
    \begin{align*}
        \pr( \| \mathsf{k}^{(n)}_{\beta,2a_n} \ind(\max(x,y)\ge 1-\eps)\|_2^2 \ge \tfrac{1}{16} n^{-1+\alpha}) &\le \frac{40(\log n)^2}{a_n n^{\alpha}},\\
        \pr (\|\mathsf{k}_{\beta,2a_n}\ind(\max(x,y)>1-\eps) \|_2^2\ge \tfrac{1}{16} n^{-1+\alpha}) &\le \frac{40(\log n)^2}{a_n n^{\alpha}},\\
        \pr(\|(\mathsf{k}_{\beta,2a_n}^{(n)} -\tilde {\mathsf{k}}_{\beta,2a_n}^{(n)} )\ind(\max(x,y)\le 1-\eps) \|_2^2 \ge \tfrac{1}{16}n^{-1+\alpha})&\le \frac{160 \log n}{a_n n^{\alpha}}.
    \end{align*}
    Together with the bound \eqref{eq:HS_coupling_tilde} and an application of the triangle inequality \eqref{eq:HS_triangle_tail}--\eqref{eq:HS_triangle_coupling} finishes the proof. 
\end{proof}

\begin{remark}\label{rmk:hardedge_rate}
    For fixed $\beta>0$ and $a>-1+2/\beta$, by following the arguments in \cite{BVBV_19} we expect that our coupling gives the rate of convergence $\|k_{\beta,a}^{(n)}-k_{\beta,a}\|_2^2 = O(\frac{(\log n)^{\kappa}}{n})$ for all large $n$. Given the length of the present paper, we leave a detailed study of this question for future work.
\end{remark}

\subsection{Review of the coupling method and single step coupling}\label{subsec:coupling_idea}

Recall the kernels $\mathsf{k}_{\beta,2a_n}$ and $\tl {\mathsf{k}}^{(n)}_{\beta,2a_n}$ defined in \eqref{eq:kernel_K} and \eqref{eq:kernel_tlKn}, respectively. Constructing an exact coupling between $\mathsf{k}_{\beta,2a_n}$ and $\tl {\mathsf{k}}^{(n)}_{\beta,2a_n}$ is equivalent to embedding a random walk, whose increments are given by the difference of two independent $\log\chi$-distributed random variables, into a one-dimensional Brownian motion with drift. 
Following the coupling techniques developed in \cite{BVBV_19}, we first focus on constructing a coupling between a single step in the random walk path and the corresponding increment of the drifted Brownian motion.

\begin{fact}
    Let $Y_1\sim\chi_{2p}$ and $Y_{2}\sim\chi_{2q}$ be independent, then $\log Y_1-\log Y_2$ is a random variable supported on $\R$ with density
\begin{align}\label{eq:logchi_ratio_pdf}
f(x) = \frac{2\Gamma(p+q)}{\Gamma(p)\Gamma(q)}\frac{e^{2px}}{(1+e^{2x})^{p+q}}. 
\end{align}
\end{fact}
We say $\xi\sim \log \mathrm{beta}'(p,q)$ if the probability density function of $\xi$ is \eqref{eq:logchi_ratio_pdf}.
Observe that the function $f(\cdot)$ in \eqref{eq:logchi_ratio_pdf} satisfies 
\begin{align}\label{eq:logbetaprime}
    \frac{f(x)}{f(x)+f(-x)} = \frac{e^{2px}}{e^{2px}+e^{2qx}} = \frac{1}{1+e^{2(q-p)x}}.
\end{align}
This matches with the hitting probability of a one-dimensional Brownian motion with drift $\cdr=p-q$.  More precisely, let $B^{(\cdr)}(t)$ be standard Brownian motion with drift $\cdr$ started from 0. Set $\tau_x:=\inf\{t\ge 0:|B^{(\cdr)}(t)|=x\}$. Then by the optimal stopping theorem we have
\begin{align}\label{eq:hitting}
\pr(B^{(\cdr)}(\tau_x) =x) =1-\pr(B^{(\cdr)}(\tau_x) = -x) = \frac{e^{2\cdr x}-1}{e^{2\cdr x}-e^{-2\cdr x}}=\frac{1}{1+e^{-2\cdr x}}.
\end{align}
This observation leads to the following distributional identity.
\begin{proposition}\label{prop:coupling}
   Let $\xi\sim \log \mathrm{beta}'(p,q)$, and let  $B^{(\cdr)}(t)$ be an independent standard Brownian motion with drift $\cdr=p-q$. Set $\tau:=\inf\{t\ge 0:|B^{(\cdr)}(t)|=|\xi|\}$, then $B^{(\cdr)}(\tau)\ed \xi$.
\end{proposition}
\begin{proof}
Let $f(\cdot)$ be the probability density function of $\xi$ defined in \eqref{eq:logchi_ratio_pdf}, then $|\xi|$ has density function given by $(f(x)+f(-x))\ind_{x\ge 0}$.
Using the hitting probability of drifted Brownian motion and the independence of $\xi$ and $B^{(\cdr)}$, for $r\ge 0$ we have
    \[
    \pr(B^{(\cdr)}(\tau)\ge r) = \int_r^\infty (f(x)+f(-x))\frac{1}{1+e^{-2\cdr x}}dx = \int_r^\infty f(x)dx = \pr(\xi\ge r). 
    \]
where in the second equality we use that $f(x)+f(-x) = f(x)(1+e^{-2\cdr x})$. A similar calculation gives $\pr(B^{(\cdr)}(\tau)\le -r)= \pr(\xi\le -r)$.  
\end{proof}

Next, we estimate the total variation distance (denoted by $d_{TV}$ below) between  $|B^{(\cdr)}(\tfrac{1}{p+q})|$ and $|\xi|$ for $\xi\sim \log \mathrm{beta}'(p,q)$ and $\cdr=p-q$.
\begin{proposition}\label{prop:tv}
   Consider the same setup as in Proposition \ref{prop:coupling}. Then $|\xi|$ stochastically dominates $|B^{(\cdr)}(\frac{1}{p+q})|$, i.e.~$\pr(|\xi|\ge x)\ge \pr(|B^{(\cdr)}(\frac{1}{p+q})|\ge x)$ for all $x\ge 0$. Moreover, for $p,q$ satisfying $p+q\ge 10(p-q)^2+10$ we have  $d_{TV}(|\xi|, |B^{(\cdr)}(\frac{1}{p+q})|)\le \frac{1}{p+q}$.
   \end{proposition}
The proof is postponed to the Appendix. 
Our next result provides a representation of  the $\log\mathrm{beta}'$ distribution as  the standard Brownian motion with drift evaluated at a  stopping time that is `almost' deterministic.

\begin{proposition}\label{prop:single_coupling} 
    Let $B(t)$ be standard Brownian Motion, and let $U\sim\textup{Unif }[0,1]$ be independent. Let $\mathcal{F}_t,t\ge 0$ be the filtration of $B$ enlarged with $U$. Suppose $\xi\sim\log\mathrm{beta}'(p,q)$ with $p+q\ge 10(p-q)^2+10$, then there exists a finite stopping time $\sigma$ with respect the filtration $\mathcal{F}_t, t\ge 0$ such that 
    \begin{enumerate}[(a)]
        \item $B(\sigma)+(p-q)\sigma\ed \xi $,
        \item $\pr(\sigma \ge \frac{1}{p+q})=1$ and $\pr(\sigma\neq \frac{1}{p+q})\le \frac{1}{p+q}$,
        \item There exist absolute constants $c_1,c_2>0$ such that $\pr(\sigma>\frac{u}{p+q})\le c_1e^{-c_2u^{1/3}}$.       
    \end{enumerate}
\end{proposition}

The proof of the statement relies on Propositions \ref{prop:coupling}, \ref{prop:tv} and the following standard coupling lemma.

\begin{lemma}[Lemma 11 of \cite{BVBV_19}]\label{lem:tv_coupling}
    Assume that $X_1$ and $X_2$ are non-negative random variables so that $X_2$ stochastically dominates $X_1$ and $d_{TV}(X_1,X_2)=\eps$. Then there exists a measurable function $g : \R^2 \mapsto \R$ so that if $U\sim\text{Unif }[0,1]$ independent of $X_1$, then 
    \begin{enumerate}[(a)]
        \item $g(X_1,U)\ed X_2$,
        \item $\pr(X_1\le g(X_1,U))=1$,
        \item $\pr(X_1=g(X_1,U))=1-\eps$.
    \end{enumerate}
\end{lemma}
We now turn to the proof of Proposition \ref{prop:single_coupling}.
\begin{proof}[Proof of Proposition \ref{prop:single_coupling}] Let $r$ and $\cdr$ be defined with $p=r+\cdr/2$, $q=r-\cdr/2$.  We denote by $B^{(\cdr)}(t):=B(t)+\cdr t$ the standard Brownian motion with drift $\cdr=p-q$. For $x>0$ let $\nu_x$ denote the first hitting time of  $B^{(\cdr)}(t)$ of $\pm x$. 
Then by the Markov property, for $0<x<y$ we have
\begin{align}
\pr_0&(B^{(\cdr)}({\nu_y})=y)
\notag\\&=\pr_0(B^{(\cdr)}({\nu_y})=y\vert B^{(\cdr)}({\nu_x})=x)\pr_0(B^{(\cdr)}({\nu_x})=x)\notag\\
&\qquad+\pr_0(B^{(\cdr)}({\nu_y})=y\vert B^{(\cdr)}({\nu_x})=-x)\pr_0(B^{(\cdr)}({\nu_x})=-x)\notag\\
&=\pr_x(B^{(\cdr)}({\nu_y})=y)\pr_0(B^{(\cdr)}({\nu_x})=x)+P_{-x}(B^{(\cdr)}({\nu_y})=y)\pr_0(B^{(\cdr)}({\nu_x})=-x)\label{eq:BM_hitting_Markov},
\end{align}
where we denote by $\pr_x$ the distribution of $B^{(\cdr)}$ if the process starts from $x$ at time $0$. 
Moreover, we have $\pr_x(B^{(\cdr)}({\nu_y})=y)=1-\pr_x(B^{(\cdr)}({\nu_y})=-y)=\frac{1-e^{-2\cdr (x+y)}}{1-e^{-4\cdr y}}$.  

By Proposition \ref{prop:tv}, $|\xi|$ stochastically dominates $|B^{(\cdr)}(\frac{1}{2r})|$, and their total variation distance is bounded by $\frac{1}{2r}$. Now consider the function $g:\R^2\mapsto \R$ produced by Lemma \ref{lem:tv_coupling}.  We have that $R:=g(|B^{(\cdr)}(\frac{1}{2r})|,U)\ed |\xi|$, and almost surely $R\ge |B^{(\cdr)}(\frac{1}{2r})|$ with $\pr(|B^{(\cdr)}(\frac{1}{2r})|\neq R)\le \frac{1}{2r}$. 
Set 
    \[
    \sigma = \inf \{v\ge \tfrac{1}{2r}: |B^{(\cdr)}(v)|=R\}.
    \] 
By definition, $B^{(\cdr)}(\sigma)$ has the distribution of $B^{(\cdr)}({\nu_R})$ under $\pr_{B^{(\cdr)}(\frac{1}{2r})}$. By the hitting probability \eqref{eq:hitting} we have
\[
\pr_0 \left(B_{1/(2r)}^{(\cdr)} = \rho \Big\vert  |B_{1/(2r)}^{(\cdr)}|=\rho\right) = \frac{1}{1+e^{-2\cdr \rho}}.
\]
Next we condition on  $\{|B_{1/(2r)}^{(\cdr)}|=\rho, U=u\}$, note that then we have $R=g(\rho,u)$.
By the strong Markov property, we have
\begin{align*}
    \pr&\left(B^{(\cdr)}(\sigma)=R \Big \vert |B^{(\cdr)}(\tfrac{1}{2r})|=\rho, U=u\right) \\
    &= \pr_0 \left(B^{(\cdr)}(\tfrac{1}{2r}) = \rho \big\vert |B^{(\cdr)}(\frac{1}{2r})|=\rho, U=u\right)\cdot \frac{1-e^{-2\cdr(\rho+R)}}{1-e^{-4\cdr R}}\\
    &\qquad +  \pr_0 \left(B^{(\cdr)}(\tfrac{1}{2r}) = -\rho\big \vert |B^{(\cdr)}(\frac{1}{2r})|=\rho, U=u\right)\cdot\frac{1-e^{-2\cdr(R-\rho)}}{1-e^{-4\cdr R}}\\
     &= \frac{1}{1+e^{-2\cdr R}}.  
\end{align*}
Hence for $x>0$ we have $P(R=x\vert |R|=x)= \frac{1}{1+e^{-2\cdr x}}$, and since $R\ed |\xi|$, by \eqref{eq:logbetaprime} this proves that $B^{(\cdr)}(\sigma) \ed \xi$.
Together with Proposition \ref{prop:tv} completes the proof of the first two statements. 

It remains to prove the tail bound $\textit{(c)}$. The proof is a simple modification of the proof of Proposition 10 of \cite{BVBV_19}, and we defer the details to Section \ref{app:couple} in the Appendix.
\end{proof}

\subsection{Path coupling and path bounds}\label{subsec:coupling}
The goal of this section is to apply the single step coupling introduced in Proposition \ref{prop:single_coupling} to the random paths appearing in the kernels $\tl {\mathsf{k}}^{(n)}_{\beta,2a_n}$ and $\mathsf{k}_{\beta,2a_n}$. Proposition \ref{prop:path_coupling} below constructs the coupling between the paths, and Proposition \ref{prop:path_comparison}
shows that we can control the distance between the paths under our coupling.

\begin{proposition}\label{prop:path_coupling}
    Fix $\beta>0$ and assume $a_n\le (\log n)^{1/2}$. 
    Let $B^{(\cdr_n)}(t)$ be standard Brownian motion with drift $\cdr_n=-a_n\beta$ and let $U$ be an independent $\textup{Unif }[0, 1]$ random variable. Let $\mathcal{F}_t,t\ge 0$ be the filtration of $B$ enlarged with $U$.
    Let $\xi_k, \lfloor\log^6 n\rfloor\le k\le n-1$ be independent $\log\mathrm{beta}'(p_k,q_k)$-distributed random variables with $p_k=\beta k/2,q_k=\beta(k+2a_n)/2$. 
    
    Then there exists a sequence of stopping times $\{\tau_{n,k},\lfloor\log^6 n\rfloor\le k\le n\}$ with respect to $\mathcal{F}_t,t\ge 0$ such that the following statements hold for all $n$ larger than a fixed $\beta$ dependent constant.
    \begin{enumerate}[(a)]
        \item We have $0=\tau_{n,n}<\tau_{n,n-1}<\dots<\tau_{n,\lfloor\log^6 n\rfloor}$. For $\lfloor\log^6 n\rfloor\le k\le n-1$ the differences $\Delta \tau_{n,k} := \tau_{n,k}-\tau_{n,k+1}$ are independent, and
        \[
            B^{(\cdr_n)}(\tau_{n,k}) - B^{(\cdr_n)}(\tau_{n,k+1})\ed \xi_k.
        \]
        \item Let $t_{n,k} = \frac{1}{\beta}\log\frac{n}{k}$ for $\lfloor\log^6 n\rfloor\le k\le n$, then  we have 
        \begin{align*}
               \pr\left(t_{n,k} -\frac{a_n}{k\beta }\le \tau_{n,k}\le t_{n,k} + \frac{\log^{4+1/2}n}{k}, \quad \lfloor\log^6 n\rfloor\le k\le n\right)\ge 1-e^{-\log^{10/9}n}.
        \end{align*}    
    \end{enumerate}
\end{proposition}
\begin{proof}
    Using a standard argument we can find measurable functions $f_{1,k}$, $f_{2,k}$ so that $U_k =f_{1,k}(U)\sim \text{Unif }[0, 1]$, $\xi_k = f_{2,k}(U)\sim \log \mathrm{beta'}(p_k,q_k)$, and the random variables $U_1,U_2,\dots$ and  $\xi_1,\xi_2,\dots$ are independent. 

   Note that $\cdr_n=p_k-q_k=-a_n\beta$ is independent of $k$.
    By the assumptions that $a_n\le (\log n)^{1/2}$, we have $\lfloor\log^6 n\rfloor\ge 10(\cdr_n^2+1)$ for all $n$ large enough ($n\ge e^{2\beta+20}$ would work).  We define $\tau_{n,k}$ recursively starting with $\tau_{n,n}=0$. Suppose $\tau_{n,n},\dots,\tau_{n,k+1}$ are defined. For the next step, by applying Proposition \ref{prop:single_coupling} with $p_k=k\beta/2, q_k=(k+2a_n)\beta/2$, and $U_k$ for the drifted Brownian motion $\tl B^{(\cdr_n)}(t) = B^{(\cdr_n)}(t+\tau_{n,k+1})-B^{(\cdr)}(\tau_{n,k+1})$, $\cdr_n=-a_n\beta$, we can find a stopping time $\sigma_{n,k}$ such that $\tl B^{(\cdr)}(\sigma_{n,k}) \ed \xi_k$. 
    Setting $\tau_{n,k}:=\tau_{n,k+1}+\sigma_{n,k}$ proves the first statement.

    For the second statement, note that under our coupling we have a.s.~for $\lfloor\log^6 n\rfloor\le k\le n-1$ that 
    \begin{align*}
        \tau_{n,k} \ge \sum_{j=k}^{n-1}\frac{1}{p_k+q_k} = \sum_{j=k}^{n-1}\frac{1}{(k+a_n)\beta} \ge \frac{1}{\beta}\log \frac{n+a_n}{k+a_n}\ge t_{n,k}-\frac{a}{\beta k}.
    \end{align*}
    This proves the lower bound for $\tau_{n,k}$. On the other hand, using the tail bound in Proposition \ref{prop:single_coupling} we also have 
    \begin{align*}
        \pr\Big(\sigma_{n,k}\ge \frac{\log^{3+3/8}n}{p_k+q_k}\Big)\le c_1e^{-c_2\log^{1+1/8}n}. 
    \end{align*}
    Set 
    \[
    Z_{n,k} = \frac{1}{p_k+q_k}\ind({\sigma_{n,k}= \tfrac{1}{p_k+q_k}}) + 2\frac{\log^{3+3/8}n}{p_k+q_k} \ind({\sigma_{n,k}\neq \tfrac{1}{p_k+q_k}}),
    \]
    By a  union bound, we have
    \begin{align}\label{eq:stopping_bound1}
        \pr(\sigma_{n,k}\le Z_{n,k}, \text{ for all } \lfloor\log^6 n\rfloor\le k\le n-1) \ge 1-c_1ne^{-c_2\log^{9/8}n}\ge 1-\tfrac12 e^{-\log^{10/9}n},
    \end{align}
for $n$ large enough. This implies the same lower bound on the probability that  $\tau_{n,k}\le \sum_{j=k}^{n-1}Z_{n,k}$ for all $\lfloor\log^6 n\rfloor\le k\le n-1$.
    To finish the argument, we can just repeat the proof of Proposition 12 of \cite{BVBV_19} to get an exponential  Markov inequality for the tails of $Z_{n,k}$. This  leads to the estimate 
    \begin{align}\label{eq:stopping_bound2}
    \pr\left(\sum_{j=k}^{n-1}Z_{n,k}\le t_{n,k}+\frac{\log^{4+1/2}n}{k},\text{ for all }  \lfloor\log^6 n\rfloor \le k\le n-1\right) \ge 1-\tfrac12 e^{\log^{10/9}n},
    \end{align}
for $n$ large enough.    Combining the bounds \eqref{eq:stopping_bound1} and \eqref{eq:stopping_bound2} completes the proof.
\end{proof}

Consider the coupling in Proposition  \ref{prop:path_coupling}, set  
\begin{align}\label{def:path_log}
\tl B^{(\cdr_n)}_n(t) = B^{(\cdr_n)}(\tau_{n,\lceil n(1-t)\rceil}),\quad\tl B^{(\cdr_n)}(t) = B^{(\cdr_n)}\big(\tfrac{1}{\beta}\log\tfrac{1}{1-t})\big),\quad 0\le t\le \mathsf{t}_n:=1-\tfrac{\lfloor \log^6 n\rfloor}{n}.
\end{align}
The next result provides an upper bound on the distance between these two processes.
\begin{proposition}\label{prop:path_comparison}
    Consider the same setup as in Proposition \ref{prop:path_coupling}, we have 
\begin{align*}
    \pr\left(|\tl B^{(\cdr_n)}(t) - \tl B^{(\cdr_n)}_n(t)| \le \frac{\log^{3-1/8}n}{\sqrt{n(1-t)}},\quad 0\le t\le \sft_n\right)\ge 1-2e^{-\log^{10/9}n}.
    \end{align*}
\end{proposition}
\begin{proof}
    By Proposition \ref{prop:path_coupling}, we have 
    \begin{align*}
        |\tau_{n,\lceil n(1-t)\rceil} + \frac{1}{\beta}\log(1-t)| \le \frac{a_n}{\lceil n(1-t)\rceil}+\frac{\log^{4+1/2}n}{\lceil n(1-t)\rceil}\le \frac{\log^{4+3/5}n}{n(1-t)},\quad 0\le t\le \sft_n,
    \end{align*}
    where we used the assumption that $a_n\le (\log n)^{1/2}$ in the second inequality.
    Therefore, 
    \begin{align}\label{eq:pathbound_00}
         |\tl B^{(\cdr_n)}(t) - \tl B^{(\cdr_n)}_n(t)| \le a_n\beta\,  h+ \max\{|B(s)-B(s+u)|: |u|\le h, 0\le s+u\},
    \end{align}
    where $h=h(t)=\frac{\log^{4+3/5} n}{n(1-t)}\le \frac{1}{\log n}$ and $s=s(t)=-\frac{1}{\beta}\log(1-t)$. 

 For the drift term in \eqref{eq:pathbound_00}, it is direct to check that when $a_n\le (\log n)^{1/2}$ we have 
    \begin{align}\label{eq:path_bound2}
    a_n \beta h= a_n\beta \frac{\log^{4+3/5} n}{n(1-t)}\le \frac12\frac{\log^{3-1/8}n}{\sqrt{n(1-t)}},\quad t\le \mathsf{t}_n,
    \end{align}
    for $n$ large enough. 

To estimate the second term in \eqref{eq:pathbound_00} we first note that $s(t)+h(t)\le c_\beta \log n:=T_n$ with a fixed constant $c_\beta$. Using  modulus continuity bounds of standard Brownian motion (see e.g.~the remark below Lemma 3.2 in \cite{Kurtz1978}) we have that
\begin{align}
    M=\max_{0<s_1<s_2<T_n} \frac{|B(s_1)-B(s_2)|}{\sqrt{(s_2-s_1)(1+ \log(\frac{T_n}{s_2-s_1})})}
\end{align}
is a finite random variable that satisfies $\ev e^{\lambda M^2}<\infty$ for some positive $\lambda$. Using an exponential Markov bound on $\pr(M^2>\log^{9/8} n)$ and collecting all our estimates on $h, s$ and $T_n$ finishes the proof. 
\end{proof}

\subsection{Proof of Proposition \ref{prop:HS_coupling}}\label{subsec:HS_coupling}
Using the path bounds in Proposition \ref{prop:path_comparison}, we are able to control  $\|\tl {\mathsf{k}}^{(n)}_{\beta,2a_n}-\mathsf{k}_{\beta,2a_n}\|_2$. 
\begin{proof}[Proof of Proposition \ref{prop:HS_coupling}]
   Let  
   \[
   \mathcal{C}_n=\left\{|\tl B^{(\cdr_n)}(t) - \tl B^{(\cdr_n)}_n(t)| \le \frac{\log^{3-1/8}n}{\sqrt{n(1-t)}},\quad 0\le t\le \sft_n\right\}.
   \]
Note that $\frac{\log^{3-1/8}n}{\sqrt{n(1-\sft_n)}}\le 1$ if $n$ is large enough.    By Proposition \ref{prop:path_comparison}, we have $\pr(\mathcal{C}_n) \ge 1-2e^{-\log^{10/9}n}$ (if $n$ is large enough), and on $\mathcal{C}_n$ we have 
    \begin{align*}
         &\Big|\exp\big(\tl B^{(\cdr_n)}(y)-\tl B^{(\cdr_n)}(x)\big) - \exp\big(\tl B^{(\cdr_n)}_n(y)-\tl B^{(\cdr_n)}_n(x)\big)\Big|\\
         &\qquad\le 2\exp\big(\tl B^{(\cdr_n)}(y)-\tl B^{(\cdr_n)}(x)\big)\Big(\big|\tl B^{(\cdr_n)}(y)-\tl B^{(\cdr_n)}_n(y)\big|+\big|\tl B^{(\cdr_n)}(x)-\tl B^{(\cdr_n)}_n(x)\big|\Big) \\
         &\qquad\le 4\exp\big(\tl B^{(\cdr_n)}(y)-\tl B^{(\cdr_n)}(x)\big)\frac{\log^{3-1/8}n}{\sqrt{n(1-y)}},
    \end{align*}
    for $0\le x<y\le \sft_n$, where in the second line we could replace the exponentials with linear functions because of our upper bound on the path differences.
    This path bound can be further translated into an upper bound for $|\mathsf{k}_{\beta,2a_n}-\tl {\mathsf{k}}_{\beta,2a_n}^{(n)}|$ on the event $\mathcal{C}_n$. Integrating over the domain $\max\{x,y\}\le \sft_n$ under the event $\mathcal{C}_n$, we get 
    \begin{align*}
      \ind_{\mathcal{C}_n}  \|(\tl {\mathsf{k}}^{(n)}_{\beta,2a_n} -  \mathsf{k}_{\beta,2a_n}) &\ind(\max(x,y)\le \sft_n)\|^2_2 \le  \ind_{\mathcal{C}_n}  16 \tfrac{(\log n)^{6-1/4}}{n}\int_0^{\sft_n}\int_x^{\sft_n} \frac{e^{2(\tl B^{(\cdr_n)}(y)-\tl B^{(\cdr_n)}(x))}}{(1-x)(1-y)} dy dx.
    \end{align*}
    Taking expectation, we get
    \begin{align*}
        \int_0^{\sft_n}\int_x^{\sft_n} \frac{\ev [e^{2(\tl B^{(\cdr_n)}(y)-\tl B^{(\cdr_n)}(x)}]}{(1-x)(1-y)} dy dx &= \int_0^{\sft_n}\int_x^{\sft_n} (1-x)^{-1-2a_n+2/\beta}(1-y)^{2a_n-1-2/\beta} dydx\\
        &= \frac{-\log (1-\sft_n) }{2a_n-2/\beta}\le \frac{\log n}{a_n}.
    \end{align*}
 Hence by Markov's inequality, for any $0<\alpha<1$ we get 
    \begin{align*}
     \pr\big(\|(\tl k^{(n)}_{\beta,2a_n} -  k_{\beta,2a_n})&\ind(\max(x,y)\le \sft_n)\|^2_2\ge n^{-1+\alpha}\big) \le 16 \frac{(\log n)^{7-1/4}}{a_nn^{\alpha}}+P(\mathcal{C}_n^c).
    \end{align*}
Since $\pr(\mathcal{C}_n^c) < 2e^{-\log^{10/9}n}$, the statement now follows. 
\end{proof}

\section{Proofs of Theorems \ref{thm:Gaussian} and \ref{thm:Laguerre}} \label{sec:other_thms}
Following our proof for Theorem \ref{thm:main_lower} in the case when $a_n\gg (\log\log n)^3$, we present a general framework to prove  operator level edge scaling limits of beta-ensembles to the $\Airyop$ operator. As applications, we sketch the proofs in the cases of the lower edges of Gaussian beta-ensemble and Laguerre beta-ensemble when $\liminf_{n\to\infty} a_n/n >0$. The scaling limits of the upper edges of the Gaussian and Laguerre beta-ensembles can be derived by symmetry and by a simple modification of the lower edge cases, respectively. 

\subsection{Outline of a general framework}
Suppose $\mathbf{M}_n$ is a sequence of invertible symmetric tridiagonal matrices of form \eqref{eq:tridiagonal}, that can be approximated by a discrete Laplacian with some time scale $m_n^{-1}$.
 To approximate the discrete Laplacian, we scale $\mathbf{M}_n$ by $m_n^2$.
The size of the discrete Laplacian is determined by the leading order of $e_1$, denoted by $\rho_n$. Define $\wtl {\mathbf{M}}_n = \sigma_n \mathbf{M}_n$, where $\sigma_n = m_n^2/\rho_n$. We would like to show that under certain moment conditions and concentration bounds for the entries of $\ma{M}_n$, there is a coupling so that
\begin{align*}
    \|\mathsf{K}_{\wtl{\mathbf{M}}_n}-\mathsf{K}_{\Ai}\|_{2}\to 0\qquad \text{in probability as $n\to\infty$},
\end{align*}
where 
\begin{align*}
    \mathsf{K}_{\wtl{\mathbf{M}}_n}(x,y) = \sigma_n^{-1}m_n[\ma{M}^{-1}_n]_{\lceil x m_n\rceil, \lceil y m_n\rceil} \ind(0<x, y\le n/m_n).  
\end{align*}
Following the proof of Theorem \ref{thm:main_lower} when $a_n\gg(\log\log n)^3$, it suffices to derive the analogous versions of Lemmas \ref{lem:HS_T} and \ref{lem:HS_tail} for the kernel $\mathsf{K}_{\wtl{\ma{M}}_n}$. For a fixed $n$ and $\ma{M}\equiv \ma{M}_n$, consider again the sequence $u_i$ satisfying \eqref{eq:uv_1} associated with $\ma{M}$. By \eqref{eq:udef2}, $u_i$ satisfies the recursion
\[
-u_{i+1} e_{i} + u_i d_i - u_{i-1}e_{i-1}=0,\qquad u_0=0, u_1=1.
\]
The two key ingredients of proving Lemmas \ref{lem:HS_T} and \ref{lem:HS_tail} are:
\begin{enumerate}
    \item  For $i/m_n$ uniformly on compacts we have 
    \begin{align}\label{eq:conv_to_Airy_d}
    (m_n^{-1} u_i, u_{i+1}- u_i)\Rightarrow (\psi_\ddd,\psi_\ddd'),
    \end{align}
    in law with respect to the Skorokhod topology.
    \item Denote by $\ma{M}_n^{(T)}$ the upper-left $n_0\times n_0$ corner of $\ma{M}$ with $n_0:=\lfloor Tm_n\rfloor$. By Lemma \ref{lem:tri_inverse} we have  $[\ma{M}_n^{-1}]_{ij}=u_iv_j$ for $1\le i\le j\le n$ and $[(\ma{M}_n^{(T)})^{-1}]_{ij}=u_iv_j^{(T)}$ for $1\le i\le j\le n_0$. Then 
\begin{equation}\label{eq:HS_tail_general}
\lim_{T\to\infty}\lim_{n\to\infty}\, \sigma_n^{-2}\left(\sum_{j\ge n_0,j\ge i}u_i^2v_j^2 + \sum_{1\le i\le j\le n_0}u_i^2(v_j-v_j^{(T)})^2\right)=0,\text{\quad in probability.}
\end{equation}
\end{enumerate}
We will illustrate these ideas by studying the lower edge scaling limits of Gaussian beta-ensembles and Laguerre beta-ensembles when $\lim_{n\to\infty} a_n/n\to c\in(0,\infty]$. We will also explain why the estimates are much simpler than in the proof of Theorem \ref{thm:main_lower} when $(\log\log n)^3\ll a_n\ll n$.

\subsection{Edge limits of Gaussian beta-ensemble}\label{subsec:Gaussian}

Recall the Dumitriu-Edelman matrix models $\ma{G}_n$ for Gaussian beta-ensembles in \eqref{eq:tri_G}. 
Let $u_i\equiv u_i^G, 0\le i\le n$ be the sequence satisfying the recursions \eqref{eq:udef1} and \eqref{eq:udef2} with respect to the entries of $\ma{M}_n\equiv\ma{M}_n^G:=(\ma{G}_n+2\sqrt{n})$,
\begin{align}\label{eq:u_recursion_G}
    u_{k+1}  = \frac{2\sqrt{n}+g_k}{Y_k}u_k - \frac{Y_{k-1}}{Y_k}u_{k-1},\qquad u_0=0,u_1=1.
\end{align}
Set 
\[
\bar u_k \equiv \bar u_k^G:= u_k\prod_{j=1}^{k-1}\tfrac{\sqrt{n-j}}{Y_j},\quad 1\le k\le n.
\]
then we have 
\begin{align*}
    \bar u_{k+1} = \frac{n-k}{Y_k^2}\cdot\frac{2\sqrt{n}+g_k}{\sqrt{n-k}}\bar u_k - \frac{n-k}{Y_k^2}\sqrt{\frac{n-k+1}{n-k}}\bar u_{k-1}.
\end{align*}
Introduce 
\begin{align*}
    Z_{1,k} \equiv Z_{1,k}^G =\frac{n-k}{Y_k^2}\cdot \frac{2\sqrt{n}+g_k}{\sqrt{n-k}}-2,\qquad Z_{2,k} \equiv Z_{2,k}^G= \frac{n-k}{Y_k^2}\cdot \frac{2\sqrt{n}+g_k-\sqrt{n-k+1}}{\sqrt{n-k}}-1,
\end{align*}
then we get 
\begin{align}\label{eq:ubar_diff_G}
    (\bar u_{k+1}-\bar u_k) - (\bar u_k -\bar u_{k-1}) = Z_{1,k}(\bar u_k-\bar u_{k-1})  + Z_{2,k} \bar u_{k-1}.
\end{align}
(This is of the same form as \eqref{eq:ubar_diff_recursion} in the proof of Proposition \ref{prop:ubar_conv}.)
The sequence $\bar u_k$ is constructed in a way so that $Z_{1,k},Z_{2,k}$ are independent of $\bar u_j,j\le k-1$. Using the independence of the random entries $g_k,Y_k$ of $\ma{G}_n$, we obtain the following moment bounds for $Z_{1,k},Z_{2,k}$. Set $m_n\equiv m_n^G=n^{1/3}$ and define $n_0\equiv n_0(T):=\lfloor Tm_n\rfloor$ for $T>0$ large.

\begin{proposition}\label{prop:moments_G}
    For $1\le k\le n_0$, we have 
    \begin{align*}
        \ev[Z_{1,k}] &= \frac{k}{n}+O(n^{-1}),\qquad  \Var[Z_{1,k}]=O(n^{-1}),\\
       \ev[Z_{2,k}]&= \frac{k}{n}+O(n^{-1}),\qquad \Var[Z_{2,k}]=\frac{4}{\beta n}+O(kn^{-2}),\\
         \Cov[Z_{1,k},Z_{2,k}]&=O(n^{-1}),\qquad \ev[Z_{1,k}^4]=O(n^{-2}), \qquad \ev[Z_{2,k}^4]=O(n^{-2}). 
    \end{align*}
    For $n_0\le k\le n_1:=\lfloor\frac{n}{2}\rfloor$ we have   \begin{align*}
        \ev[Z_{1,k}] &= \frac{2k}{\sqrt{n-k}(\sqrt{n}+\sqrt{n-k})}+O(n^{-1}), \qquad \Var[Z_{1,k}]=O\Big(\frac{1}{n-k}\Big),\\
        \quad \ev[Z_{2,k}] &= \frac{2k}{\sqrt{n-k}(\sqrt{n}+\sqrt{n-k})}+O(n^{-1}),
        \qquad \Var[Z_{2,k}]=\frac{4}{\beta(n-k)}+O\Big(\frac{k^2}{n^2}\Big).
    \end{align*}
\end{proposition}
Then using Proposition \ref{prop:diffuapprox} we obtain the analogue of Proposition \ref{prop:ubar_conv} for the Gaussian beta-ensembles.
\begin{proposition}\label{prop:ubar_diffusion_G}
    As $n\to\infty$ we have 
    \begin{align*}
        (m_n^{-1}\bar u_{\lfloor xm_n\rfloor}, \bar u_{\lfloor xm_n\rfloor +1}-\bar u_{\lfloor xm_n\rfloor}) \Rightarrow (\psi_d(x),\psi_d'(x)),
    \end{align*}
    in law with respect to the Skorokhod topology on $\R_+$.
\end{proposition}
Repeating the argument of the proof of Lemma \ref{lem:HS_T} leads to the following result.
\begin{proposition}\label{prop:HS_T_G}
    Denote by $\ma{M}_n^{(T)}\equiv \ma{M}_n^{G,(T)} $ the upper-left $n_0\times n_0$ corner of $\ma{M}_n$. Set
    \begin{align*}
    \mathsf{K}_n^{(T)}(x,y)\equiv\mathsf{K}_n^{G,(T)}(x,y) = \sigma_n^{-1}m_n\Big[\big(\ma{M}_n^{(T)}\big)^{-1}\Big]_{\lceil x m_n\rceil, \lceil y m_n\rceil} \ind(0<x, y\le T),
\end{align*}
    then for any fixed $T>0$, there exists a coupling so that $\|\mathsf{K}_n^{(T)}-\mathsf{K}_{\Ai}^{(T)}\|_{2} \to 0$ a.s.~as $n\to\infty$.
\end{proposition}

What remains is to prove \eqref{eq:HS_tail_general} for the Gaussian beta-ensemble.
Compared to the proofs in Section \ref{sec:biga}, the arguments in the Gaussian case are easier in the following sense:
\begin{enumerate}
    \item For the martingale arguments in Section \ref{subsec:martingale}, we can work with the homogeneous time scale parameter $m_n\equiv n^{1/3}$, rather than the inhomogeneous parameter $h_k = \frac{a^{2/3}}{n-k}$ used for the Laguerre beta-ensemble.
    
    \item We can apply the martingale argument for $\bar u_k$ for $n_0
    \le k\le n/2$. By Lemma \ref{lem:martingale_fluct} it is straightforward to show that $u_k/\bar u_k$ is of constant order in the regime $n_0\le k\le n/2$. 
    
    \item The concentration bounds are more effective, i.e.~an analogue of Proposition \ref{prop:u_ratio_concentration} for the Gaussian beta-ensemble applies for $k\ge n/2$, rather than only for $k\ge n-\lfloor a\mathfrak{f}(a)\rfloor$ as in Proposition \ref{prop:u_ratio_concentration}.
\end{enumerate}
We next provide the necessary ingredients to complete the proof. Using the tail bounds for Gaussian and chi random variables, the following event holds with high probability.
\begin{lemma}
Define 
\begin{align*}
    \mathcal{A}_{n,1}^G &= \{|g_k|\le \log n,\quad 1\le k\le n\},\\
    \mathcal{A}_{n,2}^G &= \{ |Y_k-\sqrt{n-k}|\le  2 \beta^{-1/2}\sqrt{\log(\beta(n-k))},\quad 1\le k\le n-\lfloor\sqrt{n}\rfloor\},\\
    \mathcal{A}_{n,3}^G &= \{\sqrt{\tfrac{{n-k}}{\log n}}\le Y_k\le \sqrt{(n-k)\log n},\quad n-\lfloor\sqrt{n}\rfloor\le k<n\}.
\end{align*}
    Let $\mathcal{A}_n^G=\mathcal{A}_{n,1}^G\cap \mathcal{A}_{n,2}^G\cap \mathcal{A}_{n,3}^G$, then  $\pr(\mathcal{A}_n^G)\ge 1- c_\beta(\log n)^{-\beta/2}$, where $c_\beta>0$ is independent of $n$.
\end{lemma}
For $n_0\le i\le n/2$, we study the discrete Riccati transform $\bar p_i\equiv \bar p_i^G := m_n (\frac{\bar u_{i}}{\bar u_{i-1}}-1)$.  By \eqref{eq:ubar_diff_G} we have     
\begin{align}
        \bar p_{i+1}-\bar p_i
        &=m_n\ev[Z_{2,i}]+ \ev[Z_{1,i}-Z_{2,i}]\frac{\bar p_i}{1+m_n^{-1}\bar p_i} - m_n^{-1}\frac{\bar p_i^2}{1+m_n^{-1}\bar p_i} \label{eq:pbar_drift_G}\\
        &\quad + m_n\tl Z_{2,i}+(\tl Z_{1,i}-\tl Z_{2,i})\frac{\bar p_i}{1+m_n^{-1}\bar p_i}\label{eq:pbar_fluc_G}.
    \end{align}
When $\bar p_k$ deviates from $\sqrt{k/m_n}$ for some $k\ge n_0$, using the martingale bounds in Lemmas \ref{lem:martingale_fluct} and \ref{lem:martingale_fluc2} we can show that the fluctuation of the terms in \eqref{eq:pbar_fluc_G} can be bounded by  the initial value of $\bar p_k$ plus the drift terms in \eqref{eq:pbar_drift_G}. Repeating the arguments in the proofs of Proposition \ref{prop:pbar_square_root} and \ref{prop:one-step-drop} leads to the following result.

\begin{proposition}\label{prop:pbar_square_root_G}
For any fixed  $0<\delta<\frac12$,  we have 
    \begin{align}\label{eq:pbar_square_root_G}
      \lim_{T\to\infty}\lim_{n\to\infty}\pr\left(\Big|\bar p_n(k)-\sqrt{\tfrac{k}{m_n}}\Big|\le \delta\sqrt{\tfrac{k}{m_n}}\text{\, for all $n_0\le k\le n/2$}\right)=1.
    \end{align}
\end{proposition}

For $n/2\le i< n$, we can simply control the ratio $\frac{u_{i+1}}{u_i}$ by the concentration bounds. The following result can be viewed as the analogue of Proposition \ref{prop:u_ratio_concentration}.

\begin{proposition}\label{prop:u_ratio_concentration_G}
For  $k\ge n_1$ and for $n$ large enough we have 
\begin{align*}
\Big\{\frac{u_k}{u_{k-1}}\ge 1\Big\}\cap\mathcal{A}_n^G \subset \left\{\frac{u_{k+1}}{u_k}>1+\min\Big\{\tfrac12\sqrt{\tfrac{k}{n-k}},n^{1/4}\Big\}\right\}\cap\mathcal{A}_n^G.
\end{align*}
\end{proposition}
\begin{proof}
    Under the event $\{u_k\ge u_{k+1}\}\cap\mathcal{A}_n^G$, by \eqref{eq:u_recursion_G} and the concentration bounds of $g_k,Y_k$ we have 
    \begin{align*}
        \frac{u_{k+1}}{u_k} - 1 \ge \frac{2\sqrt{n}+g_k}{Y_k}-\frac{Y_{k-1}}{Y_k}-1\ge \begin{cases}
            \frac{1}{2}\sqrt{\frac{k}{n-k}}\quad &\mbox{if $k\le n-\sqrt{n}$,}\\
            n^{1/4}  &\mbox{if $n-\sqrt{n}\le k< n-\sqrt{n}$,}
        \end{cases}
    \end{align*}
    proving the statement.
\end{proof}
We remark that an exponential decay of $u_k/u_{k-1}$ for $k\ge n_1$ is sufficient for our estimate on $\|\mathsf{K}_n^G-\mathsf{K}_n^{G,(T)}\|_{2}$, and the bounds in Proposition \ref{prop:u_ratio_concentration_G} are much stronger than what we need when $n-k=o(n)$. 
Following the estimates in Section \ref{subsec:HS_tail}, we define $F_1\equiv F_1^G, F_2\equiv F_2^G$ as
\begin{align}
    F_1(i):= v_iu_i =\left(\sum_{\ell=i}^{n-1}\frac{u_{i}^2}{u_{\ell}u_{\ell+1} Y_\ell}+ \frac{u_i^2}{u_n^2}\frac{1}{(2\sqrt{n}+g_n)-Y_{n-1}\frac{u_{n-1}}{u_n}}\right),\quad F_2(i):=\sum_{j\le i}\frac{u_j^2}{u_i^2}.
\label{def:F1F2_G}
\end{align}
By similar arguments as in \eqref{eq:v-v^T} and \eqref{eq:kernel_upper_bound}, we have 
\begin{align}\label{eq:kernel_upper_bound_G}
    \|\mathsf{K}_n^G-\mathsf{K}_n^{G,(T)}\|^2_{2}\lesssim n^{-1/3}F_1(n_0+1)^2F_2(n_0)^2 + n^{-1/3}\sum_{i\ge n_0}F_1(i)^2F_2(i).
\end{align}
Using the bounds in Propositions \ref{prop:pbar_square_root_G} and \ref{prop:u_ratio_concentration_G}, one can prove the following bounds for $F_1, F_2$. 

\begin{proposition}\label{prop:F_bound_G}
    Under the event where the bounds in Propositions \ref{prop:pbar_square_root_G} and \ref{prop:u_ratio_concentration_G} hold true, we have
    \begin{align}
        F_1(i)\lesssim \begin{cases}
        \frac{1}{\sqrt{i}}+\frac{\log n}{\sqrt{n}}\quad &\mbox{if $n_0\le i\le n/2$}\\
        \frac{\log n}{\sqrt{n-i}}\quad &\mbox{if $i>n/2$}
            \end{cases},\qquad 
        F_2(i)\lesssim \begin{cases}
       \sqrt{\frac{n}{i}}\log T \quad &\mbox{if $n_0\le i\le n/2$}\\
        \sqrt{\frac{n}{i}}\quad &\mbox{if $i>n/2$}
        \end{cases}.
    \end{align}
\end{proposition}

Collecting these bounds together, we obtain the analogue  of Lemma \ref{lem:HS_tail} for the Gaussian beta-ensembles.
\begin{proposition}\label{prop:HS_tail_G}
    \[
    \lim_{T\to\infty}\lim_{n\to\infty}\|\mathsf{K}_n^{G}-\mathsf{K}_n^{G,(T)}\|_{2}\to 0 \quad \text{in probability.}
    \]
\end{proposition}
\begin{proof}
We work on the event where the bounds in Propositions \ref{prop:pbar_square_root_G} and \ref{prop:u_ratio_concentration_G} hold true.
By the bounds in Proposition \ref{prop:F_bound_G}, 
and \eqref{eq:kernel_upper_bound_G},  we have
    \begin{align*}
    \|\mathsf{K}_n^G-\mathsf{K}_n^{G,(T)}\|^2_{2}& \lesssim T^{-1/2}\log T+n^{-1/3} (\log n)^3,
\end{align*}
which goes to $0$ after sending $n\to\infty$ and then $T\to\infty$. Since for any $\eps>0$, the bounds in  Propositions \ref{prop:pbar_square_root_G} and \ref{prop:u_ratio_concentration_G} happen with probability at least $1-\eps$ for all sufficiently large $n$ and $T$, the statement follows.
\end{proof}
\begin{proof}[Proof of Theorem \ref{thm:Gaussian}]
    The statement follows from Lemma \ref{lem:Airy_tail}, Propositions \ref{prop:HS_T_G}, \ref{prop:HS_tail_G} and an application of the triangle inequality. 
\end{proof}

\subsection{Edge limits of the Laguerre beta-ensemble}
In this section, we outline the proof of the soft edge limits of the Laguerre beta-ensembles near the upper edge, and the lower edge when $\liminf_{n\to\infty} a_n/n\in(0,\infty]$.
Recall the Dumitriu-Edelman tridiagonal matrix model $\ma{L}_n\ma{L}_n^\top$ with $\ma{L}_n\equiv \ma{L}_{n,\beta,a_n}$ defined in \eqref{eq:DE-bidiagonal}, and  the dimension parameter $\kappa = n+a_n$. The diagonal and off-diagonal entries of $\ma{L}_n$ are independent such that 
\begin{align}\label{eq:XY_general}
\{X_i\sim \beta^{-1/2}\chi_{\beta(\kappa-i+1)},1\le i\le n\},\qquad \{Y_i\sim\beta^{-1/2}\chi_{\beta(n-i)},1\le i\le n-1\}
\end{align}
By the Marchenko-Pasture law, the lower and upper edge centering parameters are set as  $\mu_{n,low}=(\sqrt{\kappa}-\sqrt{n})^2$, and $\mu_{n,up} = (\sqrt{\kappa}+\sqrt{n})^2$, respectively. 

Following the notations in \cite{RRV}, we set the (inverse of the) time scaling parameter as 
\begin{align*}
  m_n = \begin{cases}
        m_{n,up}=\left(\frac{\sqrt{n\kappa }}{\sqrt{n}+\sqrt{\kappa}}\right)^{2/3} &\mbox{upper edge,}\\
         m_{n,low}=\left(\frac{\sqrt{n\kappa }}{\sqrt{\kappa}-\sqrt{n}}\right)^{2/3} &\mbox{lower edge,}
    \end{cases}
\end{align*}
and we set the spatial scaling parameter as 
\begin{align*}
    \sigma_n = \begin{cases}
        \sigma_{n,up} = \frac{(m_{n,up})^2}{\sqrt{n\kappa}} = \frac{(n\kappa)^{1/6}}{(\sqrt{\kappa}+\sqrt{n})^{4/3}},&\mbox{upper edge,}\\
         \sigma_{n,low}= \frac{(m_{n,low})^2}{\sqrt{n\kappa}} = \frac{(n\kappa)^{1/6}}{(\sqrt{\kappa}-\sqrt{n})^{4/3}},&\mbox{lower edge.}
    \end{cases}
\end{align*}
Note that in the definition of $\sigma_n$, the term $m_n^2$ comes from Laplacian scaling, and the  term  $\sqrt{n\kappa}$ comes from the asymptotics of $Y_1X_2$. When $\liminf_{n\to\infty} \kappa/n\in(1,\infty]$, we always have $m_n=O(n^{1/3})$ for both the lower and upper edges. 

We focus on the lower edge and work with its corresponding parameters: $\mu_n=\mu_{n,low}, m_n=m_{n,low},\sigma_n=\sigma_{n,low}$, and we will comment on the necessary modifications for the upper edge at the end of the section. We may assume that $\kappa/n\to c \in (1,\infty]$. Recall the kernel $\mathsf{K}_n^L(x,y)\equiv \mathsf{K}_{n,\beta,a_n}^{L,low}(x,y)$ defined in \eqref{def:Kn_kernel_general}. 
\begin{proof}[Outline of the proof of Theorem \ref{thm:Laguerre}]
Consider the setup as in Section \ref{subsec:diffusion}, and define the sequence $u_k$ as in \eqref{eq:u_recursion} with independent entries $X_j,Y_j$ according to \eqref{eq:XY_general}.
Set $
\bar u_k = u_k\prod_{j=1}^{k-1}\frac{X_{j+1}}{Y_j}\gamma_j$ for  $1\le k\le n$ and $\gamma_j = \sqrt{\frac{n-j}{\kappa-j}}$.
Define also the sequences $Z_{1,k}$ and $Z_{2,k}$ as in \eqref{def:Z1Z2} under the new choices of $X_k$ and $Y_k$, then the difference equation \eqref{eq:ubar_diff_recursion} for $\bar u_{k}-\bar u_{k-1}$ still holds, and the increments $Z_{1,k},Z_{2,k}$ are independent from $\bar u_j,j<k$.
By an exact computation of the moments of  $Z_{1,k},Z_{2,k}$, one obtains the convergence result \eqref{eq:conv_to_Airy_d} for the Laguerre beta-ensembles.
Note that due to the assumption that $\kappa/n\to c\in (1,\infty]$, the required moment computations are much easier than in Proposition \ref{prop:moments}.

With the self-evident notations, repeating the proof of Lemma \ref{lem:HS_T}  we obtain that  for any fixed $T>0$ under a suitable coupling   \begin{align}\label{eq:HS_T_Lag_general}
\|\mathsf{K}_{n}^{L,(T)}-\mathsf{K}_{\Ai}^{(T)}\|_2\to 0\text{\quad a.s.~$n\to\infty$.}
\end{align}
For $n_0\le i\le n/2$, one can repeat the analysis of the discrete Riccati process $\bar p_i^L:=m_n(\bar u_i/\bar u_{i-1}-1)$. Since the increments $Z_{1,k},Z_{2,k}$ have the same asymptotics as in the Gaussian case (c.f.~Proposition \ref{prop:moments_G}), one can show that with high probability for any fixed $0<\delta<\frac12$, the bounds $|\bar p_i^L-\sqrt{i/m_n}|\le \delta\sqrt{i/m_n}, n_0\le i\le n/2$ hold.  For $i\ge n/2$, under the concentration bounds of the chi-random variables and the assumption that $\frac{u_k}{u_{k-1}}\ge 1$, one can check the main term of $\frac{u_{k+1}}{u_k}-1$ is 
    \begin{align*}
    \frac{\kappa-j+n-j-(\sqrt{\kappa}-\sqrt{n})^2}{\sqrt{\kappa-j}\sqrt{n-j}}-2
    \gtrsim \frac{(\kappa-n)^2j}{\sqrt{\kappa-j}\sqrt{n-j}\kappa^{3/2}n^{1/2}}.
    \end{align*}
By the assumption $\kappa/n\to c\in (1,\infty]$ we have $
\frac{(\kappa-n)^2}{\kappa^{3/2}\sqrt{\kappa -j}} = O(1)$. This leads to the analogue of Proposition \ref{prop:u_ratio_concentration_G}.
For the Hilbert--Schmidt norm estimate, the argument is the same as in the Gaussian case, except that the concentration bounds for the entries now give
\[
\frac{1}{X_{\ell+1}Y_\ell} \lesssim \begin{cases}
    \frac{1}{\sqrt{n\kappa}}\quad &\mbox{if $\ell \le n/2$}\\
    \frac{\log n}{\sqrt{n-\ell}\sqrt{\kappa}}\quad &\mbox{if $\ell > n/2$}
\end{cases},\qquad  X_n^2-\mu_n - X_{n-1}Y_{n-1}\gtrsim \sqrt{n\kappa}.
\]
This produces an extra factor  of $\kappa^{-1/2}$ in the estimate of the function $F_1$, which cancels out with the factor $\sqrt{\kappa}$ term in the definition of $\sigma_n$. Here we recall that $m_n=O(n^{1/3})$ and $\sigma_n = m_n^2/\sqrt{n\kappa}$. Repeating the computations in the proofs of Lemma \ref{lem:HS_tail} and Proposition \ref{prop:HS_tail_G} gives that $\|\mathsf{K}_n^L-\mathsf{K}_n^{L,(T)}\|_2 \to 0$ in probability as $n\to\infty$ and then $T\to\infty$. 
    The proof follows by Lemma \ref{lem:Airy_tail}, convergence \eqref{eq:HS_T_Lag_general}, and the triangle inequality.
\end{proof}

We end this section with a remark on the upper edge of the Laguerre beta-ensemble.
\begin{remark}\label{rmk:Lag_upper}
    Consider the bidiagonal matrix $\tl{\ma{L}}_n\equiv \tl {\ma{L}}_{n,\beta,a_n}$, whose diagonal entries are $X_1,\dots, X_n$, and upper diagonal entries are $Y_1,\dots, Y_{n-1}$. Then the eigenvalues of $\tl {\ma{L}}_n\tl {\ma{L}}_n^\top$ still have joint density \eqref{eq:Lag_PDF}. Set $\tl {\ma{M}}_n:= (\mu_{n,up} - \tl {\ma{L}}_n\tl {\ma{L}}_n^\top)$, which is a symmetric tridiagonal matrix of form \eqref{eq:tridiagonal}. One can then repeat the analysis outlined above, and the resulting asymptotics analysis is the same as the corresponding ones near the lower edge when $\kappa/n\to c\in (1,\infty]$. Note that for the upper edge we do not need to assume anything about the growth of $\kappa$ to obtain the limit.
\end{remark}

\section{Appendix} \label{sec:Appendix}

\subsection{Moment computations}

We present the proof of Proposition \ref{prop:moments}. Recall that if $W\sim \chi_p$ then 
  \begin{align}\label{eq:chi_moments}
        \ev[W^k] = 2^{k/2}\frac{\Gamma(\frac{k+p}{2})}{\Gamma(\frac{p}{2})},\qquad \text{for } k>-p.
    \end{align}

\begin{proof}[Proof of Proposition \ref{prop:moments}]
In all of the computations in this proof the $O(\cdot)$ terms will have constants that only depend on $\beta$. 
From the definition of $\gamma_k$ and the bound $n-k\ge \lfloor a \mathfrak{f}(a)\rfloor $ we get 
\begin{align}
\gamma_k &= 1 - \frac{a}{n-k} + O\left(\frac{a^2}{(n-k)^2}\right), \qquad 
1-\frac{\gamma_{k-1}}{\gamma_k} = \frac{a}{(n-k)^2} + O\left(\frac{a^2}{(n-k)^3}\right).
\end{align}
From \eqref{eq:chi_moments} we get the following:
\begin{align*}
    &\ev\left[\frac{X_k^2}{Y_k^2}\right]=\frac{n+2 a-k+1}{n-k-2/\beta}=\gamma_k^{-2}\left(1+O\left(\frac{1}{n-k}\right)\right), \quad 
    \Var\left(\frac{X_k^2}{Y_k^2}\right)=\frac{4}{\beta(n-k)}+O\left(\frac{a}{(n-k)^2}\right)\\[5pt]
    &\ev\left[\frac{1}{Y_k^2}\right]=\frac{1}{n-k-2/\beta}=\frac{1}{n-k}+O\left(\frac{1}{(n-k)^2}\right), \quad
    \Var\left(\frac{1}{Y_k^2}\right)=\frac{2}{\beta (n-k)^3}+O\left(\frac{1}{(n-k)^4}\right)\\[5pt]
    &\operatorname{Cov}\left(\frac{X_k^2}{Y_k^2},\frac{1}{Y_k^2}\right)=\frac{2 (n-k+2 a+1)}{\beta  (n-k-4/\beta) (n-k-2/\beta)^2}=\frac{2}{\beta (n-k)^2}+O\left(\frac{a}{(n-k)^3}\right)\\[5pt]
    &\ev \left(\frac{X_k^2}{Y_k^2}-\ev\left[\frac{X_k^2}{Y_k^2}\right]\right)^4=O\left(\frac{1}{(n-k)^2}\right), \qquad 
     \ev \left(\frac{1}{Y_k^2}-\ev\left[\frac{1}{Y_k^2}\right]\right)^4=O\left(\frac{1}{(n-k)^6}\right)
\end{align*}
    We also record the following bound on $\mu_n$:
 \begin{align}\label{eq:mu_bound}
\mu_n=\mu=\frac{(2a)^2}{n(\sqrt{1+2a/n}+1)^2}=\frac{a^2}{n}\left(1+O\left(\frac{a}{n}\right)\right).
\end{align}   
We start with the asymptotics for the expectation of $Z_{2,k}$ which requires some subtle estimates.
We first compute
\begin{align*}
    \ev\left[(\gamma_k^{-1}-1)^2- \frac{\mu}{Y^2_k}\right]&=(\sqrt{\tfrac{n-k+2a}{n-k}}-1)^2- \frac{(\sqrt{n+2a}-\sqrt{n})^2}{n-k-2/\beta}\\
    &=(\sqrt{1+\tfrac{2a}{n-k}}-1)^2- \frac{n}{n-k}(\sqrt{1+\tfrac{2a}{n}}-1)^2+O\left(\tfrac{a^2}{n(n-k)^2}\right)\\
    &=
\frac{a^2 k}{n(n-k)^2}+O\left(\frac{a^2}{n(n-k)^2}+\frac{a^3 k}{n(n-k)^3}\right).
\end{align*}
The last step is a consequence of the following inequality which holds uniformly for $0<y\le x$:
\[
\left|\left(\sqrt{1+2 x}-1\right)^2-\frac{x}{y} \left(\sqrt{1+2 y}-1\right)^2-x (x-y)\right|\le (x+y)x(x-y).
\]
We can now compute
\begin{align}
    \ev[Z_{2,k}]&=\gamma_k \ev\left[(\gamma_k^{-1}-1)^2- \frac{r}{Y^2_k}\right]+1-\gamma_{k-1}\gamma_k^{-1}+\gamma_k(1-\gamma_{k-1})\ev\left[\frac{X_k^2}{Y_k^2}-\gamma_k^{-2}\right]\\
    &=\frac{a^2 k}{n(n-k)^2}+O\left(\frac{a^3 k}{n(n-k)^3}+\frac{a}{(n-k)^2}\right).
\end{align}
For the variance we have
\begin{align*}
 \gamma_k^{-2} \Var(Z_{2,k})&=(1-\gamma_{k-1})^2\Var\left(\frac{X_k^2}{Y_k^2}\right)+ \mu^2    \Var\left(\frac{1}{Y_k^2}\right)-2(1-\gamma_k) \mu \Cov\left(\frac{X_k^2}{Y_k^2},\frac{1}{Y_k^2}\right)\\
 &=\frac{4a^2}{\beta(n-k)^3}+O\left(\frac{a^3}{(n-k)^3 n}\right).
\end{align*}
For the expectation of $Z_{1,k}$, first note that $Z_{1,k}-Z_{2,k}=\frac{X_k^2}{Y_k^2}\gamma_k \gamma_{k-1}-1$, hence
\begin{align*}
    \ev[Z_{1,k}-Z_{2,k}]=\gamma_{k-1}\gamma_k^{-1}\left(1+O\left(\frac{1}{n-k}\right)\right)-1=O\left(\frac{1}{n-k}\right).
\end{align*}
This implies the stated asymptotics for $\ev[Z_{1,k}]$. The asymptotics for $\Var(Z_{1,k})$ and $\Cov(Z_{1,k}, Z_{2,k})$ can be computed  similarly as it was done for $\Var(Z_{2,k})$. Finally, the bounds for the fourth moment can be obtained using the preceding  moment estimates together with the inequality 
\begin{align*}
    \ev[Z_{i,k}^4]&\le 8 |\ev[Z_{i,k}]|^4+8 \ev|Z_{i,k}-\ev[Z_{i,k}]|^4.\qedhere
\end{align*}
\end{proof}

\subsection{Why the RRV method fails in the case when $1\ll a_n\ll n$}

In \cite{RRV} the authors prove Theorem \ref{thm:uppersoft}, and their methods can be extended to prove Theorem \ref{thm:lowersoft} as well. 
In this section, we briefly explain why their method does not cover lower edge limit in the case when $1\ll a_n\ll n$. 

Let $\ma{L}_n=\ma{L}_{n,\beta,a_n}$ be the bidiagonal matrix defined in \eqref{eq:DE-bidiagonal}, and set $\kappa:=n+a_n$. The top left corner of $\ma{L}_{n}\ma{L}_{n}^\top$ has roughly $n+\kappa$ on the diagonal, and $-\sqrt{n\kappa}$ on the off-diagonal. To approximate a discrete Laplacian, one needs to first subtract $(\sqrt{\kappa}-\sqrt{n})^2$ and then divide by $\sqrt{n\kappa}$ near the lower edge of the spectrum.
This leads to the conjecture
\begin{align*}
       \frac{m_n^2}{\sqrt{n\kappa}} \big(\Lambda_{n,\beta,a_n}-(\sqrt{\kappa}-\sqrt{n})^2\big) \Rightarrow \Airyb,\qquad m_n=\left(\frac{\sqrt{n\kappa}}{\sqrt{\kappa}-\sqrt{n}}\right)^{2/3}.
\end{align*}
In the case when $a_n$ diverges sublinearly, we have $ m_n\sim a_n^{-2/3}n$. Note that when $\kappa/n\to c>1$ (corresponding the case discussed in Theorem \ref{thm:lowersoft} and Theorem \ref{thm:Laguerre}), we have $m_n=O(n^{1/3})$. 

The idea of proving the convergence to the $\Airyop$ operator in \cite{RRV} is that, after subtracting the discrete Laplacian on the time scale $m_n^{-1}$, the running sum of the process of remaining terms (on the diagonal and off-diagonals) converge to the integrated potential of the $\Airyop$ operator, i.e.~$\frac{x^2}{2}+ \frac{2}{\sqrt{\beta}}B_x$. 

We will explain why this is not true in the regime when $1\ll a_n\ll n$. Let
\begin{align*}
    \Delta y_{n,1,j} &= \frac{m_n}{\sqrt{n\kappa}}(\beta^{-1}(\chi_{\beta(n-j)}^2+\chi_{\beta(\kappa-j+1)}^2)-(n+\kappa)),\\
    \Delta y_{n,2,j} &= 2\frac{m_n}{\sqrt{n\kappa}}(\sqrt{n\kappa}-\beta^{-1}\chi_{\beta(n-j)}\chi_{\beta(\kappa-j)})
\end{align*}
be the (shifted) increments of the diagonals and off-diagonals after we remove the discrete Laplacian. 
In the regime when $\kappa/n\to c>1$, using Proposition \ref{prop:diffuapprox} one can show that the running sums $y_{n,i}(x):=\sum_{j=1}^{xm_n} \Delta y_{n,i,j},i=1,2$ converge to the appropriate limits so that $y_{n,1}(x)+y_{n,2}(x)\Rightarrow \frac{x^2}{2}+\frac{2}{\sqrt{\beta}}B_x$. 
However, when $1\ll a_n\ll n$ the required moment asymptotics of Proposition \ref{prop:diffuapprox} do not hold. For example
we have 
\begin{align*}
    m_n \ev[\Delta y_{n,1,j}] = \frac{m_n^2}{\sqrt{n\kappa}} (-2j+1) \sim -j n a_n^{-4/3},
\end{align*}
which does not converge when $j=\lfloor x m_n\rfloor$ for a fixed $x\ge 0$. 
A similar calculation shows that $m_n\ev[\Delta y_{n,2,j}]$ and $m_n\ev[\Delta y_{n,1,j}+\Delta y_{n,2,j}]$ do not converge either.

For the second moment, we have 
\begin{align*}
    m_n\ev[\Delta y_{n,1,j}^2] &= \frac{m_n^3}{n\kappa} \ev\left[\big(\beta^{-1}(\chi_{\beta(n-j)}-\chi_{\beta(\kappa-j)})^2 - (\sqrt{\kappa}-\sqrt{n})^2\big)^2\right]\sim\frac{m_n^3}{n\kappa} (\frac{3}{\beta^2}+\frac{a^2}{\beta n}) 
\end{align*}
which might not converge if $a$ grows sub-linearly.

In summary, these moment computations show that the noise in the process $(y_{n,1}, y_{n,2})$ is too large for the process to have a meaningful limit in the appropriate scaling when $1\ll a_n\ll n$. Instead, when $(\log\log n)^3\ll a_n\ll n$, we work with the inverse of the tridiagonal matrix model, where the noise is smoothed out. This allows us to establish the convergence of the rescaled eigenvalues.

\subsection{Estimates in the single step coupling}\label{app:couple}

\begin{proof}[Proof of Proposition \ref{prop:tv}]
Let $Z=|B^{(\cdr)}(\frac{1}{p+q})|$ and $R=|\xi|$.
        Let $r$ and $\cdr$ be defined with  $p=r+\cdr/2$ and $q=r-\cdr/2$. For $x\ge 0$, the density functions of $R$ and $Z$ are given by 
    \begin{align*}
        f_R(x) &
        = \frac{\Gamma(2r)}{4^{r-1}\Gamma(r+\tfrac{\cdr}{2})\Gamma(t-\tfrac{\cdr}{2})}\operatorname{sech}^{2r}(x)\cosh(\cdr x):=A(r,\cdr)\sech^{2r}(x)\cosh(\cdr x),\\
        f_Z(x)&=  2\sqrt{\frac{r}{\pi}}e^{-\frac{\cdr^2}{4r}}e^{-rx^2}\cosh(\cdr x):=B(r,\cdr)e^{-rx^2}\cosh(\cdr x).
    \end{align*}
    As $x\to\infty$, the function $f_R$ decays like $e^{-2rx}$ while $f_Z$ has a Gaussian tail. Hence $f_R(x)<f_Z(x)$ for all $x$ large enough, and it suffices to prove that $f_R$ and $f_Z$ have only one intersection point on $\R_+$. Note that $f_R(x) = f_Z(x)$ iff
    \begin{align*}
       A(r,\cdr)\sech^{2r}(x) = B(r,\cdr)e^{-rx^2}.
    \end{align*}
    Take logarithms and divide by $r$, we have $f_R=f_Z$ at $x=x_0$ iff
    \begin{align*}
        \Phi(x) = 2\log \sech(x)+x^2 +\frac1r(\log A(r,\cdr)-\log B(r,\cdr))
    \end{align*}
    equals $0$ at $x=x_0$. Differentiate $\Phi$ we get
    \begin{align*}
        \Phi'(x) = 2x-2\tanh(x) >0.
    \end{align*}
    This proves that $\Phi(\cdot)$ has at most one zero on $\R_+$. On the other hand since both $f_R$ and $f_Z$ are density functions on $\R_+$ and $f_R>f_Z$ when $x$ is large enough, they must have at least one intersection point. This proves that $\pr(R>u)\ge \pr(Z>u)$ for any $u\ge0$. 
    
    The above argument also implies that 
    \begin{align}\label{eq:normalizing_const}
        f_R(0)=A(r,\cdr)<B(r,\cdr) = f_Z(0).
    \end{align}
    To compute the total variation distance, we introduce
    \begin{align*}
        f_1(x) = B(r,\cdr)\sech^{2r}(x)\cosh({\cdr}x),\qquad 
        f_2(x) = A(r,\cdr)e^{-rx^2}\cosh(\cdr x).
    \end{align*}
    Since $e^{-rx^2}\ge \sech^{2r}(x)$, we have $f_1\ge \max\{f_R,f_Z\}\ge \min\{f_R,f_Z\}\ge f_2$. Therefore,
    \begin{align}\label{eq:TV_bound}
        d_{TV}(R,Z) = \int_0^\infty |f_R(x)-f_Z(x)| dx
        &\le \int_0^\infty (f_1-f_2)(x)dx
        =\frac{B(r,\cdr)}{A(r,\cdr)}-\frac{A(r,\cdr)}{B(r,\cdr)}.
    \end{align}

    Since 
    \begin{align*}
    \log\Gamma(x) =(x-\tfrac12)\log x -x +\tfrac12\log (2\pi)+2\int_0^\infty \frac{\arctan(t/x)}{e^{2\pi t}-1}dt,
    \end{align*}
    we have 
    \begin{align*}
       (x-\tfrac12)\log x -x +\tfrac12\log (2\pi) \le \log\Gamma(x) \le (x-\tfrac12)\log x -x +\tfrac12\log (2\pi)+ \tfrac{1}{10 x}.
    \end{align*}
    Therefore, for $r\ge 10 \cdr^2+10$ we get
    \begin{align*}
        1\le \frac{B(r,\cdr)}{A(r,\cdr)}=2^{2r-1}e^{-\frac{\cdr^2}{4r}}\sqrt{\frac{r}{\pi}} \cdot \frac{\Gamma(r+\tfrac{\cdr}{2})\Gamma(t-\tfrac{\cdr}{2})}{\Gamma(2r)}\le 
        \exp\big((5r)^{-1}+O(r^{-2})\big).
    \end{align*}
    Plugging into \eqref{eq:TV_bound} yields that $
    d_{TV}(R,Z) \le \frac{1}{2r}$.
\end{proof}

\begin{proof}[Proof of \textit{(c)} of Proposition \ref{prop:single_coupling}]
    Note that by modifying $c_1$ we may assume that $u\ge 10$. We follow the same strategy as in the  proof of Proposition 10 of \cite{BVBV_19}. We first bound
    \begin{align}\label{eq:stopping_time_tail}
        \pr\Big(\sigma>\frac{u}{2r}\Big) \le \pr\Big(|B^{(\cdr)}(\sigma)|\ge \frac{u^{1/3}}{r^{1/2}}\Big)+\pr\Big(\sigma>\frac{u}{2r}, |B^{(\cdr)}(\sigma)|\le \frac{u^{1/3}}{r^{1/2}}\Big).
    \end{align}
Since $|B^{(\cdr)}(\sigma)| \ed |\xi|$, by an exact calculation we have 
    \begin{align*}
        \pr\Big(|B^{(\cdr)}(\sigma)|\ge \frac{u^{1/3}}{r^{1/2}}\Big)
        &=\int_{u^{1/3}r^{-1/2}}^\infty \frac{\Gamma(2r)}{4^{r-1}\Gamma(r+\tfrac{c}{2})\Gamma(t-\tfrac{c}{2})}\sech^{2r}(x)\cosh(\cdr x) dx\\
        &\le  \int_{u^{1/3}r^{-1/2}}^\infty  2\sqrt{\frac{r}{\pi}}e^{-\frac{c^2}{4r}} \sech^{2r}(x)\cosh(\cdr x) dx.
    \end{align*}
  where we used the inequality \eqref{eq:normalizing_const} in the second step. In the case when $u^{1/3}r^{-1/2}\ge 4\log 2$, by the bounds $\sech(x)\le 2e^{-x}$ and $\cosh(\cdr x)\le e^{|\cdr|x}$, we get 
    \begin{align}\label{eq:stop_bound1}
    \pr\Big(|B^{(\cdr)}(\sigma)|\ge \frac{u^{1/3}}{r^{1/2}}\Big) &\le  \int_{u^{1/3}r^{-1/2}}^\infty  2^{2r+1}\sqrt{\frac{r}{\pi}}e^{-\frac{\cdr^2}{4r}}  e^{(-2r+|\cdr|)x}dx\le\frac{4}{\sqrt{\pi r}}e^{-\frac{\cdr^2}{4r}} e^{-u^{1/3}r^{1/2}}.
    \end{align}
    If $u^{1/3}r^{-1/2}\le 4\log 2$, we first write 
    \begin{align}\label{eq:stop_split}
        \pr\Big(|B^{(\cdr)}(\sigma)|\ge \frac{u^{1/3}}{r^{1/2}}\Big)  \le  \left(\int_{u^{1/3}r^{-1/2}}^{4\log 2}+ \int_{4\log 2}^\infty \right) 2\sqrt{\frac{r}{\pi}}e^{-\frac{\cdr^2}{4r}} \sech^{2r}(x)\cosh(\cdr x)dx.
    \end{align}
    The second integral on the right side of \eqref{eq:stop_split} can be upper bounded by $e^{-2r\log 2}\le e^{-u^{1/3}r^{1/2}/2}$. For $u^{1/3}r^{-1/2}\le x\le 4\log 2$, we have $\sech(x)\le e^{-x^2/3}$, and by an approximation of Gaussian type integrals, we have 
    \begin{align}\label{eq:stop_bound2}
        \int_{\frac{u^{1/3}}{r^{1/2}}}^{4\log 2}2\sqrt{\frac{r}{\pi}}e^{-\frac{\cdr^2}{4r}} \sech^{2r}(x)\cosh(\cdr x) dx &\le 2\sqrt{\frac{r}{\pi}}e^{-\frac{\cdr^2}{4r}}   \int_{\frac{u^{1/3}}{r^{1/2}}}^{4\log 2} e^{-2rx^2/3+|\cdr|x}dx\le 2e^2 e^{-\frac12 u^{2/3}}.
    \end{align}
    Collecting the bounds \eqref{eq:stop_bound1}---\eqref{eq:stop_bound2} shows that there exists $c_1,c_2>0$ such that 
    \[
      \pr\Big(|B^{(\cdr)}(\sigma)|\ge \frac{u^{1/3}}{r^{1/2}}\Big) \le c_1e^{-c_2u^{1/3}}.
    \]
    Now we turn to estimate the second term on the right hand side of \eqref{eq:stopping_time_tail}.  We have 
    \begin{align*}
       \pr\Big(\sigma>\frac{u}{2r}, |B^{(\cdr)}(\sigma)|\le \frac{u^{1/3}}{r^{1/2}}\Big) &\le \pr\Big(\max_{1/(2r)\le v\le u/(2r)}|B^{(\cdr)}(v)|\le \frac{u^{1/3}}{r^{1/2}}\Big)\\
       &\le \pr\Big(\max_{1/(2r)\le v\le u/(2r)}|B^{(\cdr)}(v)-B^{(\cdr)}(\tfrac{1}{2r})|\le 2\frac{u^{1/3}}{r^{1/2}}\Big)\\
       &\le  \pr\Big(\max_{0\le v\le (u-1)/(2r)}|B^{(\cdr)}(v)|\le 2\frac{u^{1/3}}{r^{1/2}}\Big).
    \end{align*}
   Since $B^{(\cdr)}_t$ is absolute continuous with respect to the standard Brownian motion on finite interval, by Girsanov's theorem, we have with $w:=(u-1)/(2r)$ that 
    \begin{align*}
        \pr\Big(\max_{0\le v\le w}|B^{(\cdr)}(v)|\le 2\frac{u^{1/3}}{r^{1/2}}\Big) &= \ev\big[\ind_{\max_{0\le v\le w}|B(v)|\le 2\frac{u^{1/3}}{r^{1/2}}}\exp(-\tfrac{\cdr^2}{2}w+\cdr B(w))\big]\\
        &\le \exp\big(2\cdr u^{1/3}r^{-1/2}-\tfrac{\cdr^2}{4} ur^{-1}\big)\pr\Big(\max_{0\le v\le w}|B(v)|\le 2\frac{u^{1/3}}{r^{1/2}}\bigg)\\
        &\le 2 \exp\big(2\cdr u^{1/3}r^{-1/2}- \tfrac{\cdr^2}{4} ur^{-1}-\tfrac{\pi^2}{40} u^{1/3}\big).
    \end{align*}
    In the last step, we used the following identity (and bound) for the standard Brownian motion 
    \begin{align*}
        \pr\Big(\max_{0\le s\le t}|B(s)|\le y\Big) = \frac{4}{\pi}\sum_{k=0}^{\infty}\frac{(-1)^k}{2k+1}\exp\Big(-\frac{(2k+1)^2\pi^2t}{8y^2}\Big)\le \frac{4}{\pi} e^{-\frac{\pi^2 t}{8 y^2}},
    \end{align*}
    see e.g.~Section 7.4 of \cite{morters2010brownian}.
    Combining all these estimates, we can find $c_1,c_2>0$ such that 
    \[
     \pr\Big(\max_{0\le v\le w}|B^{(\cdr)}(v)|\le 2\frac{u^{1/3}}{r^{1/2}}\Big) \le c_1e^{-c_2u^{1/3}}.
    \]
   This shows that we have the appropriate upper bound on both terms of \eqref{eq:stopping_time_tail} which implies 
   the 
   statement $\textit{(c)}$ of Proposition \ref{prop:single_coupling}.
\end{proof}
\bibliography{rmt}

  {{  
 		\bigskip
 		\footnotesize

 		\noindent Yun Li, \textsc{Yau Mathematical Sciences Center, Tsinghua University, Beijing 100084, China.}\\
        \textit{Email:} \texttt{liyun1723@hotmail.com}
 		
 		\medskip
 		
 		\noindent Benedek Valk\'o, \textsc{Department of Mathematics, University of Wisconsin -- Madison,
 			Madison, WI 53706, USA.}
        \textit{Email:} \texttt{valko@math.wisc.edu}

         		\medskip
 		
 		\noindent Jiaming Xu, \textsc{Department of Mathematics, The Ohio State University, Columbus, OH 43210, USA.} 
     \\
        \textit{Email:} \texttt{jxu0800@gmail.com}
 }}

\end{document}